\documentclass[UTF-8,reqno]{amsart}
\usepackage{enumerate}
\usepackage{amssymb,url,color, booktabs}
\usepackage{mathrsfs}

\usepackage{color}
\usepackage[colorlinks=true]{hyperref}
\hypersetup{
    linkcolor=blue,          
    citecolor=red,        
    filecolor=blue,      
    urlcolor=cyan
}

\numberwithin{equation}{section}

\newcommand{\be}{\begin{eqnarray}}
\newcommand{\ee}{\end{eqnarray}}
\newcommand{\ce}{\begin{eqnarray*}}
\newcommand{\de}{\end{eqnarray*}}
\newtheorem{theorem}{Theorem}[section]
\newtheorem{lemma}[theorem]{Lemma}
\newtheorem{remark}[theorem]{Remark}
\newtheorem{definition}[theorem]{Definition}
\newtheorem{proposition}[theorem]{Proposition}
\newtheorem{Examples}[theorem]{Example}
\newtheorem{corollary}[theorem]{Corollary}

\def\N{{\mathbb N}}
\def\Z{{\mathbb Z}}
\def\R{{\mathbb R}}

\def\mbP{{\mathbb P}}
\def\mbE{{\mathbb E}}
\def\mbF{{\mathbb F}}

\def\1{{\mathbbm 1}}
\def\mcA{{\mathcal A}}
\def\mcB{{\mathcal B}}
\def\mcD{{\mathcal D}}
\def\mcF{{\mathcal F}}
\def\mcH{{\mathcal H}}
\def\mcL{{\mathcal L}}
\def\mcM{{\mathcal M}}
\def\mcO{{\mathcal O}}
\def\mcP{{\mathcal P}}

\def\mcX{{\mathcal X}}

\def\mfB{{\mathfrak B}}
\def\mfJ{{\mathfrak J}}
\def\mfI{{\mathfrak I}}
\def\mfL{{\mathfrak L}}

\def\msA{{\mathscr A}}
\def\msF{{\mathscr F}}
\def\msI{{\mathscr I}}

\def\msL{{\mathscr L}}
\def\msP{{\mathscr P}}
\def\var{\varepsilon}
\def\veps{\varepsilon}
\def\d{{\rm d}}
\def\e{{\rm e}}

\allowdisplaybreaks
\numberwithin{equation}{section}

\title[Averaging principle]
{Averaging principle and normal deviation for multi-scale SDEs with polynomial nonlinearity}

\author{Mengyu Cheng}
\address{M. Cheng: School of Mathematics and Statistics,
Beijing Institute of Technology, Bejing 100081, P. R. China}
\email{mengyu.cheng@hotmail.com; mcheng@bit.edu.cn}

\author{Zhenxin Liu}
\address{Z. Liu:
School of Mathematical Sciences,
Dalian University of Technology, Dalian 116024, P. R. China}
\email{zxliu@dlut.edu.cn}

\author{Michael R\"ockner}
\address{M. R\"ockner: Fakultat f\"ur Mathematik,
Universit\"at Bielefeld, D-33501 Bielefeld, Germany}
\email{roeckner@math.uni-bielefeld.de}



\begin{document}

\begin{abstract}
We investigate three types of averaging principles and the normal deviation for
multi-scale stochastic differential equations (in short, SDEs) with polynomial nonlinearity.
More specifically, we first demonstrate the strong convergence of the solution of SDEs,
which involves highly oscillating components and fast processes, to that of the averaged equation.
Then we investigate the small fluctuations of the system around its average,
and show that the normalized difference weakly converges to an Ornstein-Uhlenbeck
type process, which can be viewed as a functional central limit theorem.
Additionally, we show that the attractor of the original system tends to that of the averaged
equation in probability measure space as the time scale $\veps$ goes to zero.
Finally, we establish the second Bogolyubov theorem; that is to say, we prove that
there exists a quasi-periodic solution in a neighborhood of the stationary solution of
the averaged equation when the $\veps$ is small.
\end{abstract}

\maketitle

\tableofcontents

\section{Introduction}

Consider the following periodically forced Van der Pol's equation:
\begin{equation}\label{vdPeq}
y''+\mu(y^2-1)y'+y=a\sin\left(2\pi\nu t\right),
\end{equation}
where $\mu\gg1$, $\nu$ represents the frequency and $a$ the amplitude of the forcing.
Let $t=\mu\tau$ and
\[
x=\frac{1}{\mu^2}\frac{\d y}{\d\tau}+\frac{y^3}{3}-3,
\]
which is also called the Li\'enard transformation.
Define $\veps=\frac{1}{\mu^2}$. Hence we can transform \eqref{vdPeq} into the following system
\begin{equation}\label{sfODE}
\left\{
   \begin{aligned}
   &\ \frac{\d x}{\d \tau}=a\sin\left(2\pi\nu\frac{1}{\sqrt{\veps}}\tau\right)-y\\
   &\ \frac{\d y}{\d \tau}=\frac{1}{\veps}\left(x-\frac{y^3}{3}+y\right).
   \end{aligned}
   \right.
\end{equation}

Let $\theta=\frac{1}{\sqrt{\veps}}\tau$.
Then note that \eqref{sfODE} is a multi-scale system,
including the slow variable $x$, the fast variable $y$ and
the highly oscillating time component $\theta$.
Van der Pol's equation is one of the most important examples of multi-scale systems.
It exhibits a wide variety of interesting dynamical phenomena and appears frequently in applications in various fields, including, but not limited to, neuroscience, seismology, electrical circuits, networks, and systems biology.
Multi-scale models appear frequently in many real-world dynamical systems,
such as climate weather interactions (see e.g. \cite{Kifer1999, MTV2001}),
macro-molecules (see e.g. \cite{BKPR2006, KK2013}),
stochastic volatility in finance (see e.g. \cite{FFK2012}), etc.

Usually, studying multi-scale models is relatively difficult because of the presence of widely separated
times scales and the interactions between them.
To understand the dynamics of multi-scale models, it is desirable to seek a simplified system,
which can simulate and predict the evolution of the original system over a long time scale.
This is the basic idea of the averaging principle.

The averaging principle was first developed for deterministic systems by Krylov,
Bogolyubov and Miltropolsky \cite{BM1961, KB1943},
and extended to SDEs by Khasminskii \cite{Khas1968}.
After that, numerous studies have been carried out on the averaging principle for SDEs, see e.g. \cite{BK2004, FW2006, HL2020, KY2004, LRSX2020, MSV1991, RX2021, Vere1990, XDX2012}
and the references therein. Furthermore, similar results concerning stochastic partial differential equations can be found in references like
\cite{CerraAAP2009, CF2009, CL2023, CLR2023, DSXZ2018, DW2014, Gao2019, HLL2021, WR2012}.
Despite considerable advances in the averaging principle, it seems that there
is no work on multi-scale SDEs which includes the slow variable, the fast variable and
the highly oscillating time component.

Building upon the motivations mentioned above, in this paper
we investigate the averaging principle of the following multi-scale SDEs
with polynomial nonlinearity:
\begin{equation}\label{sfmain}
 \left\{
   \begin{aligned}
   &\ \d X^{\veps}_t=f(\veps^{-\gamma}t,X^{\veps}_t,Y^{\veps}_t)\d t+\sigma(\veps^{-\gamma}t,X^{\veps}_t)\d W^{1}_t\\
   &\ \d Y^{\veps}_t=(\veps^{-2\alpha}B(X^{\veps}_t,Y^{\veps}_t)
   +\veps^{-\beta}b(X^{\veps}_t,Y^{\veps}_t))\d t
   +\veps^{-\alpha}g(X^{\veps}_t,Y^{\veps}_t)\d W^{2}_t,
   \end{aligned}
   \right.
\end{equation}
where $0\leq\beta<2\alpha$, $0<\gamma<2\alpha$,
$f:\R^{1+d_1+d_2}\rightarrow \R^{d_1}$, $\sigma:\R^{1+d_1}\rightarrow\R^{d_1}\otimes\R^{d_1}$,
$B:\R^{d_1+d_2}\rightarrow\R^{d_2}$, $b:\R^{d_1+d_2}\rightarrow \R^{d_2}$,
$g:\R^{d_1+d_2}\rightarrow\R^{d_2}\otimes\R^{d_2}$,
and $0<\veps\ll1$ is a small parameter.
Here $W^1$ and $W^2$ are independent standard Brownian motions.
See Section \ref{Rstate} for detailed conditions for coefficients
$f$, $\sigma$, $B$, $b$ and $g$.
If $f$ and $\sigma$ are time-independent, and $b\equiv0$ then
\eqref{sfmain} reduces to the classical slow-fast SDEs.

More precisely, as the time scale $\veps$ goes to zero
we first consider the so-called {\em first Bogolyubov theorem},
which focuses on the strong convergence of the solution of \eqref{sfmain} to
that of the following averaged equation on finite time intervals:
\begin{equation}\label{avemain}
\d\bar{X}_t=\bar{f}(\bar{X}_t)\d t+\bar{\sigma}(\bar{X}_t)\d W_t^1,
\end{equation}
where
\begin{equation*}
\bar{f}(x)=\lim_{T\rightarrow\infty}\frac{1}{T}\int_t^{t+T}\int_{\R^{d_2}}f(s,x,y)\mu^x(\d y)\d s,~ \lim_{T\rightarrow\infty}\frac{1}{T}\int_t^{t+T}|\sigma(s,x)-\bar{\sigma}(x)|_{HS}^2\d s=0
\end{equation*}
for all $(t,x)\in\R^{1+d_1}$ (see more details about the assumptions in Section 2.1),
where $|\cdot|_{HS}$ is the Hilbert-Schmidt norm; see Theorem \ref{FAPth}.
Here $\mu^x$ is the invariant measure of
\begin{equation}\label{Feq}
 \d Y_t^x=B(x,Y_t^x)\d t+g(x,Y_t^x)\d W_t.
\end{equation}
Secondly, we consider the case where the coefficients $f$ and $\sigma$ are time-independent.
In this case, we assume $2\alpha=1$ for simplicity.
Then by using regularity estimates for the solutions
to Poisson equations, we obtain the optimal strong convergence rate, i.e.
\begin{equation}\label{Ineq01}
\mbE\left(\sup_{t\in[0,T]}|X_t^\veps-\bar{X}_t|^2\right)\leq C_T\veps,
\end{equation}
where $C_T$ is a constant (see Theorem \ref{Opcrth}).
Furthermore, if $\sigma$ is constant, we study the {\em normal deviation}. In other words,
we prove that the normalized difference
\[
Z_t^\veps:=\frac{X_t^\veps-\bar{X}_t}{\sqrt{\veps}}
\]
weakly converges to $\bar{Z}_t$ as $\veps$ goes to zero. Here $\bar{Z}_t$ is the solution to
\begin{equation}\label{NDLeq}
\d \bar{Z}_t=\nabla\bar{f}(\bar{X}_t)\bar{Z}_t\d t
+G(\bar{X}_t)\d \widetilde{W}_t^1,
~\bar{Z}_0=0\in\R^{d_1},
\end{equation}
where
\begin{equation*}
G(x)=\sqrt{\int_0^\infty\int_{\R^{d_2}}\mbE\left[f(x,Y_t^x(y))-\bar{f}(x)\right]
\left[f(x,y)-\bar{f}(x)\right]^T\mu^x(\d y)\d t},
\end{equation*}
and $\widetilde{W}_t^1$ is another standard Brownian motion that is independent of $W_t^1$
(see Theorem \ref{NDth}). Here $Y_t^x(y),t\geq0$ is the solution to
\eqref{Feq} with $Y_0^x=y$.
Such a result is also known as the Gaussian approximation.
In addition, our investigation includes the study of the {\em global averaging principle}
in the weak sense, i.e.
we prove that the measure attractor of \eqref{sfmain} converges,
as $\veps$ goes to zero, to that of \eqref{avemain} (see Theorem \ref{GAPth}).
Finally, we establish the {\em second Bogolyubov theorem}, which states
that the stationary solution of \eqref{avemain}
approximates the recurrent solution of \eqref{sfmain} in the sense of \eqref{SMSeq}
in Theorem \ref{SAPth}.

Compared with \cite{LRSX2020}, where they proved the fist Bogolyubov theorem
for two time scale SDEs with locally Lipschitz coefficients, we study a broader class of SDEs
\eqref{sfmain}. The slow process $X_t^\veps$ here interacts not only with
the fast process $Y_t^\veps$ but also with the highly oscillating time component $\veps^{-\gamma}t$.
To overcome the difficulty, we employ the Poisson equation,
the technique of time discretization and the technique of truncation.
And we also obtain the optimal rate of strong convergence
when the coefficients of the slow equation are time-independent
and $f$ satisfies the following locally monotone condition: for all $x_1,x_2\in\R^{d_1}$ and $y\in\R^{d_2}$
\begin{equation}\label{Inmeq}
\langle f(x_1,y)-f(x_2,y),x_1-x_2\rangle\leq M\left(1+|y|^{\theta_2}\right)|x_1-x_2|^2,
\end{equation}
where $M\geq0$ and $\theta_2\geq1$; see Section \ref{ND} for more detailed conditions.

The rate of convergence is interesting in its own right since it plays a crucial role in
constructing efficient numerical schemes.
The main motivation comes from the well-known Heterogeneous Multi-Scale
Methods used to approximate the slow component; see e.g. \cite{ELV2005, KT2004}.
Recall that the optimal strong convergence order is also obtained in \cite{SSWX2022} for monotone SDEs.
However, it should be noted that our result cannot be covered by those in \cite{SSWX2022} because there it is assumed that
the coefficient $f$ must be monotone uniformly with respect to (in short, w.r.t.) $y$.
There are certain classes of systems, such as $f(x,y)=x-x^3+y^k\sin x, ~k\in\N$, that do not satisfy monotonicity uniformly w.r.t. $y$
but satisfy condition \eqref{Inmeq}; see Example \ref{Ex1}.

In order to obtain the optimal strong convergence order under monotonicity conditions,
we need to estimate two crucial terms:
\begin{align*}
\msI_1
&
:=\mbE\left(\sup_{t\in[0,T]}\int_0^t\langle f(X_s^\veps,Y_s^\veps)-\bar{f}(X_s^\veps),X_s^\veps-\bar{X}_s\rangle\d s\right),
\end{align*}
\begin{align*}
\msI_2
&
:=\mbE\left(\sup_{t\in[0,T]}\int_0^t\langle\bar{f}(X_s^\veps)-\bar{f}(\bar{X}_s),X_s^\veps-\bar{X}_s\rangle\d s\right).
\end{align*}
For $\msI_1$, regularity estimates for solutions to the Poisson equation can be employed to handle its estimation.
The remaining part of the proof is to show that $\bar{f}$ is monotone
when $f$ satisfies the locally monotone condition \eqref{Inmeq}.
Actually, thanks to the stability of the stationary solution to \eqref{Feq},
we can complete the proof (see Lemma \ref{AFMClem} for more details).

The first Bogolyubov theorem can be viewed as a functional law of large numbers,
indicating the convergence of the slow process $X_t^\veps$ to the averaged process $\bar{X}_t$.
However, it is crucial to acknowledge that even for small positive values of $\veps$,
the slow process $X_t^\veps$ still experiences fluctuations around the averaged process $\bar{X}_t$.
Consequently, it is natural to go one step further and consider
the functional central limit theorem, i.e. the normal deviation.
By studying these deviations, we can contribute to the understanding of
the behavior of the system and its relationship with the averaged process.
The fundamental paper about the normal deviation of multi-scale SDEs is by Khasminskii \cite{Khas1966}.
Since then, further developments were acquired;
see e.g. \cite{CerraiJMPA2009, HL2018, KV2014, KT2004, Kons2014, PTW2012, RX2021, WR2012}.

To the best of our knowledge, it seems that there is no work on the normal deviation
for SDEs with polynomial nonlinearity. Therefore,
we focus on investigating the deviations
of the solutions $X_t^\veps$ to monotone SDEs of type \eqref{sfmain} from $\bar{X}_t$ in this paper.
Specifically, we establish that, under appropriate conditions, the deviation process
$Z_t^\veps$ converges weakly to an Ornstein-Uhlenbeck type process $\bar{Z}_t$.
Such a result is closely related to the homogenization for solutions of
partial differential equations with singularly perturbed terms; see e.g. \cite{FW1979}.

More specifically, we prove that for any $\varphi\in C_b^\infty(\R^{d_1})$
\begin{equation}\label{InNDeq}
\lim_{\veps\rightarrow0}\sup_{t\in[0,T]}\left|\mbE\varphi(Z_t^\veps)-\mbE\varphi(\bar{Z}_t)\right|=0,
\end{equation}
where $C_b^\infty(\R^{d_1})$ is the space of all smooth functions
with bounded $j$-th derivatives for all integers $j\in[0,\infty)$.
To this end, employing It\^o's formula, we have
\begin{align*}
&
\left|\mbE\varphi(Z_t^\veps)-\mbE\varphi(\bar{Z}_t)\right|\\\nonumber
&
\leq
\left|\mbE\int_{0}^t\frac{1}{\sqrt{\veps}}\langle f(X_s^\veps,Y_s^\veps)-\bar{f}(X_s^\veps),\nabla\varphi(Z_s^\veps)\rangle-\frac12 Tr[\nabla^2\varphi(Z_s^\veps)GG^T(X_s^\veps)]\d s\right|\\\nonumber
&\quad
+\left|\mbE\int_0^t\langle \nabla\bar{f}(\bar{X}_s+\iota(X_s^\veps-\bar{X}_s))Z_s^\veps,\nabla\varphi(Z_s^\veps)\rangle
-\langle\nabla\bar{f}(\bar{X}_s)\bar{Z}_s,\nabla\varphi(\bar{Z}_s)\rangle\d s\right|\\\nonumber
&\quad
+\left|\mbE\int_0^t\frac12 Tr[\nabla^2\varphi(Z_s^\veps)GG^T(X_s^\veps)]
-\frac12 Tr[\nabla^2\varphi(\bar{Z}_s)GG^T(\bar{X}_s)]\d s\right|=:\msI_1+\msI_2+\msI_3,
\end{align*}
where $\iota\in[0,1]$. Therefore, we just need to show that
$\msI_1$, $\msI_2$ and $\msI_3$ go to zero as $\veps\rightarrow0$.
Combining the regularity estimates of the solutions to the Poisson equation
and the optimal strong convergence \eqref{Ineq01}, we prove that $\lim_{\veps\rightarrow0}\msI_1=0$.
For $\msI_2$ and $\msI_3$, we first prove that the subset
$$\{Z^\veps,\bar{Z}:0<\veps\leq1\}\subset C([0,T];\R^{d_1})$$
is tight; see Lemma \ref{TightLem1}. By utilizing the tightness of the set and
the separation properties of $C([0,T];\R^{d_1})$,
we can then conclude that $\lim_{\veps\rightarrow0}\msI_2=0$ and $\lim_{\veps\rightarrow0}\msI_3=0$,
and complete the proof of \eqref{InNDeq}; see Section \ref{PNDth} for details.
This provides a better approximation and is also known as
Van Kampen's scheme in physics; see e.g. \cite{Arnold2001}.

Another main ingredient of this paper is to study the long-time
asymptotic behavior of solutions to \eqref{sfmain}.
So, we aim to establish the global averaging principle in the weak sense.
Namely, we prove that the attractor of \eqref{sfmain} tends to that of \eqref{avemain}
in the space of probability measures. The global averaging principle of deterministic systems
was proved in \cite{HV1990, Ilyin1996, Ilyin1998, Zelik2006} and the references therein.
There are few works on the global averaging principle for stochastic systems.
In \cite{CL2023, CLR2023} the global averaging principle was established
in the weak sense for stochastic partial differential equations with highly time oscillating
components.

As we mentioned before, the drift coefficient $f$ in \eqref{sfmain} exhibits a complex and general structure, and it encompasses not only
the highly time oscillating component $\veps^{-\gamma} t$, but also the fast variable
$Y_t^\veps$.
Due to the coupling between the slow process $X_t^\veps$ and the fast process $Y_t^\veps$
in \eqref{sfmain}, it is necessary to consider the attractor of the entire multi-scale system
as a unified entity. The dynamics of the slow process and the fast process are interconnected, and their mutual influence plays a role in shaping the behavior of the system as a whole.
Therefore, we consider the attractor of the whole multi-scale system \eqref{sfmain} instead of
the single slow equation.

More exactly, fix $0<\veps\leq1$.
Define the transition probability
\[
P^\veps(s,(x,y),t,D):=\mbP\circ\left[\left(X_{s,t}^\veps(x),Y_{s,t}^\veps(y)\right)\right]^{-1}(D)
\]
for all $s\leq t$, $(x,y)\in\R^{d_1+d_2}$ and $D\in\mcB(\R^{d_1+d_2})$,
where $\mcB(\R^{d_1+d_2})$ is the Borel $\sigma$-algebra of $\R^{d_1+d_2}$.
Then for any $0<\veps\leq1$ it associates the Markov operator $P^*_\veps$
acting on the probability measure space $\msP(\R^{d_1+d_2})$:
\begin{equation}\label{InTP}
P_\veps^*(s,t,m)(D):=\int_{\R^{d_1+d_2}}P^\veps(s,(x,y),t,D)m(\d(x,y))
\end{equation}
for any $m\in\msP(\R^{d_1+d_2})$ and $D\in\mcB(\R^{d_1+d_2})$.

Note that $P_\veps^*$ is time inhomogeneous, so we employ the method of skew
product to consider its pullback attractors.
For detailed definitions of skew product flows, attractors, and pullback attractors, please see Section \ref{PDS}.
Finally, we investigate the convergence of pullback measure attractors
for \eqref{sfmain}.

Finally, we establish the second Bogolyubov theorem for \eqref{sfmain}.
Since numerous physical models have periodic forces, such as \eqref{vdPeq},
we approximate the periodic solution of the original system
by utilizing the stationary solution to the averaged equation.
It is worth noting that we obtain the convergence of a broader class of recurrent solutions,
including periodic, quasi-periodic, almost periodic solutions among others; see Remark \ref{MRS}.
For brevity, we focus on the analysis of quasi-periodic solutions in this paper.
To be specific, under some suitable conditions we show that there
exists a unique solution to \eqref{sfmain}, which is
quasi-periodic in distribution,
if $f$ and $\sigma$ are quasi-periodic. Then the law of the slow component of
the quasi-periodic solution converges to the law of the stationary solution for \eqref{avemain}
uniformly w.r.t. $t\in\R$ as the time scale $\veps$ goes to zero.

Now we summarize the structure of the paper. In Section 2, we state our main results.
In Section 3, we study the frozen equation and the Poisson equation.
In Section 4, we investigate the first Bogolyubov theorem.
In Section 5, we first prove the optimal strong convergence order.
Then we establish the normal deviation. In Section 6, we prove the global averaging principle
in the weak sense and the second Bogolyubov theorem.
In the appendix at the end, we show the existence and uniqueness of solutions to \eqref{sfmain}.

~\\

{\bf Notations.}
Throughout this paper, let $|\cdot|$ be the Euclidean norm
and $\langle\cdot,\cdot\rangle$ be the Euclidean inner product on $\R^d, d\in\N$.
For a vector-valued or matrix-valued function $x\mapsto\varphi(x)$ defined on $\R^d$ or $(x,y)\mapsto\varphi(x,y)$
defined on $\R^{d_1+d_2} ,d_1,d_2\in\N$, we denote the $i$-th order derivative of $\varphi$ by $\nabla^i \varphi(x)$, and the $i$-th and $j$-th order partial derivative of $\varphi(x,y)$ w.r.t. $x$ and $y$ by $\partial_x^i \partial_y^j \varphi(x,y)$, respectively,
where $i,j\in\N$.
For all $i\in\N$,
let $C^{i,2,4}(\R^{d+d_1+d_2})$ be the space of all continuous mappings
$f:\R^{d+d_1+d_2}\rightarrow\R^{d_1+d_2}$
such that $\partial_h^{i'} f$ and $\partial_{x}^{j'}\partial_{y}^{k'}f$  are continuous for any
$0\leq i'\leq i$ and $0\leq 2j'+k'\leq4$.
Let $[C]$ denote the integer part of $C$ for any $C\geq0$.
We use $C_b^k(\R^{d_1})$ to denote the space of all functions $f:\R^{d_1}\rightarrow\R$ whose $j$-th derivative is continuous and bounded for all $j\in[0,k]$.
For any complete metric space $(\mcX,d)$, let $C(\R,\mcX)$ be the space of all continuous mappings $\varphi:\R\rightarrow\mcX$ with the compact-open topology.
Define the Hausdorff semi-metric
${\rm dist}_{\mcX}(A,B):=\sup_{x\in A}\inf_{y\in B}d(x,y)$ for any $A,B\subset \mcX$.
Let $A^T$ denote the transpose of a matrix $A$.
Let $\msL(X)$ denote the distribution or law of random variable $X$,
and $(\Omega,\msF,\mbP)$ be a complete probability space.
We use $C$ with or without subscripts to denote some constant, which may change from line to line. In this paper, solutions to SDEs are always meant to be strong solutions.

\section{Statement of the main results}\label{Rstate}
In this section, we formulate our main results.

\subsection{The first Bogolyubov theorem}
First of all, we introduce the following conditions about the coefficients $B$, $b$, and $g$:
\begin{enumerate}
  \item [{\bf(H$_y^1$)}] There exist constants $\eta>0$, $\eta'\geq0$, $\theta\geq2$, $\widetilde{\eta}\in\R$ and $K_1\in\R$
  such that for all $(x,y)\in\R^{d_1+d_2}$
  \[
  2\langle B(x,y),y\rangle+|g(x,y)|^2_{HS}\leq -\eta|y|^2-\eta'|y|^{\theta}+K_1,\quad
  2\langle b(x,y),y\rangle\leq \widetilde{\eta}|y|^2+K_1.
  \]
\end{enumerate}
\begin{enumerate}
  \item [{\bf(H$_y^2$)}] There exists a constant $K_2>0$ such that for all $(x,y)\in\R^{d_1+d_2}$
  \[
  K_2^{-1}I\leq a(x,y)\leq K_2I,
  \]
  where $a(x,y)=\frac12gg^T(x,y)$.
  \item [{\bf(H$_y^3$)}]
\begin{itemize}
\item[(i)] There exist constants $\eta>0$ and $\theta\geq2$ such that
for all $x\in\R^{d_1}$,$y_1,y_2\in\R^{d_2}$
\begin{align*}
2\langle B(x,y_1)-B(x,y_2),y_1-y_2\rangle
+|g(x,y_1)-g(x,y_2)|_{HS}^2
\leq-\eta|y_1-y_2|^2.
\end{align*}
\item[(ii)] There exists a constant $L_g>0$ such that for all $(x_1,y_1),(x_2,y_2)\in\R^{d_1+d_2}$
$$|g(x_1,y_1)-g(x_2,y_2)|_{HS}\leq L_g(|x_1-x_2|+|y_1-y_2|).$$
\item[(iii)] There exist constants $\kappa_1,\kappa_2\geq 1$ and $K_3>0$ such that
for all $x_1,x_2\in\R^{d_1}$ and $y\in\R^{d_2}$
\begin{equation*}
|B(x_1,y)-B(x_2,y)|\leq K_3\left(1+|y|^{\kappa_2}\right)|x_1-x_2|,
\end{equation*}
\begin{align*}
&
|b(x_1,y_1)-b(x_2,y_2)|\\
&
\leq K_3\left(1+|x_1|^{\kappa_1}+|x_2|^{\kappa_1}+|y_1|^{\kappa_2}+|y_2|^{\kappa_2}\right)
\left(|x_1-x_2|+|y_1-y_2|\right).
\end{align*}
\end{itemize}
\item[{\bf(H$_y^4$)}]  There exist constants $\varsigma>0$ and $\varsigma_1>8$
    such that for all $(x,y)\in\R^{d_1+d_2}$ and $\xi\in\R^{d_2}$
    \begin{align*}
    2\langle\partial_y B(x,y)\xi,\xi\rangle+(\varsigma_1-1)|\partial_y g(x,y)\xi|^2
        \leq -\varsigma|\xi|^2.
    \end{align*}

\item [{\bf(H$_y^5$)}] There exist $K_3>0$ and $\kappa_1\geq1$ such that
$B\in C^{3,3}(\R^{d_1+d_2})$ and $g\in C^{3,3}(\R^{d_1+d_2})$ satisfy
\[
\sum_{1\leq i+j\leq3}\left(|\partial_y^j\partial_x^i B|+|\partial_y^j\partial_x^i g|\right)
\leq K_3\left(1+|y|^\kappa\right).
\]
\end{enumerate}

\begin{remark}\rm
Note that {\bf{(H$_y^4$)}} and {\bf{(H$_y^5$)}} are not necessarily required
if we study the averaging principle by the technique of time discretization.
However, in this context, we aim to investigate the optimal strong convergence rate for the
averaging principle based on the Poisson equation,
and we need {\bf{(H$_y^4$)}}--{\bf{(H$_y^5$)}} to obtain the
well-posedness of the Poisson equation (see e.g. \cite{CDGOS2022,SSWX2022}).
\end{remark}

Note that if {\bf(H$_y^1$)} and {\bf(H$_y^2$)} hold, then for any $x\in\R^{d_1}$
\begin{equation}\label{fmain}
\d Y_t^x=B(x,Y_t^x)\d t+g(x,Y_t^x)\d W^2_t
\end{equation}
admits a unique invariant measure $\mu^x$; see e.g. \cite{PV2003}.
Set
$$
\hat{f}(t,x):=\int_{\R^{d_2}}f(t,x,y)\mu^x(\d y),
\quad \forall (t,x,y)\in\R^{1+d_1+d_2}.
$$
Next, we introduce the conditions concerning the coefficients $f$ and $\sigma$:
\begin{enumerate}
  \item[{\bf(A$_f$)}] Let $R\in\R_+$. There exist $\omega_R^f:\R\rightarrow\R_+$ satisfying $\omega_R^f(T)\rightarrow0$ as $T\rightarrow\infty$
      and $\bar{f}:\R^{d_1}\rightarrow\R^{d_1}$
  such that for all $t\in\R$ and $|x|\leq R$
  \[
  \frac1T\left|\int_t^{t+T}(\hat{f}(s,x)-\bar{f}(x))\d s\right|\leq\omega_R^f(T).
  \]
  \item[{\bf(A$_\sigma$)}] There exists
$\omega^\sigma:\R\rightarrow\R_+$ satisfying $\omega^\sigma(T)\rightarrow0$ as $T\rightarrow\infty$
and $\bar{\sigma}:\R^{d_1}\rightarrow\R^{d_1}\otimes\R^{d_1}$
  such that for all $(t,x)\in\R^{1+d_1}$
  \[
  \frac1T\int_t^{t+T}|\sigma(s,x)-\bar{\sigma}(x)|_{HS}^2\d s\leq\omega^\sigma(T)(1+|x|^2).
  \]
\end{enumerate}
\begin{enumerate}
\item [{\bf(H$_x^1$)}] There exist constants $K_4,K_5\in\R$ such that for all $(t,x,y)\in\R^{1+d_1+d_2}$
  \[
  2\langle f(t,x,y),x\rangle+|\sigma(t,x)|^2\leq K_4(1+|x|^2)+K_5|y|^{\theta},
  \]
where $\theta$ is as in {\bf(H$_y^1$)}.
\item  [{\bf(H$_x^2$)}] There exist constants $K_6>0$, $\theta_1,\theta_2>1$ such that for all
  $(t,x,y)\in\R^{1+d_1+d_2}$
  \begin{align*}
  &
|\partial_tf(t,x,y)|+\sum_{0\leq 2i+j\leq4}|\partial_y^j\partial_x^if(t,x,y)|
  \leq K_6(1+|x|^{\theta_1}+|y|^{\theta_2}).
  \end{align*}
\item [{\bf(H$_x^3$)}] There exist constants $K_7,L_\sigma>0$
  such that for all $t\in\R$ and
  $x_1,x_2\in\R^{d_1}$
  \[
  |\sigma(t,x_1)-\sigma(t,x_2)|_{HS}\leq L_\sigma|x_1-x_2|,\quad |\sigma(t,0)|\leq K_7.
  \]
\end{enumerate}


\begin{remark}\label{Rem01}\rm
(i) If $\theta=2$ or $K_5=0$ in {\bf(H$_x^1$)}, then we can assume that $\eta'=0$ in {\bf(H$_y^1$)}.

(ii) Note that {\bf(H$_x^2$)} implies that there exists a constant $C>0$, depending only on $K_6,\theta_1,\theta_2$,
such that for all $t\in\R$, $x_1,x_2\in\R^{d_1}$ and $y_1,y_2\in\R^{d_2}$
 \begin{align*}
 &
  |f(t,x_1,y_1)-f(t,x_2,y_2)|\\
  &
  \leq C\left(1+|x_1|^{\theta_1}+|x_2|^{\theta_1}+|y_1|^{\theta_2}+|y_2|^{\theta_2}\right)
  \left(|x_1-x_2|+|y_1-y_2|\right).
  \end{align*}

(iii) If $f$ and $\sigma$ satisfy {\bf(A$_f$)} and {\bf(A$_\sigma$)} respectively then
for all $t\in\R$ and $x\in\R^{d_1}$
$$\bar{f}(x)=\lim\limits_{T\rightarrow\infty}\frac{1}{T}\int_t^{t+T}\hat{f}(s,x)\d s,\quad
\bar{\sigma}(x)=\lim\limits_{T\rightarrow\infty}\frac{1}{T}\int_t^{t+T}\sigma(s,x)\d s.
$$

(iv) Assume that {\bf(H$_y^1$)}, {\bf(H$_y^3$)} and {\bf(H$_x^1$)}--{\bf(H$_x^3$)} hold.
Let $0<\veps\leq \sqrt[2\alpha]{\eta'/K_5}\wedge1$.
Then for
any $(x,y)\in\R^{d_1+d_2}$ there exists a unique solution
$(X_t^\veps(x),Y_t^\veps(y))$ to \eqref{sfmain} satisfying
$(X_0^\veps(x),Y_0^\veps(y))=(x,y)$; see Lemma \ref{EUSFSDEth} for details.
Moreover, if $g\in C_b(\R^{d_1+d_2})$ then for any $x\in\R^{d_1}$ there exists a unique solution $\bar{X}_t(x)$ to \eqref{avemain} with $\bar{X}_0(x)=x$; see
Remark \ref{EUASDE}.

(v) In this paper, we focus on the asymptotic dynamics of the multi-scale system \eqref{sfmain} when $\veps$ goes to zero. There exists $\veps_0>0$ such that our results in this paper hold for any $0<\veps\leq \veps_0$. Therefore, we state our results for all $0<\veps\leq1$ in this section for brevity.
\end{remark}

Now we establish the first Bogolyubov theorem for \eqref{sfmain}.

\begin{theorem}\label{FAPth}
Assume that {\bf(H$_x^1$)}--{\bf(H$_x^3$)}, {\bf(H$_y^1$)}--{\bf(H$_y^5$)}, {\bf(A$_f$)} and {\bf(A$_\sigma$)} hold.
Then we have
\begin{equation*}\label{FAPtheq}
\lim_{\var\rightarrow0}\sup_{t\in[0,T]}\mbE|X_t^\var(x)-\bar{X}_t(x)|^2=0.
\end{equation*}
\end{theorem}

\subsection{Normal deviation}\label{ND}

Let $f$ and $\sigma$ be independent of time $t$. Without loss of generality, assume that
$\alpha=\frac12$ . Then we can consider the following system
\begin{equation}\label{sfmain2}
  \left\{
   \begin{aligned}
   &\ \d X^{\veps}_t=f(X^{\veps}_t,Y^{\veps}_t)\d t+\sigma(X^{\veps}_t)\d W^{1}_t\\
   &\ \d Y^{\veps}_t=\left(\frac{1}{\veps}B(X^{\veps}_t,Y^{\veps}_t)
   +\frac{1}{\veps^{\beta}}b(X_t^\veps,Y_t^\veps)\right)\d t
   +\veps^{-\frac12}g(X^{\veps}_t,Y^{\veps}_t)\d W^{2}_t,
   \end{aligned}
   \right.
  \end{equation}
where $\beta<1$.
To obtain the normal deviation, we need the following condition:
\begin{enumerate}
  \item[{\bf(H$_x^4$)}]
  \begin{enumerate}
      \item[(i)]
There exist $M>0$ and $\theta_2>1$ such that for all $x_1,x_2\in\R^{d_1}$  and $y\in\R^{d_2}$
  \[
  \langle f(x_1,y)-f(x_2,y),x_1-x_2\rangle\leq M\left(1+|y|^{\theta_2}\right)|x_1-x_2|^2.
  \]
       \item[(ii)]
There exists $C>0$ and $\theta_2>1$ such that for all $x\in\R^{d_1}$ and $y_1,y_2\in\R^{d_2}$
\[
\left|f(x,y_1)-f(x,y_2)\right|\leq C\left(1+|y_1|^{\theta_2}+|y_2|^{\theta_2}\right)|y_1-y_2|.
\]
  \end{enumerate}
\end{enumerate}

\begin{remark}\label{OpRem}\rm
There is a work on the optimal strong convergence rate
for monotone SDEs in \cite{SSWX2022}.
Our condition {\bf(H$_x^4$)}--(i), however, is more general than the condition \cite[(2.1)]{SSWX2022}.
\end{remark}

The optimal strong convergence rate of the first Bogolyubov theorem for \eqref{sfmain2}
we prove in this paper is contained in the following theorem.
\begin{theorem}\label{Opcrth}
Assume that {\bf(H$_x^1$)}--{\bf(H$_x^4$)} and {\bf(H$_y^1$)}--{\bf(H$_y^5$)} hold.
Then there exists a constant $C_T>0$ such that
\begin{equation}\label{Opcrsfeq}
\mbE\left(\sup_{t\in[0,T]}|X_t^\veps(x)-\bar{X}_t(x)|^{2}\right)\leq C_T\veps.
\end{equation}
\end{theorem}

For simplicity, set
$X_t^\veps:=X_t^\veps(x)$,
$ \bar{X}_t:=\bar{X}_t(x)$
for all $x\in\R^{d_1}$ in the following.
Define $Z_t^\veps:=\frac{X_t^\veps-\bar{X}_t}{\sqrt{\veps}}$. It is clear that $Z_t^\veps$ solves
\begin{equation*}\label{NDeq}
\d Z_t^\veps=\frac{1}{\sqrt{\veps}}\left(f(X_t^\veps,Y_t^\veps)-\bar{f}(\bar{X}_t)\right)
\d t+\frac{1}{\sqrt{\veps}}\left(\sigma(X_t^\veps)-\sigma(\bar{X}_t)\right)\d W_t^1,
~Z_0^\veps=0\in\R^{d_1}.
\end{equation*}
If $\sigma(x)\equiv\sigma$ is a constant,
then we can show that, as $\veps\rightarrow0$, $Z_t^\veps$ converges weakly to $\bar{Z_t}$,
which is the solution of \eqref{NDLeq} with $\bar{Z}_0=0$.

\begin{theorem}\label{NDth}
Assume that $\sigma$ is a constant.
Furthermore, suppose that {\bf(H$_x^1$)}--{\bf(H$_x^4$)} and {\bf(H$_y^1$)}--{\bf(H$_y^5$)} hold.
Then for any $\varphi\in C_b^\infty(\R^{d_1})$
we have
\[
\lim_{\veps\rightarrow0}\sup_{t\in[0,T]}
\left|\mbE\varphi(Z_t^\veps)-\mbE\varphi(\bar{Z}_t)\right|=0.
\]
\end{theorem}

\subsection{Global averaging principle}

Now we investigate the convergence of measure attractors for multi-scale SDEs \eqref{sfmain}.
As mentioned in the Introduction,
$P^*_\veps$ defined by \eqref{InTP} is time inhomogeneous for any fixed $0<\veps\leq 1$.
Therefore, we employ the classical method called the method of skew product,
which has been widely used in studying non-autonomous problems arising
from deterministic differential equations and dynamical systems,
to analyze its pullback attractor; see Section \ref{PDS} for detailed definitions of
cocycle, skew product flow, attractors, and pullback attractors.

More precisely, for any $0<\veps\leq1$ we characterize $P^*_\veps$ as
a cocycle over some base space. Indeed,
let
$$v:=(x,y)^T, \quad F_{\varepsilon}(t,v):=\left(f(\veps^{-\gamma}t,x,y),\veps^{-2\alpha}
B(x,y)+\veps^{-\beta}b(x,y)\right)^{T}$$
and
$$G_\veps(t,v):=\left(\sigma(\veps^{-\gamma}t,x),\veps^{-\alpha} g(x,y)\right)^{T},
\quad W:=\left(W^{1},W^{2}\right)^{T}
$$
for all $0<\veps\leq1$ and $(t,x,y)\in\R^{1+d_1+d_2}$.
Then \eqref{sfmain} can be written as
\begin{equation}\label{eqsys}
\d V_{t}^{\varepsilon}=F_{\varepsilon}(t,V_{t}^{\varepsilon})\d t
+G_\veps(t,V_{t}^{\varepsilon})\d W_t.
\end{equation}
Fix $0<\veps\leq1$.
Let $\mbF_\veps:=(F_\veps,G_\veps)$, and
\begin{equation*}\label{DefHull}
\mcH(\mbF_\veps):=\overline{\{\mbF_\veps^\tau:~\tau\in\R\}}
\end{equation*}
with the closure being taken under the metric $d$ given by \eqref{BUCM} below,
where $\mbF^\tau$ is the {\em $\tau$-translation} of $\mbF$
for each $\mbF:\R\times\R^{d_1+d_2}\rightarrow\R^{d_1+d_2}$, i.e.
$\mbF^\tau(t,v):=\mbF(t+\tau,v)$,
for all $(t,v)\in\R^{1+d_1+d_2}$.
For any $\mbF_1,\mbF_2:\R\times\R^{d_1+d_2}\rightarrow\R^{d_1+d_2}$, let
\begin{equation}\label{BUCM}
d(\mbF_1,\mbF_2):=\sum_{n=1}^\infty\frac{1}{2^n}\frac{d_n(\mbF_1,\mbF_2)}{1+d_n(\mbF_1,\mbF_2)},
\end{equation}
where
$d_n(\mbF_1,\mbF_2):=\sup_{|t|\leq n,|v|\leq n}\left|\mbF_1(t,v)-\mbF_2(t,v)\right|.$

\begin{remark}\label{Rem0406}\rm
Fix $0<\veps\leq1$.
We note that for any $\widetilde{\mbF}_\veps:=(\widetilde{F}_\veps,\widetilde{G}_\veps)\in\mcH(\mbF_\veps)$
there exists $\{t_n\}\subset\R$ such that for all $l,r>0$
$$
\lim_{n\rightarrow\infty}\sup_{|t|\leq l,|v|\leq r}|F_\veps(t+t_n,v)-\widetilde{F}_\veps(t,v)|=0,
\quad
\lim_{n\rightarrow\infty}\sup_{|t|\leq l,|v|\leq r}|G_\veps(t+t_n,v)-\widetilde{G}_\veps(t,v)|_{HS}=0.
$$
Therefore,
if $(F_\veps,G_\veps)$ satisfies {\bf(H$_y^1$)}, {\bf(H$_y^3$)} and {\bf(H$_x^1$)}--{\bf(H$_x^3$)}
then $\widetilde{\mbF}_\veps:=(\widetilde{F}_\veps,\widetilde{G}_\veps)$ does so for any
$\widetilde{\mbF}_\veps\in\mcH(\mbF_\veps)$.
Furthermore, let $\sigma_\tau\mbF:=\mbF^\tau,~\tau\in\R$. Then
$(\mcH(\mbF_\veps),\R,\sigma)$ is a shift dynamical system (also called Bebutov shift flow);
see Definition \ref{DefNDY}.

\end{remark}

Assume that {\bf(H$_y^1$)}, {\bf(H$_y^3$)} and {\bf(H$_x^1$)}--{\bf(H$_x^3$)} hold.
Fix $0<\veps\leq1$. We aim to show that $P_\veps^*$ is a cocycle over the base space $(\mcH(\mbF_\veps),\R,\sigma)$.
It follows from Remark \ref{Rem0406} that for each $\widetilde{\mbF}_\veps\in\mcH(\mbF_\veps)$
and $v\in\R^{d_1+d_2}$, there exists a unique solution $V_{\widetilde{\mbF}_\veps}(t,s,v)$
of \eqref{eqsys} by replacing $\mbF_\veps$ with $\widetilde{\mbF}_{\varepsilon}$, i.e.
\begin{equation*}
\d V_{\widetilde{\mbF}_\veps}(t,s,v)=\widetilde{F}_{\varepsilon}(t,V_{\widetilde{\mbF}_\veps}(t,s,v))\d t
+\widetilde{G}_\veps(t,V_{\widetilde{\mbF}_\veps}(t,s,v))\d W_t,~V_{\widetilde{\mbF}_\veps}(s,s,v)=v.
\end{equation*}
Here we explicitly denote the solution
$V_{\widetilde{\mbF}_\veps}(t,s,v)$ with a subscript $\widetilde{\mbF}_\veps$
to indicate its dependence on $\widetilde{\mbF}_\veps$.
Similarly, we express the dependence of the associated Markov operators on $\widetilde{\mbF}_\veps$
by writing
\[
P_\veps^*\left(t,s,\widetilde{\mbF}_\veps,m\right)(D)
:=\int_{\R^{d_1+d_2}}\mbP\circ\left(V_{\widetilde{\mbF}_\veps}(t,s,v)\right)^{-1}(D)m(\d v)
\]
for all $D\in\mcB(\R^{d_1+d_2})$ and $m\in\msP(\R^{d_1+d_2})$.
Set
$$
P^*_\veps(t,\widetilde{\mbF}_\veps,m)(D):=P^*_\veps(0,t,\widetilde{\mbF}_\veps,m)(D).
$$
Then we show that for any $0<\veps\leq1$ $P^*_\veps$ is a cocycle
over the base space $(\mcH(\mbF_\veps),\R,\sigma)$, and
\[
\Phi(t,\widetilde{\mbF}_\veps,m):=
\left(\sigma_t\widetilde{\mbF}_\veps,P_\veps^*(t,\sigma_t\widetilde{\mbF}_\veps,m)\right)
\]
is the homogeneous Markov semi-flow in the extended phase space; see Lemma \ref{SPFth}.
We call $(P_\veps^*,\sigma)$ a skew product flow.
Furthermore, we consider the existence of pullback attractors for $(P_\veps^*,\sigma)$ and the convergence of the pullback attractors when the time scale goes to zero; see Theorem \ref{GAPth} below for more details.
To this end, we need the following dissipativity
condition:\begin{itemize}
\item[{\bf(H$_x^5$)}] There exist $\lambda_1,K_4,K_5>0$ such that for any $(t,x,y)\in\R^{1+d_1+d_2}$
\[
2\langle f(t,x,y),x\rangle+|\sigma(t,x)|_{HS}^2\leq-\lambda_1|x|^2+K_5|y|^{\theta}+K_4,
\]
where $\theta$ is as in {\bf(H$_y^1$)} and {\bf(H$_x^1$)}.
\end{itemize}

For any $v:=(x,y)^T\in\R^{d_1+d_2}$, let
$\pi_1(v):=x$ and  $\pi_2(v):=y$.
Define
\[
\msP_{2,\theta}(\R^{d_1+d_2}):=\left\{m\in \msP_2(\R^{d_1+d_2}):\int_{\R^{d_2}}|y|^{\theta}~m\circ\pi_2^{-1}(\d y)<\infty\right\},
\]
equipped with the following bounded Lipschitz distance (also called Fortet-Mourier distance)
\[
d_{BL}(m_1,m_2):=\sup\left\{\left|\int_{\R^{d_1+d_2}} f\d m_1-\int_{\R^{d_1+d_2}} f\d m_2\right|:\|f\|_{BL}\leq1\right\}
\]
for all $m_1,m_2\in \msP_{2,\theta}(\R^{d_1+d_2})$,
where $\|f\|_{BL}:=Lip(f)+\|f\|_\infty$ for all Lipschitz continuous $f\in C_b(\R^{d_1+d_2})$.
It can be verified that $(\msP_{2,\theta}(\R^{d_1+d_2}),d_{BL})$ is a Polish space.
We say that $D\subset \msP_{2,\theta}(\R^{d_1+d_2})$ is {\em bounded} if there exist $r_1,r_2>0$ such that for all $m\in D$
\[
\int_{\R^{d_1+d_2}}|v|^2m(\d v)\leq r_1,\quad
\int_{\R^{d_2}}|y|^\theta m\circ\pi_2^{-1}(\d y)\leq r_2.
\]

\begin{theorem}\label{GAPth}
Consider equation \eqref{sfmain}. Assume that
{\bf(H$_y^1$)}--{\bf(H$_y^5$)}, {\bf(H$_x^2$)}--{\bf(H$_x^3$)}, {\bf(H$_x^5$)},
{\bf(A$_f$)} and {\bf(A$_\sigma$)} hold. If $\mcH(\mbF_\veps)$ is compact for any $0<\veps\leq1$,
then the following conclusions hold:
\begin{enumerate}
  \item For any $0<\veps\leq1$ \eqref{sfmain} is  associated with a skew product flow $(\sigma,P^*_\veps)$ on $\left(\msP_{2,\theta}(\R^{d_1+d_2}),d_{BL}\right)$,
  and $(\sigma,P^*_\veps)$ admits a pullback attractor $\msA^\veps$ with component subsets $\msA_{\widetilde{\mbF}_\veps},~\widetilde{\mbF}_\veps\in\mcH(\mbF_\veps)$;
  \item The averaged equation has a global attractor $\bar{\mcA}$;
  \item Furthermore, for all $\widetilde{\mbF}_\veps\in\mcH(\mbF_\veps)$
\[
\lim_{\veps\rightarrow0}{\rm dist}_{\msP(\R^{d_1})}
\left(\Pi_1(\msA_{\widetilde{\mbF}_\veps}),\bar{\mcA}\right)=0,
\]
where $\bar{\mcA}$ is the global attractor of $\bar{P}^*$ and
$\Pi_1m:=m\circ\pi_1^{-1}$ for all $m\in \msP(\R^{d_1+d_2})$.
\end{enumerate}
\end{theorem}

\subsection{The second Bogolyubov theorem}
In this subsection, we will consider the convergence of the recurrent solutions for \eqref{sfmain}.
More precisely, we show that,
as the time scale $\veps$ goes to 0, the slow components of the quasi-periodic solutions weakly converge to the stationary solution of the averaged equation.
To this end, we need the following conditions:
\begin{itemize}
\item [{\bf(H$_x^6$)}] There exist constants $\lambda_1,\lambda_2>0$ such that for all $t\in\R$ and
$x_1,x_2\in\R^{d_1}$, $y_1,y_2\in\R^{d_2}$
\begin{align*}
&
2\langle f(t,x_1,y_1)-f(t,x_2,y_2),x_1-x_2\rangle+|\sigma(t,x_1)-\sigma(t,x_2)|_{HS}^2\\
&
\leq-\lambda_1|x_1-x_2|^2+\lambda_2|y_1-y_2|^2.
\end{align*}
\item [{\bf(H$_y^6$)}] There exists $L_b>0$ such that for all $t\in\R$ and
$y_1,y_2\in\R^{d_2}$
  \[
  |b(x_1,y_1)-b(x_2,y_2)|\leq L_b\left(|x_1-x_2|+|y_1-y_2|\right).
  \]
\end{itemize}

\begin{remark}\rm
It can be verified that {\bf(H$_x^3$)} and {\bf(H$_x^6$)} imply {\bf(H$_x^5$)}.
\end{remark}

First, we recall the definition of quasi-periodic functions. Let $\mcX$ be a Polish space.
\begin{definition}\label{QPF}\rm
A function $\varphi\in C(\R,\mcX)$ is called {\em quasi-periodic with the spectrum
of frequencies $\nu_1,...,\nu_k$} if it satisfies the following conditions:
\begin{itemize}
  \item [(i)] the numbers $\nu_1,...,\nu_k$ are rationally independent;
  \item [(ii)] there exists a continuous function $\phi:\R^k\rightarrow\mcX$
  such that for all $(t_1,...,t_k)\in\R^k$
  \[
  \phi(t_1+2\pi,...,t_k+2\pi)=\phi(t_1,...,t_k);
  \]
  \item [(iii)] $\varphi(t)=\phi(\nu_1t,...,\nu_kt)$ for $t\in\R$.
\end{itemize}

\end{definition}

\begin{definition}\rm
We say a $\mcX$-valued continuous stochastic process $X_t,t\in\R$ is {\em quasi-periodic in distribution},
if the mapping $\msL(X_\cdot):\R\rightarrow\msP(\mcX)$ is quasi-periodic.
\end{definition}

Now we can formulate our result, which is called the second Bogolyubov theorem.
\begin{theorem}\label{SAPth}
Let $B(x,y)\equiv B(y)$ and $g(x,y)\equiv g(y)$ for all $(x,y)\in\R^{d_1+d_2}$.
Assume that $\beta<\alpha$ or $\beta=\alpha$ and $\lambda_1>\frac{L_b^2}{\eta}$.
Furthermore, suppose that {\bf(H$_y^1$)}--{\bf(H$_y^6$)}, {\bf(H$_x^2$)}--{\bf(H$_x^3$)}
and {\bf(H$_x^6$)} hold.
If $f$ and $g$ are quasi-periodic, then for any $0<\varepsilon\leq 1$
there exists a unique solution $V^{\varepsilon}_t:=\left(X_t^\veps,Y_t^\veps\right), t\in\R$,
of \eqref{sfmain}, which is quasi-periodic in distribution,
and
\begin{equation}\label{SMSeq}
\lim_{\varepsilon\rightarrow0}\sup_{t\in\R}d_{BL}(\msL
(X^{\varepsilon}_t),\msL(\bar{X}_t))=0,
\end{equation}
where $\bar{X}$ is the unique stationary solution of
the averaged equation \eqref{avemain}.
\end{theorem}

\begin{remark}\label{MRS}\rm
(i) For brevity, we just illustrate the case of quasi-periodic solutions in this paper.
Indeed, our method applies to more general compact recurrent solutions.

(ii) Although there is a more general result on the second Bogolyubov theorem in \cite{CL2023},
which can cover unbounded recurrent solutions such as Levitan almost periodic solutions,
the proof presented here is more concise than \cite{CL2023}.
Furthermore, the system \eqref{sfmain} is more general, and the result \eqref{SMSeq} is stronger than \cite[Theorem 4.7]{CL2023}.
Since we employ the global averaging principle to establish the second Bogolyubov theorem,
it is required that the hull $\mcH(\mbF_\veps)$ is compact.
It is worth noting that $\mcH(\mbF_\veps)$ is compact provided $\mbF_\veps$ is Birkhoff recurrent,
which includes periodic, quasi-periodic, almost periodic, almost automorphic, and Birkhoff recurrent functions.

(iii) It is well-known that the uniform attractor (see \cite[Definition 5.6]{CL2023}) is a pullback attractor, but not vice versa. Compared to \cite{CL2023} and \cite{CLR2023}, we consider the more general pullback attractor instead of the uniform attractor.
\end{remark}

\subsection{Examples}
To illustrate our results, we will present two examples in this subsection.
For simplicity, we just consider the one-dimensional case, but one can easily
extend this to the multi-dimensional case. Let $W_t^1,t\in\R$ and $W_t^2,t\in\R$ be
independent two-sided standard Brownian motions.
\begin{Examples}\label{Ex1}\rm
Consider the following slow-fast SDEs:
\begin{equation*}\label{Ex1eq01}
\left\{
   \begin{aligned}
   &\ \d X^{\veps}_t=\left(X_t^\veps-(X_t^\veps)^3+(Y_t^\veps)^2\sin X_t^\veps+Y_t^\veps\right)\d t
   +\d W^{1}_t\\
   &\ \d Y^{\veps}_t=\frac1\veps\left(-(\sin X_t^\veps)^2(Y_t^\veps)^5-(Y_t^\veps)^3-Y_t^\veps\right)\d t
   +\frac{1}{\sqrt{\veps}}\d W^{2}_t.
   \end{aligned}
   \right.
\end{equation*}
We define
$f(x,y):=x-x^3+y^2\sin x +y$,
$B(x,y):=-(\sin x)^2y^5-y^3-y$,
$\forall(x,y)\in\R^{2}$.
It can be verified that $f$ and $B$ satisfy {\bf(H$_y^1$)}--{\bf(H$_y^5$)}
and {\bf(H$_x^1$)}--{\bf(H$_x^4$)}.
Then by Theorems \ref{Opcrth} and \ref{NDth}, one sees that
there exists a constant $C>0$ such that
\begin{equation*}
\mbE\left(\sup_{t\in[0,T]}|X_t^\veps-\bar{X}_t|^2\right)\leq C\veps,
\end{equation*}
where $\bar{X}_t$ is the solution to the corresponding equation,
and that
\begin{equation*}
Z_t^\veps:=\frac{X_t^\veps-\bar{X}_t}{\sqrt{\veps}}
\end{equation*}
weakly converges to an Ornstein-Uhlenbeck type process, as $\veps$ goes to zero.
\end{Examples}

\begin{Examples}\rm
Consider the following multi-scale SDEs:
\begin{equation}\label{Ex2eq01}
\left\{
   \begin{aligned}
   \d X^{\veps}_t=
   &\ \Big[-a_1 X_t^\veps-(X_t^\veps)^3
   +\left(a_2Y_t^\veps+a_3(Y_t^\veps)^3\right)\left(\cos\left(t/\sqrt{\veps}\right)
   +\sin\left(\sqrt{2}t/\sqrt{\veps}\right)\right)\\
   &\
   -a_4X_t^\veps(Y_t^\veps)^4\left(\sin(t/\sqrt{\veps})\right)^2\Big]\d t
   +\d W^{1}_t\\
   \d Y^{\veps}_t=
   &\ \left[\frac1\veps\left(-(Y_t^\veps)^3-Y_t^\veps\right)
   +\frac{1}{\sqrt[3]{\veps}}\left(X_t^\veps+Y_t^\veps\right)\right]\d t
   +\frac{1}{\sqrt{\veps}}\d W^{2}_t,
   \end{aligned}
   \right.
\end{equation}
where $a_1>0$, $a_2,a_3\in\R$ and $a_4\geq0$.
For all $(t,x,y)\in\R^{3}$, define
\[
f(t,x,y):=-a_1x-x^3+(a_2y+a_3y^3)(\cos t+\sin \sqrt{2}t)
-a_4xy^4(\sin t)^2,
\]
$B(y):=-y^3-y$ and $b(x,y):=x+y$.
We note that
$f$, $B$ and $b$ satisfy conditions {\bf(H$_y^1$)}--{\bf(H$_y^5$)},
{\bf(H$_x^2$)}--{\bf(H$_x^3$)} and {\bf(H$_x^5$)}.
Set
$$\mbF_\veps(t,x,y):=\left(f(\veps^{-1/2}t,x,y),
\veps^{-1}B(y)+\veps^{-1/3}b(x,y)\right)^T,
~\forall (t,x,y)\in\R^3.$$
Then by Theorem \ref{GAPth},
we have the following conclusions:
\begin{enumerate}
  \item For any $0<\veps\leq1$ \eqref{Ex2eq01} is associated to a skew product flow $(\sigma,P^*_\veps)$ on $\left(\msP_{2,6}(\R^{2}),d_{BL}\right)$,
  and $(\sigma,P^*_\veps)$ admits a pullback attractor $\msA^\veps$ with component subsets $\msA_{\widetilde{\mbF}_\veps},~\widetilde{\mbF}_\veps\in\mcH(\mbF_\veps)$;
  \item The corresponding averaged equation has a global attractor $\bar{\mcA}$;
  \item Furthermore, for all $\widetilde{\mbF}_\veps\in\mcH(\mbF_\veps)$
\[
\lim_{\veps\rightarrow0}{\rm dist}_{\msP(\R^{d_1})}\left(\Pi_1(\msA_{\widetilde{\mbF}_\veps}),\bar{\mcA}\right)=0.
\]
\end{enumerate}

Furthermore, assume that $a_4=0$. Then it can be verified that {\bf(H$_x^6$)} also holds.
Recall that $f$ is quasi-periodic.
Then in view of Theorem \ref{SAPth}, for all
$0<\veps\leq 1$ there is
a unique solution $(X_t^\veps,Y_t^\veps),t\in\R$, which is quasi-periodic in
distribution, and
\[
\lim_{\varepsilon\rightarrow0}\sup_{t\in\R}d_{BL}(\msL
(X^{\varepsilon}_t),\msL(\bar{X}_t))=0,
\]
where $\bar{X}$ is the stationary solution to the following
averaged equation:
\begin{equation*}
\d \bar{X}_t=\left(-a_1\bar{X}_t-(\bar{X}_t)^3\right)\d t
+\d W_t^1.
\end{equation*}
\end{Examples}

\section{Frozen equation and Poisson equation}\label{Peq}

Consider the following so-called frozen equation
\begin{equation*}
\d Y_s^{x,y}=B(x,Y_s^{x,y})\d s+g(x,Y_s^{x,y})\d W_s^2, \quad Y_0^{x,y}=y\in\R^{d_2},
\end{equation*}
where $x\in\R^{d_1}$ is a frozen parameter.

\begin{lemma}\label{StableSY}
Assume that {\bf{(H$_y^1$)}} and {\bf{(H$_y^3$)}} hold.
For any $\xi\in\mcL^2(\Omega,\msF_s,\mbP;\R^{d_2})$, $t\geq s$ and $x\in\R^{d_1}$,
let $Y_{s,t}^x(\xi),t\geq s$ be the unique solution to \eqref{fmain}. Then for any $\xi_1,\xi_2\in\mcL^2(\Omega,\msF_s,\mbP;\R^{d_2})$ we have
\begin{equation}\label{StableSY00}
\mbE\left|Y_{s,t}^x(\xi_1)-Y_{s,t}^x(\xi_2)\right|^2\leq
\mbE|\xi_1-\xi_2|^2{\rm e}^{-\eta(t-s)}.
\end{equation}
Moreover, if $\xi\in\mcL^{2p}(\Omega,\msF_s,\mbP;\R^{d_2})$ and
$g\in C_b(\R^{d_1+d_2})$ for any $p\geq1$, then we have
\begin{equation}\label{StableSY01}
\mbE\left|Y_{s,t}^x(\xi)\right|^{2p}\leq \mbE|\xi|^{2p}{\rm e}^{-\frac{\eta p}{2}(t-s)}+M,
\end{equation}
and there exists a constant $C>0$ such that
for all $x_1,x_2\in\R^{d_1}$
\begin{equation*}\label{StableSY02}
\mbE\left|Y_{s,t}^{x_1}(y)-Y_{s,t}^{x_2}(y)\right|^2
\leq C\left(1+|y|^{2\kappa_2}\right)|x_1-x_2|^2.
\end{equation*}
Here $M$ depends on $p,K_1,\|g\|_\infty$.
\end{lemma}
\begin{proof}
(i) Employing It\^o's formula and {\bf (H$_y^3$)}--(i), we have
\begin{align*}
&
\mbE\left|Y_{s,t}^x(\xi_1)-Y_{s,t}^x(\xi_2)\right|^2\\
&
=\mbE|\xi_1-\xi_2|^2+\mbE\int_0^t2\langle B(x,Y_{s,r}^x(\xi_1))-B(x,Y_{s,r}^x(\xi_2)),Y_{s,r}^x(\xi_1)-Y_{s,r}^x(\xi_2)\rangle \d r\\
&\quad
+\mbE\int_0^t|g(x,Y_{s,r}^x(\xi_1))-g(x,Y_{s,r}^x(\xi_2))|^2_{HS}\d r\\
&
\leq \mbE|\xi_1-\xi_2|^2+\mbE\int_0^t-\eta\left|Y_{s,r}^x(\xi_1)-Y_{s,r}^x(\xi_2)\right|^2\d r,
\end{align*}
which by Gronwall's inequality implies that
\begin{equation*}
\mbE\left|Y_{s,t}^x(\xi_1)-Y_{s,t}^x(\xi_2)\right|^2\leq \mbE|\xi_1-\xi_2|^2{\rm e}^{-\eta(t-s)}.
\end{equation*}

(ii) By It\^o's formula, {\bf{(H$_y^1$)}} and Young's inequality, one sees that for any $p\geq1$
\begin{align*}
&
\mbE\left|Y_{s,t}^x(\xi)\right|^{2p}\\
&
=\mbE|\xi|^{2p}+p\mbE\int_0^t|Y_{s,r}^x(\xi)|^{2p-2}\left(2\left\langle B(x,Y_{s,r}^x(\xi)), Y_{s,r}^x(\xi)\right\rangle
+\left|g(x,Y_{s,r}^x(\xi))\right|_{HS}^2\right)\d r\\
&\quad
+2p(p-1)\mbE\int_0^t|Y_{s,r}^x(\xi)|^{2p-4}|g^T(x,Y_{s,r}^x(\xi))Y_{s,r}^x(\xi)|^2\d r\\
&
\leq \mbE|\xi|^{2p}+\mbE\int_0^t p\left[\left(-\eta|Y_{s,r}^x(\xi)|^2+K_1\right)|Y_{s,r}^x(\xi)|^{2p-2}
+2(p-1)\|g\|_\infty|Y_{s,r}^x(\xi)|^{2p-2}\right]\d r\\
&
\leq \mbE|\xi|^{2p}+\mbE\int_0^t-\frac{\eta p}{2}|Y_{s,r}^x(\xi)|^{2p}+p^{-p}(p-1)^{p-1}\left(\eta/2\right)^{1-p}
\left(K_1p+2(p-1)\|g\|_\infty\right)^p\d r.
\end{align*}
Then we have
\begin{equation*}
\mbE\left|Y_{s,t}^x(\xi)\right|^{2p}\leq \mbE|\xi|^{2p}{\rm e}^{-\frac{\eta p}{2}(t-s)}
+2p^{-p}(p-1)^{p-1}\left(\eta/2\right)^{1-p}
\left(K_1p+2(p-1)\|g\|_\infty\right)^p(\eta p)^{-1}.
\end{equation*}

(iii) In view of It\^o's formula, we get
\begin{align*}
&
\mbE\left|Y_{s,t}^{x_1}(y)-Y_{s,t}^{x_2}(y)\right|^2\\
&
=\mbE\int_s^t2\left\langle B(x_1,Y_{s,r}^{x_1}(y))-B(x_2,Y_{s,r}^{x_2}(y)),Y_{s,r}^{x_1}(y)-Y_{s,r}^{x_2}(y)\right\rangle\d r\\
&\quad
+\mbE\int_s^t\left| g(x_1,Y_{s,r}^{x_1}(y))-g(x_2,Y_{s,r}^{x_2}(y))\right|_{HS}^2\d r\\
&
=\mbE\int_s^t2\left\langle B(x_1,Y_{s,r}^{x_1}(y))-B(x_1,Y_{s,r}^{x_2}(y)),Y_{s,r}^{x_1}(y)-Y_{s,r}^{x_2}(y)\right\rangle\d r\\
&\quad
+\mbE\int_s^t2\left\langle B(x_1,Y_{s,r}^{x_2}(y))-B(x_2,Y_{s,r}^{x_2}(y)),Y_{s,r}^{x_1}(y)-Y_{s,r}^{x_2}(y)\right\rangle\d r\\
&\quad
+\mbE\int_s^t\left(\left| g(x_1,Y_{s,r}^{x_1}(y))-g(x_1,Y_{s,r}^{x_2}(y))\right|_{HS}^2+\left| g(x_1,Y_{s,r}^{x_2}(y))-g(x_2,Y_{s,r}^{x_2}(y))\right|_{HS}^2\right)\d r\\
&\quad
+\mbE\int_s^t2\left\langle g(x_1,Y_{s,r}^{x_1}(y))-g(x_1,Y_{s,r}^{x_2}(y)), g(x_1,Y_{s,r}^{x_2}(y))-g(x_2,Y_{s,r}^{x_2}(y))\right\rangle_{HS}\d r.
\end{align*}
Then thanks to {\bf{(H$_y^3$)}}, Young's inequality and \eqref{StableSY01}, we obtain
\begin{align*}
&
\mbE\left|Y_{s,t}^{x_1}(y)-Y_{s,t}^{x_2}(y)\right|^2\\
&
\leq \mbE\int_s^t\Big(-\eta\left|Y_{s,r}^{x_1}(y)-Y_{s,r}^{x_2}(y)\right|^2
+L_g^2|x_1-x_2|^2
+2L_g^2\left|Y_{s,r}^{x_1}(y)-Y_{s,r}^{x_2}(y)\right||x_1-x_2|
\\
&\qquad
+2K_3\left(1+|Y_{s,r}^{x_2}(y)|^{\kappa_2}\right)
|x_1-x_2|\left|Y_{s,r}^{x_1}(y)-Y_{s,r}^{x_2}(y)\right|\Big)\d r\\
&
\leq \mbE\int_s^t\left(-\frac{\eta}{2}\left|Y_{s,r}^{x_1}(y)-Y_{s,r}^{x_2}(y)\right|^2
+\left(\frac{4}{\eta}L_g^4
+L_g^2\right)|x_1-x_2|^2\right)\d r\\
&\quad
+\int_s^t\frac{C}{\eta}\left(1+\mbE|Y_{s,r}^{x_2}(y)|^{2\kappa_2}\right)|x_1-x_2|^2\d r\\
&
\leq \mbE\int_s^t\left(-\frac{\eta}{2}\left|Y_{s,r}^{x_1}(y)-Y_{s,r}^{x_2}(y)\right|^2
+\left(\frac{4}{\eta}L_g^4
+L_g^2\right)|x_1-x_2|^2\right)\d r\\
&\quad
+\int_s^t C\left(1+|y|^{2\kappa_2}{\rm e}^{-\frac{\eta \kappa_2}{2}(r-s)}\right)|x_1-x_2|^2\d r.
\end{align*}
Therefore, by Gronwall's inequality, we obtain
\begin{equation*}
\mbE\left|Y_{s,t}^{x_1}(y)-Y_{s,t}^{x_2}(y)\right|^2
\leq C\left(1+|y|^{2\kappa_2}\right)|x_1-x_2|^2.
\end{equation*}
\end{proof}

Combining \eqref{StableSY00} and \eqref{StableSY01}, we have the following result.
\begin{corollary}\label{EIMeq}
Assume that {\bf(H$_y^1$)} and {\bf(H$_y^3$)} hold.
If $g\in C_b(\R^{d_1+d_2})$, then for any $m>0$
\begin{equation*}
\sup_{x\in\R^{d_1}}\int_{\R^{d_2}}|y|^m\mu^x(\d y)<M.
\end{equation*}
\end{corollary}

Now we show the continuous dependence on the parameter $x\in\R^{d_1}$ for stationary solutions to equation \eqref{fmain}.
\begin{lemma}\label{SSCx}
Assume that {\bf{(H$_y^1$)}} and {\bf{(H$_y^3$)}} hold. If $g\in C_b(\R^{d_1+d_2})$, then for any $x\in\R^{d_1}$ there exist a unique stationary solution $Y_t^x,~t\in\R$, to \eqref{fmain}, and a constant $C>0$ such that for any $x_1,x_2\in \R^{d_1}$
\begin{equation*}
\sup_{t\in\R}\mbE|Y_t^{x_1}-Y_t^{x_2}|^2\leq C|x_1-x_2|^2.
\end{equation*}
\end{lemma}
\begin{proof}
It follows from \eqref{StableSY00} and \eqref{StableSY01} that for any $x\in\R^{d_1}$ there exists a unique stationary solution $Y_t^x,~t\in\R$, to \eqref{fmain}.
By Corollary \ref{EIMeq}, we have for all $-n\leq t$ and $x_1,x_2\in\R^{d_1}$
\begin{align*}
\mbE|Y_t^{x_1}-Y_t^{x_2}|^2
&
\leq 3\mbE|Y_t^{x_1}-Y_{-n,t}^{x_1}(0)|^2
+3\mbE|Y_{-n,t}^{x_1}(0)-Y_{-n,t}^{x_2}(0)|^2
+3\mbE|Y_{-n,t}^{x_2}(0)-Y_t^{x_2}|^2\\
&
\leq 6M{\rm e}^{-\eta(t+n)}+3C|x_1-x_2|^2,
\end{align*}
which implies that
\begin{equation*}
\mbE|Y_t^{x_1}-Y_t^{x_2}|^2\leq C|x_1-x_2|^2
\end{equation*}
by letting $n\rightarrow\infty$, where $C$ does not depend on $t$.
\end{proof}

Finally, we investigate the well-posedness of the Poisson equation.
Consider the following equation
\begin{equation}\label{Peqmian}
\mfL_2(x,y)u(h,x,y)=-\left(\phi(h,x,y)-\bar{\phi}(h,x)\right), \quad y\in\R^{d_2},
\end{equation}
where $(h,x)\in\R^{d+d_1}$ is a parameter,
$\bar{\phi}(h,x):=\int_{\R^{d_2}}\phi(h,x,y)\mu^x(\d y)$
and
$$
\mfL_2 u(h,x,y):=\mfL_2(x,y) u(h,x,y)
:=\langle B(x,y),\partial_y u(h,x,y)\rangle
+\sum_{i,j=1}^{d_2}a_{ij}\partial_{y_jy_i}^2u(h,x,y).
$$
Here $(a_{ij})=gg^T/2$. Similarly, for all $(h,x,y)\in\R^{d+d_1+d_2}$ we define
$$\mfL_1^\var\psi(h,x,y):=\mfL_1^\var(x,y)\psi(h,x,y):=\langle f(\var^{-\gamma}t,x,y),\partial_x\psi(h,x,y)\rangle
+\sum_{i,j=1}^{d_1} A_{ij}^\veps\partial_{x_jx_i}^2\psi(h,x,y),$$
$$\mfL_3\psi(h,x,y):=\mfL_3(x,y)\psi(h,x,y):=\langle b(x,y),\partial_y\psi(h,x,y)\rangle,$$
where $\left(A_{ij}^\veps\right)=\sigma_\veps\sigma_\veps^T/2$.
When $f$ and $\sigma$ are time independent, we let
$\left(A_{ij}\right):=\sigma\sigma^T/2$, and
$$\mfL_1\psi(h,x,y):=\mfL_1(x,y)\psi(h,x,y):=\langle f(x,y),\partial_x\psi(h,x,y)\rangle
+\sum_{i,j=1}^{d_1} A_{ij}\partial_{x_jx_i}^2\psi(h,x,y),$$
$$
\mfL_{\bar{x}}\psi(x,y):=\mfL_{\bar{x}}\psi(x,y):=\langle \bar{f}(x),\partial_x\psi(x,y)\rangle
+\sum_{i,j=1}^{d_1}A_{ij}\partial_{x_jx_i}^2\psi(x,y).
$$

Let us first introduce the condition {\bf(H$_\phi^i$)}, where $i=1,2$.
\begin{enumerate}
  \item [{\bf(H$_\phi^i$)}]
  There exist constants $C_1>0$ and $m_1,m_2,m_3\geq0$ such that for all $(h,x,y)\in\R^{d+d_1+d_2}$
  \begin{align*}
  \sum_{0\leq2k+j\leq 4}|\partial_x^k\partial_y^j\phi(h,x,y)|+|\partial_h^i\phi(h,x,y)|
  \leq C_1(1+|h|^{m_1}+|x|^{m_2}+|y|^{m_3}).
  \end{align*}
\end{enumerate}

Similar to \cite[Proposition 4.1]{SSWX2022} and \cite[Theorem 3.1]{CDGOS2022}, we have the following
lemma about the existence and uniqueness of solutions to \eqref{Peqmian}.
\begin{lemma}\label{SolPoisson}
Assume that {\bf(H$_y^1$)}--{\bf(H$_y^5$)} hold.
Furthermore, suppose that  $\phi\in C^{i,2,4}(\R^{d+d_1+d_2})$ satisfies {\bf(H$_\phi^i$)}.
Then there exist a unique solution $u(t,x,\cdot)\in C^{2}(\R^{d_2})$ to \eqref{Peqmian}
and constants $m'_1,m'_2,m_3',C>0$ such that
\begin{align*}
&
|u(h,x,y)|+|\partial_yu(h,x,y)|+\sum_{j=1}^i|\partial^j_hu(h,x,y)|\\\nonumber
&
+|\partial_xu(h,x,y)|+|\partial_x^2u(h,x,y)|
\leq C(1+|h|^{m'_1}+|x|^{m'_2}+|y|^{m'_3}),
\end{align*}
and
\begin{align*}
\sum_{j=1}^2\left|\partial_x^j\bar{\phi}(h,x)\right|
\leq C\left(1+|h|^{m_1'}+|x|^{m_2'}\right).
\end{align*}
\end{lemma}

\section{The first Bogolyubov theorem}

In Section 4.1, we will prove some lemmas, which give the properties for  $\hat{f},\bar{f}$ and $\bar{\sigma}$, moment estimates of solutions to \eqref{sfmain} and \eqref{avemain}, and H\"older continuity of the slow variable in \eqref{sfmain}.
In Section 4.2, we prove Theorem \ref{FAPth}.

\subsection{Auxiliary lemmas}

In the following lemma, we show that $\hat{f},\bar{f}$ and $\bar{\sigma}$
inherit some properties from $f$ and $\sigma$.
\begin{lemma}\label{aveCCL}
Assume that $f\in C^{1,2,4}(\R^{1+d_1+d_2})$ and $\sigma$ satisfy {\bf(H$_x^1$)}--{\bf(H$_x^3$)}, {\bf(A$_f$)} and {\bf(A$_\sigma$)}.
Furthermore, suppose that {\bf(H$_y^1$)} and {\bf(H$_y^3$)} hold. Then the following conclusions hold.
\begin{enumerate}
  \item If $g\in C_b(\R^{d_1+d_2})$ then $\hat{f}$ and $\bar{f}$ satisfy {\bf(H$_x^1$)},
      $\bar{\sigma}$ satisfies {\bf(H$_x^3$)}
      and there exists $C>0$ such that for all $t\in\R$ and $x_1,x_2\in\R^{d_1}$
\begin{equation*}
\left|\hat{f}(t,x_1)-\hat{f}(t,x_2)\right|
+\left|\bar{f}(x_1)-\bar{f}(x_2)\right|
\leq C\left(1+|x_1|^{\theta_1}+|x_2|^{\theta_1}\right)|x_1-x_2|,
\end{equation*}
where $\theta_1$ is as in {\bf(H$_x^2$)}.
  \item If {\bf(H$_y^2$)} and {\bf(H$_y^4$)}--{\bf(H$_y^5$)} hold then $\hat{f}\in C^{2}(\R^{d_1})$ and $\hat{f}$ satisfies {\bf(H$_x^2$)}.
\end{enumerate}
\end{lemma}
\begin{proof}
(i)
We note that
\begin{align*}
\hat{f}(t,x)=\int_{\R^{d_2}}f(t,x,y)\mu^x(\d y),
~\forall (t,x)\in\R^{1+d_1}.
\end{align*}
In view of {\bf(H$_x^1$)} and Corollary \ref{EIMeq}, we have for any $(t,x)\in\R^{1+d_1}$
\begin{align*}
\langle \hat{f}(t,x),x\rangle
&
=\left\langle \int_{\R^{d_2}}f(t,x,y)\mu^x(\d y),x\right\rangle\\
&
\leq\int_{\R^{d_2}}\left(K_4(1+|x|^2)+K_5|y|^{\theta}\right)\mu^x(\d y)
\leq K_4(1+|x|^2)+C.
\end{align*}
Therefore, for any $t\in\R$ and $x\in\R^{d_1}$, one sees that
\begin{align*}
\langle \bar{f}(x),x\rangle
&
=\left\langle \bar{f}(x)-\frac{1}{T}\int_0^T\hat{f}(s,x)\d s,x\right\rangle
+\left\langle\frac1T\int_{0}^{T}\hat{f}(s,x)\d s,x\right\rangle\\
&
\leq \left|\bar{f}(x)-\frac1T\int_{0}^{T}\hat{f}(s,x)\d s\right||x|+K_4\left(1+|x|^2\right)+C,
\end{align*}
which implies that
\begin{equation*}
\langle \bar{f}(x),x\rangle\leq K_4\left(1+|x|^2\right)+C,
~\forall x\in\R^{d_1}
\end{equation*}
by letting $T\rightarrow\infty$ and because of Remark \ref{Rem01} (iii).

Combing {\bf(H$_x^2$)}, Corollary \ref{EIMeq},
H\"older's inequality and Lemma \ref{SSCx},
we have for any $t\in\R$ and $x_1,x_2\in\R^{d_1}$
\begin{align}
&
\left|\hat{f}(t,x_1)-\hat{f}(t,x_2)\right|\nonumber\\\nonumber
&
=\left|\int_{\R^{d_2}}f(t,x_1,y)\mu^{x_1}(\d y)
-\int_{\R^{d_2}}f(t,x_2,y)\mu^{x_2}(\d y)\right|\\\nonumber
&
\leq\int_{\R^{d_2}}\left|f(t,x_1,y)-f(t,x_2,y)\right|\mu^{x_1}(\d y)
+\left|\int_{\R^{d_2}}f(t,x_2,y)\left(\mu^{x_1}(\d y)-\mu^{x_2}(\d y)\right)\right|
\\
&
\leq C\left(1+|x_1|^{\theta_1}+|x_2|^{\theta_1}\right)|x_1-x_2|\label{aveCCL0808}\\\nonumber
&\quad
+C\mbE\left((1+|x_2|^{\theta_1}+|Y_s^{x_1}|^{\theta_2}
+|Y_s^{x_2}|^{\theta_2})
|Y_s^{x_1}-Y_s^{x_2}|\right)\\\nonumber
&
\leq C\left(1+|x_1|^{\theta_1}+|x_2|^{\theta_1}\right)|x_1-x_2|,
\end{align}
where $s\in\R$, $Y_\cdot^{x_1}$ and $Y_\cdot^{x_2}$ are stationary solutions to
\eqref{fmain} with $x_1$ and $x_2$ respectively  replacing $x$.
Then \eqref{aveCCL0808} and {\bf(A$_f$)} imply that for all $x_1,x_2\in\R^{d_1}$
\begin{equation*}
|\bar{f}(x_1)-\bar{f}(x_2)|\leq
C\left(1+|x_1|^{\theta_1}+|x_2|^{\theta_1}\right)|x_1-x_2|.
\end{equation*}

It follows from {\bf(H$_x^3$)} and {\bf(A$_\sigma$)} that $\bar{\sigma}$ satisfies {\bf(H$_x^3$)}.

(ii)
By {\bf(H$_x^2$)} and Lemma \ref{SolPoisson}, one sees that there exists $m_1>0$ such that
\begin{equation*}\label{aveCCL08}
|\partial^i_x\hat{f}(t,x)|\leq K_6\left(1+|x|^{m_1}+C\right), \quad i=0,1,2.
\end{equation*}
It follows from {\bf(H$_x^2$)} that for all $(t,x)\in\R^{1+d_1}$
\begin{align*}
\left|\partial_t\hat{f}(t,x)\right|
&
\leq \int_{\R^{d_2}}\left|\partial_t f(t,x,y)\right|\mu^x(\d y)
\leq K_6\left(1+|x|^{\theta_1}\right)+C.
\end{align*}
\end{proof}

\begin{remark}\label{EUASDE}\rm
Assume that {\bf(H$_y^1$)}, {\bf(H$_y^3$)},
{\bf(H$_x^1$)}--{\bf(H$_x^3$)}, {\bf(A$_f$)} and {\bf(A$_\sigma$)} hold. If $g\in C_b(\R^{d_1+d_2})$,
then it follows from the above lemma and Theorem 3.1.1 in \cite{LR2015}
that for any $\zeta\in\mcL^2(\Omega,\msF_0,\mbP;\R^{d_1})$ there exists a unique
solution $\bar{X}_t(\zeta)$ to \eqref{avemain} satisfying $\bar{X}_0(\zeta)=\zeta$.
\end{remark}

Now we prove moment estimates for solutions to the slow-fast SDEs \eqref{sfmain} and the averaged equation \eqref{avemain}.
\begin{lemma}\label{Pest}
Assume that {\bf(H$_y^1$)}, {\bf(H$_y^3$)} and {\bf(H$_x^1$)}--{\bf(H$_x^3$)} hold. Let $(X_t^\veps(\zeta^\veps),Y_t^\veps(\xi^\veps)),t\geq0$,
be the solution to \eqref{sfmain}, and $\bar{X}_t(\zeta),t\geq0$, the solution to \eqref{avemain}.
If $g\in C_b(\R^{d_1+d_2})$, then there exists a constant $0<\veps_0\leq 1$ such that for all $p\geq1$, $T>0$ and $0<\veps\leq\veps_0$
\begin{equation}\label{PestY}
\mbE\left(\sup_{t\in[0,T]}|Y_t^\veps(\xi^\veps)|^{2p}\right)
\leq C_{p,T}\left(1+\mbE|\xi^\veps|^{2p}\right),
\end{equation}
\begin{equation}\label{PestX}
\mbE\left(\sup_{t\in[0,T]}|X_t^\veps(\zeta^{\veps})|^{2p}\right)\leq C_{p,T}\left(1+\mbE|\zeta^\veps|^{2p}+\mbE|\xi^\veps|^{\theta p}\right)
\end{equation}
and
\begin{equation}\label{PestAX}
\mbE\left(\sup_{t\in[0,T]}|\bar{X}_t(\zeta)|^{2p}\right)\leq C_{p,T}\left(1+\mbE|\zeta|^{2p}\right),
\end{equation}
where $\theta$ is as in {\bf(H$_x^1$)},
and $C_{p,T}$ is independent of $\veps$.

\end{lemma}
\begin{proof}
Employing It\^o's formula and {\bf(H$_y^1$)}, we have
\begin{align*}
|Y_t^\veps(\xi^\veps)|^{2p}
&
=|\xi^\veps|^{2p}+p\int_0^t|Y_s^\veps(\xi^\veps)|^{2p-2}
\Bigg(2\langle \veps^{-2\alpha}B(X_s^\veps(\zeta^\veps),Y_s^\veps(\xi^\veps)),Y_s^\veps(\xi^\veps)\rangle\\
&\qquad
+2\langle\veps^{-\beta}b(X_s^\veps(\zeta^\veps),Y_s^\veps(\xi^\veps)),Y_s^\veps(\xi^\veps)\rangle
+\veps^{-2\alpha}|g(X_s^\veps(\zeta^\veps),Y_s^\veps(\xi^\veps))|_{HS}^2\Bigg)\d s\\
&\quad
+2p(p-1)\veps^{-2\alpha}\int_0^t|Y_s^\veps(\xi^\veps)|^{2p-4}|g^T(X_s^\veps(\zeta^\veps),Y_s^\veps(\xi^\veps))Y_s^\veps(\xi^\veps)|^2\d s\\
&\quad
+2p\veps^{-\alpha}\int_0^t|Y_s^\veps(\xi^\veps)|^{2p-2}\langle Y_s^\veps(\xi^\veps),g(X_s^\veps(\zeta^\veps),Y_s^\veps(\xi^\veps))\d W_s^2\rangle\\
&
\leq|\xi^\veps|^{2p}+p\int_0^t-\veps^{-2\alpha}\left(\eta-\veps^{2\alpha-\beta}\widetilde{\eta}\right)|Y_s^\veps(\xi^\veps)|^{2p}
\d s\\
&\quad
+p\int_0^t\left(\veps^{-2\alpha}K_1+\veps^{-\beta}K_1+2(p-1)
\veps^{-2\alpha}\|g\|_\infty\right)|Y_s^\veps(\xi^\veps)|^{2p-2}\d s\\
&\quad
+2p\veps^{-\alpha}\int_0^t|Y_s^\veps(\xi^\veps)|^{2p-2}\langle Y_s^\veps(\xi^\veps),g(X_s^\veps(\zeta^\veps),Y_s^\veps(\xi^\veps))\d W_s^2\rangle.
\end{align*}
Note that there exists a constant $0<\veps_0\leq 1$ such that $\eta-\veps^{2\alpha-\beta}\widetilde{\eta}>\frac\eta2$ for all $0<\veps\leq\veps_0$. Then by Young's inequality,
one sees that for any $0<\veps\leq\veps_0$
\begin{align}
|Y_t^\veps(\xi^\veps)|^{2p}
&
\leq|\xi^\veps|^{2p}+p\int_0^t-\frac14\veps^{-2\alpha}\eta|Y_s^\veps(\xi^\veps)|^{2p}\d s\nonumber\\
&\quad
+2p\veps^{-\alpha}\int_0^t|Y_s^\veps(\xi^\veps)|^{2p-2}\langle Y_s^\veps(\xi^\veps),g(X_s^\veps(\zeta^\veps),Y_s^\veps(\xi^\veps))\d W_s^2\rangle \label{PestYB}\\\nonumber
&\quad
+\int_0^t\left(\frac{p}{4(p-1)}\veps^{-2\alpha}\eta\right)^{1-p}\left(\veps^{-2\alpha}K_1+\veps^{-\beta}K_1+2(p-1)\veps^{-2\alpha}\|g\|_\infty\right)^p\d s,
\end{align}
which implies that for any stopping time $\tau\leq T$
\begin{align}\label{qqqq}
\mbE|Y_\tau^\veps(\xi^\veps)|^{2p}
&
\leq \mbE\left(|\xi^\veps|^{2p}\e^{-\frac14\veps^{-2\alpha}p\eta\tau}\right)
+C_{p,T,\eta,\|g\|_\infty}
\leq C_{p,T}\left(1+\mbE|\xi^\veps|^{2p}\right).
\end{align}

Define $\tau_r^\veps:=\inf\{t\geq0:|Y_t^\veps(\xi^\veps)|^{2p}>r\}\wedge T$.
In view of \eqref{qqqq}, for any $0<\alpha<1$ we have
\begin{align*}
\mbE\left(\sup_{t\in[0,T]}|Y_t^\veps(\xi^\veps)|^{2p}\right)^{\alpha}
&
=\alpha\int_0^\infty r^{\alpha-1}\mbP\left(\sup_{t\in[0,T]}
|Y_t^\veps(\xi^\veps)|^{2p}>r\right)\d r\\
&
\leq\alpha\int_0^\infty r^{\alpha-1}\left(1\wedge r^{-1}\mbE|Y_{\tau_r^\veps}^\veps(\xi^\veps)|^{2p}\right)\d r\\
&
\leq\alpha\int_0^\infty r^{\alpha-1}\left(1\wedge r^{-1}C_{p,T}\left(1+\mbE|\xi^\veps|^{2p}\right)\right)\d r\\
&
\leq \alpha C_{p,T}\left(1+\mbE|\xi^\veps|^{2p}\right)^{\alpha}\int_0^\infty
\lambda^{\alpha-1}\left(1\wedge\lambda^{-1}\right)\d\lambda\\
&
\leq C_{p,T}\left(1+\mbE|\xi^\veps|^{2p}\right)^\alpha
\end{align*}
by the change of variables
$r\mapsto \left[C_{p,T}\left(1+\mbE|\xi^\veps|^{2p}\right)\right]\lambda$.

Note that
\begin{align*}
X_t^\veps(\zeta^\veps)=\zeta^\veps+\int_0^tf_\veps(s,X_s^\veps(\zeta^\veps),Y_s^\veps(\xi^\veps))\d s
+\int_0^t\sigma_\veps(s,X_s^\veps(\zeta^\veps))\d W_s^1.
\end{align*}
It follows from It\^o's formula and {\bf(H$_x^1$)} that
\begin{align*}
|X_t^\veps(\zeta^\veps)|^{2p}
&
\leq|\zeta^\veps|^{2p}+p\int_0^t|X_s^\veps(\zeta^\veps)|^{2p-2}
\left(K_4(1+|X_s^\veps(\zeta^\veps)|^2)+K_5|Y_s^\veps(\xi)|^{\theta}\right)\d s\\
&\quad
+2p\int_0^t|X_s^\veps(\zeta^\veps)|^{2p-2}\langle X_s^\veps(\zeta^\veps),\sigma_\veps(s,X_s^\veps(\zeta^\veps))\d W_s^1\rangle\\
&\quad
+2p(p-1)\int_0^t|X_s^\veps(\zeta^\veps)|^{2p-2}|\sigma_\veps(s,X_s^\veps(\zeta^\veps))|^2\d s.
\end{align*}
Then by Burkholder-Davis-Gundy's inequality and Young's inequality, one sees that
\begin{align*}
&
\mbE\left(\sup_{t\in[0,T]}|X_t^\veps(\zeta^\veps)|^{2p}\right)\\\nonumber
&
\leq\mbE|\zeta^\veps|^{2p}+p\mbE\int_0^T|X_s^\veps(\zeta^\veps)|^{2p-2}
\left(K_4(1+|X_s^\veps(\zeta^\veps)|^2)+K_5|Y_s^\veps(\xi)|^{\theta}\right)\d s\\\nonumber
&\quad
+\frac12\mbE\left(\sup_{t\in[0,T]}|X_t^\veps(\zeta^\veps)|^{2p}\right)
+C_T\mbE\int_0^T\left(|X_s^\veps(\zeta^\veps)|^{2p}+1\right)\d s\\\nonumber
&\quad
+2p(p-1)\mbE\int_0^T|X_s^\veps(\zeta^\veps)|^{2p-2}\left(2L_\sigma^2|X_s^\veps(\zeta^\veps)|^2+C\right)\d s,
\end{align*}
which by \eqref{PestY} and Gronwall's inequality implies that
\begin{equation*}
\mbE\left(\sup_{t\in[0,T]}|X_t^\veps(\zeta^\veps)|^{2p}\right)
\leq C_{p,T}\left(1+\mbE|\zeta^\veps|^{2p}+\mbE|\xi|^{\theta p}\right).
\end{equation*}

Similarly, we have
\begin{equation*}
\mbE\left(\sup_{t\in[0,T]}|\bar{X}_t(\zeta)|^{2p}\right)\leq C_{p,T}\left(1+\mbE|\zeta|^{2p}\right).
\end{equation*}
\end{proof}

Finally, we prove the H\"older continuity of $X_t^\veps,t\geq0$.
\begin{lemma}\label{timecon}
Assume that {\bf(H$_y^1$)}, {\bf(H$_y^3$)} and {\bf(H$_x^1$)}--{\bf(H$_x^3$)} hold. If $g\in C_b(\R^{d_1+d_2})$, then there exists a constant $C>0$ such that for all $0\leq s\leq t\leq T$
\begin{equation}\label{timeconX}
\mbE|X_{t}^\veps(x)-X_s^\veps(x)|^2\leq C_{T,|x|,|y|}|t-s|.
\end{equation}
\end{lemma}
\begin{proof}
It follows from It\^o's formula, the Burkholder-Davis-Gundy inequality,  {\bf(H$_x^2$)}, {\bf(H$_x^3$)}, \eqref{PestY} and \eqref{PestX} that
\begin{align*}
&
\mbE|X_{t}^\veps(x)-X_s^\veps(x)|^2\\
&
\leq2\mbE\left|\int_s^{t}f_\veps(r,X_r^\veps(x),Y_r^\veps(y))\d r\right|^2
+2\mbE\left|\int_s^{t}\sigma_\veps(r,X_r^\veps(x))\d W_r^1\right|^2\\
&
\leq|t-s|
\int_s^{t}\mbE\left|f_\veps(r,X_r^\veps(x),Y_r^\veps(y))\right|^2\d r
+4\mbE\int_s^{t}\left|\sigma_\veps(r,X_r^\veps(x))\right|_{HS}^2\d r\\
&
\leq |t-s|C
\int_s^{t}\mbE\left(1+|X_r^\veps(x)|^{2\theta_1}+|Y_r^\veps(y)|^{2\theta_2}\right)\d r
+C\mbE\int_s^{t}\left(1+|X_r^\veps(x)|^2\right)\d r\\
&
\leq C_{T,|x|,|y|}|t-s|.
\end{align*}
\end{proof}

\subsection{Proof of Theorem \ref{FAPth}}
Now we are in a position to prove Theorem \ref{FAPth}.
\begin{proof}
For brevity, we define $X_t^\veps:=X_t^\veps(x)$, $Y_t^\veps:=Y_t^\veps(y)$ and $\bar{X}_t:=\bar{X}_t(x)$ for all $t\geq0$ in this proof.
By It\^o's formula and the Burkholder-Davis-Gundy inequality, we have
\begin{align}
&
\mbE|X_t^\var-\bar{X}_t|^2\nonumber\\
&
\leq2\mbE\left|\int_0^t\left(f_\var(s,X_s^\var,Y_s^\var)
-\bar{f}(\bar{X}_s)\right)\d s\right|^2
+2\mbE\left|\int_0^t\left(\sigma_\veps(s,X_s^\var)-\bar{\sigma}(\bar{X}_s)\right)\d W^1_s\right|^2\nonumber\\
&
\leq4\mbE\left|\int_0^t\left(f_\var(s,X_s^\var,Y_s^\var)-\hat{f}_\veps(s,X_s^\veps)\right)\d s\right|^2
+4\mbE\left|\int_0^t\left(\hat{f}_\veps(s,X_s^\veps)-\bar{f}(\bar{X}_s)\right)\d s\right|^2\label{FAPtheq01}\\\nonumber
&\quad
+4\mbE\int_0^t|\sigma_\veps(s,X_s^\var)-\bar{\sigma}(\bar{X}_s)|_{HS}^2\d s
=:4\mfI_1(t,\var)+4\mfI_2(t,\var)
+4\mfI_3(t,\var).
\end{align}

First of all, we estimate $\mfI_1(t,\var)$.
Set $\phi(t,x,y):=f(t,x,y)-\hat{f}(t,x)$. Let $\psi$
be the solution to the following Poisson equation
\[
\mfL_2\psi(t,x,y)=\phi(t,x,y),\quad y\in\R^{d_2},
\]
where $t\in\R$ and $x\in\R^{d_1}$ are parameters.
By Lemma \ref{SolPoisson}, one sees that $\psi\in C^{1,2,2}(\R^{1+d_1+d_2})$.
Then according to It\^o's formula, we have for any $t>0$
\begin{align*}
&
\psi(\veps^{-\gamma}t,X_t^\var,Y_t^\var)\\
&
=\psi(0,x,y)+\int_0^t(\veps^{-\gamma}\partial_s+\mfL_1^\var)\psi(\veps^{-\gamma}s,X_s^\var,Y_s^\var)\d s
+\int_0^t\var^{-2\alpha}\mfL_2\psi(\veps^{-\gamma}s,X_s^\var,Y_s^\var)
\d s\\
&\quad
+\int_0^t\veps^{-\beta}\mfL_3\psi(\veps^{-\gamma}s,X_s^\var,Y_s^\var)\d s
+\int_0^t \sigma(\veps^{-\gamma}s,X_s^\var)\cdot\partial_x\psi(\veps^{-\gamma}s,X_s^\var,Y_s^\var)\d W^1_s\\
&\quad
+\var^{-\alpha}\int_0^t g(X_s^\var,Y_s^\var)\cdot\partial_y\psi(\veps^{-\gamma}s,X_s^\var,Y_s^\var)
\d W^2_s.
\end{align*}
Therefore,
\begin{align*}
&
\mbE\left|\int_0^tf_\veps(s,X_s^\veps,Y_s^\veps)-\hat{f}_\veps(s,X_s^\veps)\d s\right|^2\\
&
\leq \veps^{4\alpha}\mbE\Bigg|\psi(\veps^{-\gamma}t,X_t^\var,Y_t^\var)-\psi(0,x,y)
-\int_0^t(\veps^{-\gamma}\partial_s+\mfL_1^\var)\psi(\veps^{-\gamma}s,X_s^\var,Y_s^\var)\d s\\
&\quad
-\int_0^t\veps^{-\beta}\mfL_3\psi(\veps^{-\gamma}s,X_s^\var,Y_s^\var)\d s
-M_t^1-\var^{-\alpha}M_t^2\Bigg|^2,
\end{align*}
where
\[
M_t^1:=\int_0^t\sigma_\veps(s,X_s^\veps)\cdot\partial_x\psi(\veps^{-\gamma}s,X_s^\veps,Y_s^\veps)\d W_s^1,
\]
\[
M_t^2:=\int_0^tg(X_s^\veps,Y_s^\veps)\cdot\partial_y\psi(\veps^{-\gamma}s,X_s^\veps,Y_s^\veps)\d W_s^2.
\]
Then by Lemma \ref{SolPoisson}, {\bf(H$_x^2$)}, {\bf(H$_x^3$)} and {\bf(H$_y^3$)},
one sees that there exist $m',m''>0$ such that
\begin{align*}
\mfI_1(t,\var)
&
:=\mbE\left|\int_0^t f_\var(s,X_s^\var,Y_s^\var)-\hat{f}_\veps(s,X_s^\var)\d s\right|^2\\
&
\leq \veps^{4\alpha}C\left(1+\mbE\left(\sup_{t\in[0,T]}|X_t^\veps|^{m'}\right)
+\mbE\left(\sup_{t\in[0,T]}|Y_t^\veps|^{m''}\right)\right)\\
&
\quad
+\left(\veps^{4\alpha-2\gamma}+\veps^{4\alpha}+
\veps^{4\alpha-2\beta}\right)C\int_0^t\left(1+\mbE|X_s^\veps|^{m'}+\mbE|Y_s^\veps|^{m''}\right)\d s\\
&\quad
+5\var^{4\alpha}\mbE|M_t^1|^2+5\var^{2\alpha}\mbE|M_t^2|^2.
\end{align*}
And it follows from the Burkholder-Davis-Gundy inequality, Lemma \ref{SolPoisson}, {\bf(H$_x^3$)} and {\bf(H$_y^3$)} that there exist
$m',m''>0$ such that
\begin{align*}
\mbE|M_t^1|^2+\mbE|M_t^2|^2
&
\leq 2\int_0^t\mbE\left|\sigma_\veps(s,X_s^\veps)\cdot\partial_x\psi(\veps^{-\gamma}s,X_s^\veps,Y_s^\veps)\right|^2\d s\\
&\quad
+2\int_0^t\mbE\left|g(X_s^\veps,Y_s^\veps)\cdot\partial_y\psi(\veps^{-\gamma}s,X_s^\veps,Y_s^\veps)\right|^2\d s\\
&
\leq C\int_0^t\left(1+\mbE|X_s^\var|^{m'}+\mbE|Y_s^\var|^{m''}\right)\d s.
\end{align*}
Hence, letting $\widetilde{\alpha}:=\min\{4\alpha-2\gamma,4\alpha-2\beta,2\alpha\}$,
by Lemma \ref{Pest}, one sees that
\begin{equation}\label{FAPtheq11}
\mfI_1(t,\veps)\leq C_T\veps^{\widetilde{\alpha}}.
\end{equation}

Now we estimate
$$\mfI_2(t,\veps):=\mbE\left|\int_0^t\left(\hat{f}_\veps(s,X_s^\veps)-\bar{f}(\bar{X}_s)\right)\d s\right|^2.$$
Define
$\tau_n^\veps:=\inf\left\{t\geq0:|X_t^\veps|+|\bar{X}_t|>n\right\}\wedge T.$
Then we have
\begin{align}\label{FAPtheq14}
\mfI_2(t,\veps)
&
\leq\mbE\left(\chi_{\{\tau_n^\veps\geq t\}}\left|\int_0^{t}\left(\hat{f}_\veps(s,X_s^\veps)-\bar{f}(\bar{X}_s)\right)\d s\right|^2\right) \\\nonumber
&\quad
+\mbE\left(\chi_{\{\tau_n^\veps\leq t\}}\left|\int_0^{t}\left(\hat{f}_\veps(s,X_s^\veps)-\bar{f}(\bar{X}_s)\right)\d s\right|^2\right)
=:\msI_1+\msI_2.
\end{align}
For $\msI_2$, by H\"older's inequality, Chebyshev's inequality,
{\bf (H$_x^2$)}, Lemmas \ref{aveCCL} and \ref{Pest}, we have
\begin{align}
\msI_2
&
\leq\left(\mbE\chi_{\{\tau_n^\veps\leq t\}}^2\right)^{\frac12}
\left(\mbE\left|\int_0^{t}\left(\hat{f}_\veps(s,X_s^\veps)-\bar{f}(\bar{X}_s)\right)\d s\right|^4\right)^{\frac12}\nonumber \\
&
\leq n^{-1}C_T\left(\mbE\left(\sup_{t\in[0,T]}|X_t^\veps|^2\right)
+\mbE\left(\sup_{t\in[0,T]}|\bar{X}_t|^2\right)\right)^{\frac{1}{2}}\label{FAPtheq15}\\\nonumber
&\quad
\times\left(\mbE\int_0^T\left(|X_s^\veps|^{4\theta_1}+|\bar{X}_s|^{4\theta_1}+1\right)\d s\right)^{\frac12}
\leq C_Tn^{-1}.
\end{align}

Set $\widetilde{X}_s^\veps:=X_{k\delta}^\veps$ for $s\in[k\delta,(k+1)\delta)$, $k\in\N$.
For $\msI_1$, by H\"older's inequality,
Lemmas \ref{aveCCL} and \ref{timecon}, one sees that
\begin{align}
\msI_1
&
\leq4\mbE\left|\int_0^{t\wedge\tau_n^\veps}\left(\hat{f}_\veps(s,X_s^\veps)
-\hat{f}_\veps(s,\widetilde{X}_s^\veps)\right)\d s\right|^2\nonumber\\
&\quad
+4\mbE\left(\chi_{\{\tau_n^\veps\geq t\}}\left|\int_0^{t}
\left(\hat{f}_\veps(s,\widetilde{X}_s^\veps)-\bar{f}(\widetilde{X}^\veps_s)\right)\d s\right|^2\right)\nonumber\\
&\quad
+4\mbE\left|\int_0^{t\wedge\tau_n^\veps}\left(\bar{f}(\widetilde{X}^\veps_s)-\bar{f}(X^\veps_s)\right)\d s\right|^2
+4\mbE\left|\int_0^{t\wedge\tau_n^\veps}\left(\bar{f}(X^\veps_s)-\bar{f}(\bar{X}_s)\right)\d s\right|^2 \label{FAPtheq16}\\\nonumber
&
\leq C_Tn^{2\theta_1}\int_0^T\mbE|X_s^\veps-\widetilde{X}_s^\veps|^2\d s+C_Tn^{2\theta_1}\int_0^T\mbE|X_s^\veps-\bar{X}_s|^2\d s
+\msI_1^2\\\nonumber
&
\leq C_Tn^{2\theta_1}\delta+C_Tn^{2\theta_1}\int_0^t\mbE|X_s^\veps-\bar{X}_s|^2\d s
+\msI_1^2,
\end{align}
where
\[
\msI_1^2:=4\mbE\left(\chi_{\{\tau_n^\veps\geq t\}}\left|\int_0^{t}
\left(\hat{f}_\veps(s,\widetilde{X}_s^\veps)-\bar{f}(\widetilde{X}^\veps_s)\right)\d s\right|^2\right).
\]
Employing the technique of time discretization,
a change of variables, {\bf(A$_f$)}, Remark \ref{Rem01} and \eqref{PestX}, we have
\begin{align}
\msI_1^2
&
\leq8\mbE\left(\chi_{\{\tau_n^\veps\geq t\}}\left|\sum_{k=0}^{t(\delta)}\int_{k\delta}^{(k+1)\delta}
\left(\hat{f}_\veps(s,X_{k\delta}^\veps)-\bar{f}(X^\veps_{k\delta})\right)\d s\right|^2\right)\nonumber\\
&\quad
+8\mbE\left(\chi_{\{\tau_n^\veps\geq t\}}\left|\int_{t(\delta)\delta}^t\left(\hat{f}_\veps(s,X_{t(\delta)\delta}^\veps)
-\bar{f}(X^\veps_{t(\delta)\delta})\right)\d s\right|^2\right)\label{FAPtheq17}\\\nonumber
&
\leq C_T\left[\delta\left(\omega_n^f(\delta/\veps^{\gamma})\right)^2
+\delta^2\right],
\end{align}
where $t(\delta):=\left[\frac{t}{\delta}\right]$.
Then \eqref{FAPtheq16} and \eqref{FAPtheq17} yield
\begin{equation}\label{FAPtheq18}
\msI_1\leq   C_T\left[\delta\left(\omega_n^f(\delta/\veps^{\gamma})\right)^2+n^{2\theta_1}\delta\right]
+C_Tn^{2\theta_1}\int_0^t\mbE|X_s^\veps-\bar{X}_s|^2\d s.
\end{equation}
Combining \eqref{FAPtheq14}, \eqref{FAPtheq15} and \eqref{FAPtheq18}, we have
\begin{equation}\label{FAPtheq19}
\mfI_2(t,\var)
\leq C_T\left[\delta\left(\omega_n^f(\delta/\veps^{\gamma})\right)^2+n^{2\theta_1}\delta+n^{-1}\right]
+C_Tn^{2\theta_1}\int_0^t\mbE|X_s^\veps-\bar{X}_s|^2\d s.
\end{equation}

Similarly, for $\mfI_3(t,\var)$, it follows from the Burkholder-Davis-Gundy inequality, {\bf(A$_\sigma$)},
{\bf(H$_x^3$)}, Lemma \ref{timecon} and \eqref{PestX} that
\begin{align}
\mfI_3(t,\var)
&
:=\mbE\left|\int_0^t\left(\sigma_\veps(s,X_s^\var)-\bar{\sigma}(\bar{X}_s)\right)\d W^1_s\right|^2\nonumber\nonumber\\
&
\leq 4\mbE\int_0^t|\sigma_\veps(s,X_s^\veps)-\sigma_\veps(s,\widetilde{X}_s^\veps)|_{HS}^2\d s
+4\mbE\int_0^t|\sigma_\veps(s,\widetilde{X}_s^\veps)-\bar{\sigma}(\widetilde{X}_s^\veps)|_{HS}^2\d s\label{FAPtheq20}\\\nonumber
&\quad
+4\mbE\int_0^t|\bar{\sigma}(X_s^\veps)-\bar{\sigma}(\widetilde{X}_s^\veps)|_{HS}^2\d s
+4\mbE\int_0^tL_\sigma|X_s^\veps-\bar{X}_s|^2\d s\\\nonumber
&
\leq C_T\left(\delta+\omega^\sigma(\delta/\veps^{\gamma})\right)
+4\mbE\int_0^t L_\sigma|X_s^\veps-\bar{X}_s|^2\d s.
\end{align}

In view of \eqref{FAPtheq01}, \eqref{FAPtheq11}, \eqref{FAPtheq19} and \eqref{FAPtheq20}, we have
\begin{align*}
\sup_{t\in[0,T]}\mbE|X_t^\veps-\bar{X}_t|^2
&
\leq C\veps^{\widetilde{\alpha}}
+C_Tn^{2\theta_1}\int_0^T\sup_{0\leq r\leq s}\mbE|X_r^\veps-\bar{X}_r|^2\d s\\
&\quad
+C_T\left[\delta\left(\omega_n^f(\delta/\veps^{\gamma})\right)^2+n^{2\theta_1}\delta+\omega^\sigma(\delta/\veps^{\gamma})+n^{-1}\right],
\end{align*}
which by Gronwall's inequality implies
\begin{align}\label{FAPord}
\sup_{t\in[0,T]}\mbE|X_t^\veps-\bar{X}_t|^2
&
\leq C_T\left(\veps^{\widetilde{\alpha}}
+n^{2\theta_1}\delta\right){\rm exp}(Cn^{2\theta_1}T)\\\nonumber
&\quad
+C_T\left(\left(\omega_n^f(\delta/\veps^{\gamma})\right)^2
+n^{-1}+\omega^\sigma(\delta/\veps^{\gamma})\right){\rm exp}(Cn^{2\theta_1}T).
\end{align}
Let $\delta=\veps^{\gamma/2}$. Taking $\veps\rightarrow 0$ and $n\rightarrow\infty$,
we obtain
\begin{align*}
\lim_{\veps\rightarrow0}\sup_{t\in[0,T]}\mbE|X_t^\veps-\bar{X}_t|^2=0.
\end{align*}
\end{proof}

\section{Normal deviation}
This section is dedicated to proving the normal deviation for slow-fast stochastic differential equations \eqref{sfmain2}. In Section \ref{OpSCR}, we prove the optimal strong convergence rate for \eqref{sfmain2}. Subsequently, we show the normal deviation in Section \ref{PNDth}.

\subsection{The optimal strong convergence rate}\label{OpSCR}
Before investigating the optimal strong convergence rate, let us first show that
$\bar{f}$ is monotone under some suitable conditions.

\begin{lemma}\label{AFMClem}
Assume that {\bf{(H$_y^1$)}}--{\bf{(H$_y^3$)}} and {\bf{(H$_x^4$)}} hold.
Then there exists a constant $C>0$ such that for any $x_1,x_2\in\R^{d_1}$
\[
\langle \bar{f}(x_1)-\bar{f}(x_2),x_1-x_2\rangle\leq C|x_1-x_2|^2.
\]
\end{lemma}
\begin{proof}
It follows from {\bf{(H$_x^4$)}}, H\"older's inequality, Corollary \ref{EIMeq}
and Lemma \ref{SSCx} that for all $x_1,x_2\in\R^{d_1}$
\begin{align*}
&
\langle \bar{f}(x_1)-\bar{f}(x_2),x_1-x_2\rangle\\\nonumber
&
=\left\langle\int_{\R^{d_2}}f(x_1,y)\mu^{x_1}(\d y)-
\int_{\R^{d_2}}f(x_2,y)\mu^{x_2}(\d y),x_1-x_2\right\rangle\\\nonumber
&
=\int_{\R^{d_2}}\left\langle f(x_1,y)-f(x_2,y),x_1-x_2\right\rangle\mu^{x_1}(\d y)
+\left\langle\int_{\R^{d_2}}f(x_2,y)\left(\mu^{x_1}-\mu^{x_2}\right)(\d y),x_1-x_2\right\rangle\\\nonumber
&
\leq |x_1-x_2|^2\int_{\R^{d_2}}M\left(1+|y|^{\theta_2}\right)\mu^{x_1}(\d y)
+\left\langle\int_{\R^{d_2}}f(x_2,y)\left(\mu^{x_1}-\mu^{x_2}\right)(\d y),x_1-x_2\right\rangle\\\nonumber
&
\leq C|x_1-x_2|^2+C\mbE\left[\left(1+|Y_t^{x_1}|^{\theta_2}+|Y_t^{x_2}|^{\theta_2}\right)
\left|Y_t^{x_1}-Y_t^{x_2}\right|\right]|x_1-x_2|\\\nonumber
&
\leq C|x_1-x_2|^2+C\left[\mbE\left(1+|Y_t^{x_1}|^{2\theta_1}+|Y_t^{x_2}|^{2\theta_2}\right)\right]^{\frac12}
\left(\mbE\left|Y_t^{x_1}-Y_t^{x_2}\right|^2\right)^{\frac12}|x_1-x_2|\\\nonumber
&
\leq C|x_1-x_2|^2,
\end{align*}
where $t\in\R$ and $Y_\cdot^{x_1}$ (respectively, $Y_\cdot^{x_2}$) is the stationary solution to \eqref{fmain} with frozen $x_1$ (respectively, $x_2$).
\end{proof}

Now we can give the proof of the optimal strong convergence rate for the first averaging principle.

\begin{proof}[Proof of Theorem \ref{Opcrth}]
For simplicity, we define
$$X_t^\veps:=X_t^\veps(x),\quad  \bar{X}_t:=\bar{X}_t(x),
\quad Y_t^\veps:=Y_t^\veps(y)$$
for any $t\geq0$ in this subsection.
By It\^o's formula, {\bf(H$_x^3$)} and Lemma \ref{AFMClem}, we have
\begin{align}
|X_t^\var-\bar{X}_t|^2
&
=\int_0^t\Big(2\langle f(X_s^\var,Y_s^\var)
-\bar{f}(X_s^\var),X_s^\var-\bar{X}_s\rangle
+2\langle\bar{f}(X_s^\var)-\bar{f}(\bar{X}_s),X_s^\var-\bar{X}_s\rangle\nonumber\\\nonumber
&\qquad
+\left|\sigma(X_s^\var)-\sigma(\bar{X}_s)\right|_{HS}^2\Big)\d s+2\int_0^t\langle X_s^\veps-\bar{X}_s,\left(\sigma(X_s^\veps)-\sigma(\bar{X}_s)\right)\d W_s^1\rangle\\
&
\leq \int_0^t2\langle f(X_s^\var,Y_s^\var)
-\bar{f}(X_s^\var),X_s^\var-\bar{X}_s\rangle+C|X_s^\var-\bar{X}_s|^2\d s\label{OFAPtheq01}\\\nonumber
&\quad
+2\int_0^t\langle X_s^\veps-\bar{X}_s,\left(\sigma(X_s^\veps)-\sigma(\bar{X}_s)\right)\d W_s^1\rangle.
\end{align}
Define $\phi(x,y,\bar{x}):=2\langle f(x,y),x-\bar{x}\rangle$, $\bar{\phi}(x,\bar{x}):=2\langle \bar{f}(x),x-\bar{x}\rangle$ for all
$(x,y,\bar{x})\in\R^{d_1+d_2+d_1}$.
For any $(x,\bar{x})\in\R^{2d_1}$ let $\Psi(x,\cdot,\bar{x})$ be the solution to
$$
\mfL_2\Psi(x,y,\bar{x})=\phi(x,y,\bar{x})-\bar{\phi}(x,\bar{x}),\quad y\in\R^{d_2}.
$$
By Lemma \ref{SolPoisson}, one sees that $\Psi\in C^{2,4,2}(\R^{d_1+d_2+d_1})$ and that
there exist constants $C,m_1,m_2,m_3>0$ such that
\begin{align}\label{EPsi}
&
|\Psi(x,y,\bar{x})|+\sum_{i=1}^2\left(|\partial_x^i\Psi(x,y,\bar{x})|+|\partial_{\bar{x}}^i\Psi(x,y,\bar{x})|\right)
+|\partial_y\Psi(x,y,\bar{x})|\\\nonumber
&
\leq C\left(1+|x|^{m_1}+|y|^{m_2}+|\bar{x}|^{m_3}\right).
\end{align}
Applying It\^o's formula to $t\mapsto\Psi(X_t^\veps,Y_t^\veps,\bar{X}_t)$, we have
\begin{align*}
&
\Psi(X_t^\var,Y_t^\var,\bar{X}_t)
-\Psi(x,y,x)\\
&
=\int_0^t(\mfL_1+\mfL_{\bar{x}}+\var^{-1}\mfL_2)\Psi(X_s^\var,Y_s^\var,\bar{X}_s)\d s
+\int_0^t \sigma(X_s^\var)\cdot\partial_x\Psi(X_s^\var,Y_s^\var,\bar{X}_s)\d W^1_s\\
&\quad
+\int_0^t\sigma(\bar{X}_s)\cdot\partial_{\bar{x}}\Psi(X_s^\veps,Y_s^\veps,\bar{X}_s)\d W_s^1
+\var^{-\frac12}\int_0^t g(X_s^\veps,Y_s^\var)\cdot\partial_y\Psi(X_s^\var,Y_s^\var,\bar{X}_s)
\d W^2_s,
\end{align*}
which implies that
\begin{align}
&
\int_0^t2\langle\left(f(X_s^\veps,Y_s^\var)
-\bar{f}(X_s^\var)\right),X_s^\var-\bar{X}_s\rangle\d s\nonumber\\
&
=\veps\left(\Psi(X_t^\var,Y_t^\var,\bar{X}_t)-\Psi(x,y,x)-\int_0^t(\mfL_1+\mfL_{\bar{x}})\Psi(X_s^\var,Y_s^\var,\bar{X}_s)\d s
\right)\label{OFAPtheq02}\\\nonumber
&\quad
-\veps\int_0^t\sigma(X_s^\veps)\cdot\partial_x\Psi(X_s^\veps,Y_s^\veps,\bar{X}_s)\d W_s^1
-\veps\int_0^t\sigma(\bar{X}_s)\cdot\partial_{\bar{x}}\Psi(X_s^\veps,Y_s^\veps,\bar{X}_s)\d W_s^1
\\\nonumber
&\quad
-\sqrt{\veps}\int_0^tg(X_s^\veps,Y_s^\veps)\cdot\partial_y\Psi(X_s^\veps,Y_s^\veps,\bar{X}_s)\d W_s^2.
\end{align}

Note that $\Psi(x,y,\bar{x})=2\langle \psi(x,y),x-\bar{x}\rangle$, where $\psi$ is the solution to
\[
\mfL_2\psi(x,y)=f(x,y)-\bar{f}(x).
\]
Then in view of \eqref{OFAPtheq01},
\eqref{EPsi}, \eqref{OFAPtheq02}, {\bf(H$_x^2$)}, {\bf(H$_x^3$)}, Lemma \ref{aveCCL} and Young's inequality, there exist constants $m'_1,m'_2,m'_3>0$ such that
\begin{align*}
&
|X_t^\var-\bar{X}_t|^2\nonumber\\\nonumber
&
\leq \int_0^tC|X_s^\var-\bar{X}_s|^2\d s
+2\int_0^t\langle X_s^\veps-\bar{X}_s,\left(\sigma(X_s^\veps)-\sigma(\bar{X}_s)\right)\d W_s^1\rangle\\\nonumber
&\quad
+\veps\left(\Psi(X_t^\var,Y_t^\var,\bar{X}_t)-\Psi(x,y,x)-\int_0^t(\mfL_1+\mfL_{\bar{x}})\Psi(X_s^\var,Y_s^\var,\bar{X}_s)\d s
\right)\\\nonumber
&\quad
+\veps\int_0^t\sigma(X_s^\veps)\cdot\partial_x\Psi(X_s^\veps,Y_s^\veps,\bar{X}_s)\d W_s^1
+\veps\int_0^t\sigma(\bar{X}_s)\cdot\partial_{\bar{x}}\Psi(X_s^\veps,Y_s^\veps,\bar{X}_s)\d W_s^1\\
&\quad
+\sqrt{\veps}2\int_0^tg(X_s^\veps,Y_s^\veps)\cdot\langle\partial_y\psi(X_s^\veps,Y_s^\veps),
X_s^\veps-\bar{X}_s\rangle\d W_s^2 \label{OFAPtheq03}\\\nonumber
&
\leq \int_0^tC|X_s^\var-\bar{X}_s|^2\d s
+2\int_0^t\langle X_s^\veps-\bar{X}_s,\left(\sigma(X_s^\veps)-\sigma(\bar{X}_s)\right)\d W_s^1\rangle\\\nonumber
&\quad
+\veps C\left(1+\sup_{t\in[0,T]}|X_t^\veps|^{m'_1}+\sup_{t\in[0,T]}|Y_t^\veps|^{m'_2}
+\sup_{t\in[0,T]}|\bar{X}_t|^{m'_3}\right)\\\nonumber
&\quad
+\veps\int_0^t\sigma(X_s^\veps)\cdot\partial_x\Psi(X_s^\veps,Y_s^\veps,\bar{X}_s)\d W_s^1
+\veps\int_0^t\sigma(\bar{X}_s)\cdot\partial_{\bar{x}}\Psi(X_s^\veps,Y_s^\veps,\bar{X}_s)\d W_s^1\\\nonumber
&\quad
+\sqrt{\veps}2\int_0^tg(X_s^\veps,Y_s^\veps)\cdot\langle\partial_y\psi(X_s^\veps,Y_s^\veps),
X_s^\veps-\bar{X}_s\rangle\d W_s^2.
\end{align*}

Then thanks to Lemma \ref{Pest} and Burkholder-Davis-Gundy's inequality, we obtain
\begin{align*}
&
\mbE\left(\sup_{t\in[0,T]}|X_t^\veps-\bar{X}_t|^2\right)\\
&
\leq C\int_0^T\mbE|X_s^\veps-\bar{X}_s|^2\d s
+6\mbE\left(\int_0^T|X_s^\veps-\bar{X}_s|^2|\sigma(X_s^\veps)-\sigma(\bar{X}_s)|^2\d s\right)^{\frac12}\\
&\quad
+\veps C_T\left(1+\mbE\left(\sup_{t\in[0,T]}|X_t^\veps|^{m'_1}\right)+\mbE\left(\sup_{t\in[0,T]}|Y_t^\veps|^{m'_2}\right)
+\mbE\left(\sup_{t\in[0,T]}|\bar{X}_t|^{m'_3}\right)\right)\\
&\quad
+\sqrt{\veps}6\mbE\left(\int_0^T|g(X_s^\veps,Y_s^\veps)|^2|\partial_y\psi(X_s^\veps,Y_s^\veps)|^2
|X_s^\veps-\bar{X}_s|^2\d s\right)^{\frac12}\\
&
\leq \frac{1}{2}\mbE\left(\sup_{t\in[0,T]}|X_s^\veps-\bar{X}_s|^2\right)
+C\int_0^T\mbE|X_s^\veps-\bar{X}_s|^2\d s+C_T\veps,
\end{align*}
which by Gronwall's inequality implies that
\begin{equation*}
\mbE\left(\sup_{t\in[0,T]}|X_s^\veps-\bar{X}_s|^2\right)\leq C_T\veps.
\end{equation*}
\end{proof}

\subsection{Proof of Theorem \ref{NDth}}\label{PNDth}

Prior to presenting the proof of the theorem regarding the normal deviation, we prove several lemmas.

\begin{lemma}
Assume that {\bf(H$_y^1$)}, {\bf(H$_y^3$)} and
{\bf(H$_x^1$)}--{\bf(H$_x^3$)} hold.
Let $(X_t^\veps(x),Y_t^\veps(y))$ be the solution to \eqref{sfmain} for any $(x,y)\in\R^{d_1+d_2}$.
Furthermore, suppose that there exist constants $c_1,c_2,C\geq1$ such that for all $(x,y)\in\R^{d_1+d_2}$
\[
|b(x,y)|+|B(x,y)|\leq C\left(1+|x|^{c_1}+|y|^{c_2}\right).
\]
Then there exists a constant $C_T>0$ such that for any $0\leq s<t\leq T$ and $0<\veps\leq 1$
\begin{equation}\label{timeconY}
\veps\mbE|Y_{t}^\veps-Y_s^\veps|^2\leq C_T|t-s|.
\end{equation}
\end{lemma}
\begin{proof}
It follows from It\^o's formula, \eqref{PestX} and \eqref{PestY} that
\begin{align*}
\mbE|Y_{t}^\veps-Y_s^\veps|^2
&
=\frac1\veps\mbE\int_s^{t}\left(2\langle B(X_r^\veps,Y_r^\veps),Y_r^\veps-Y_s^\veps\rangle
+|g(X_r^\veps,Y_r^\veps)|_{HS}^2\right)\d r\\
&\quad
+\frac{1}{\veps^\beta}\mbE\int_s^t2\langle b(X_r^\veps,Y_r^\veps),Y_r^\veps-Y_s^\veps\rangle\d r\\
&
\leq \frac1\veps\mbE\int_s^{t}C\left(1+\sup_{r\in[0,T]}|X_r^\veps|^{2c_1}
+\sup_{r\in[0,T]}|Y_r^\veps|^{2c_2}\right)\d r
\leq \frac1\veps C_T|t-s|.
\end{align*}
\end{proof}

\begin{lemma}\label{TightLem1}
If {\bf(H$_x^1$)}--{\bf(H$_x^4$)} and {\bf(H$_y^1$)}--{\bf(H$_y^3$)} hold,
then $Z^\veps$ is tight in $C([0,T];\R^{d_1})$.
\end{lemma}
\begin{proof}
According to \eqref{Opcrsfeq} and the Arzela-Ascoli theorem, it suffices to show that there exists a constant $C_T>0$ such that for all $0\leq s\leq t\leq T$
\begin{equation*}\label{TightLemeq01}
\mbE|Z_t^\veps-Z_s^\veps|\leq C_T|t-s|^{\frac12}.
\end{equation*}
To this end, by It\^o's formula, the Burkholder-Davis-Gundy inequality,
H\"older's inequality and \eqref{Opcrsfeq}, one sees that
\begin{align}
&
\mbE|Z_t^\veps-Z_s^\veps|\nonumber\\
&
\leq\mbE\left|\int_s^t\frac{1}{\sqrt{\veps}}\left(f(X_r^\veps,Y_r^\veps)-\bar{f}(\bar{X}_r)\right)\d r\right|+\mbE\left|\int_s^t\frac{1}{\sqrt{\veps}}(\sigma(X_r^\veps)-\sigma(\bar{X}_r))\d W_r^1\right|\nonumber\\
&
\leq \msI_1+\mbE\int_s^t\left|\nabla\bar{f}(\bar{X}_r+\iota(X_r^\veps-\bar{X}_r))\right|\left|Z_r^\veps\right|\d r
+C_T|t-s|^{\frac12}\label{TightLemeq02}\\\nonumber
&
\leq \msI_1
+\left(\mbE\int_s^t\left|\nabla\bar{f}(\bar{X}_r+\iota(X_r^\veps-\bar{X}_r))\right|^{2}\d r\right)^{\frac12}\left(\mbE\int_s^t|Z_r^\veps|^2\d r\right)^{\frac12}
+C_T|t-s|^{\frac12},
\end{align}
where $\iota\in[0,1]$ and
\[
\msI_1:=\mbE\left|\int_s^t\frac{1}{\sqrt{\veps}}\left(f(X_r^\veps,Y_r^\veps)-\bar{f}(X^\veps_r)\right)\d r\right|.
\]
Employing Lemma \ref{aveCCL}, \eqref{PestX}, \eqref{PestAX}
and \eqref{Opcrsfeq}, we obtain that there exists a constant $C_T>0$ such that
\begin{align}\label{TightLemeq03}
&
\left(\mbE\int_s^t\left|\nabla\bar{f}(\bar{X}_r+\iota(X_r^\veps-\bar{X}_r))\right|^{2}\d r\right)^{\frac12}
\left(\mbE\int_s^t|Z_r^\veps|^2\d r\right)^{\frac12}\\\nonumber
&
\leq \left(\mbE\int_s^t\left(|X_r^\veps|^{2\theta_1}+|\bar{X}_r|^{2\theta_1} \right)\d r\right)^{\frac12}
C_T|t-s|^{\frac12}
\leq C_T|t-s|.
\end{align}

Now we estimate $\msI_1$.
Note that by the Burkholder-Davis-Gundy inequality, we get
\begin{align*}
\msI_1
&
\leq \sqrt{\veps}\mbE\left|\psi(X_t^\var,Y_t^\var)-\psi(X_s^\var,Y_s^\var)\right|
+\sqrt{\veps}\mbE\left|\int_s^t\mfL_1\psi(X_r^\veps,Y_r^\veps)\d r\right|\\\nonumber
&\quad
+\sqrt{\veps}\mbE\left|\int_s^t \sigma(X_r^\var)\cdot\partial_x\psi(X_r^\var,Y_r^\var)\d W^1_r\right|
+\mbE\left|\int_s^t g(X_r^\veps,Y_r^\var)\cdot\partial_y\psi(X_r^\var,Y_r^\var)
\d W^2_r\right|\\\nonumber
&
\leq \sqrt{\veps}\mbE\left|\psi(X_t^\var,Y_t^\var)-\psi(X_s^\var,Y_s^\var)\right|
+\sqrt{\veps}\mbE\left|\int_s^t\mfL_1\psi(X_r^\veps,Y_r^\veps)\d r\right|\\\nonumber
&\quad
+\sqrt{\veps}3\mbE\left(\int_s^t\left|\sigma(X_r^\veps)
\partial_x\psi(X_r^\veps,Y_r^\veps)\right|^2\d r\right)^{\frac{1}{2}}
+3\mbE\left(\int_s^t \left|g(X_r^\veps,Y_r^\var)\partial_y\psi(X_r^\var,Y_r^\var)\right|^2
\d r\right)^{\frac{1}{2}},
\end{align*}
where $\psi$ is the solution to
$\mfL_2\psi(x,y)=f(x,y)-\bar{f}(x)$.
Then in view of Lemma \ref{SolPoisson}, {\bf (H$_x^2$)}, {\bf (H$_x^3$)},
Lemma \ref{Pest}
and H\"older's inequality, \eqref{timeconX} and \eqref{timeconY}, there exist constants $\iota_1,\iota_2\in(0,1)$, $C>0$ and $p_1,p_2>0$ such that
\begin{align}
\msI_1
&
\leq \sqrt{\veps}\mbE\left(|\partial_x\psi(X_s^\veps+\iota_1(X_t^\veps-X_s^\veps),Y_t^\veps)||X_t^\veps-X_s^\veps|\right)
\nonumber\\
&\quad
+\sqrt{\veps}\mbE\left(|\partial_y\psi(X_s^\veps,Y_s^\veps+\iota_2(Y_t^\veps-Y_s^\veps))||Y_t^\veps-Y_s^\veps|\right) \label{TightLemeq05}\\\nonumber
&\quad
+\sqrt{\veps}C\mbE\int_s^t\left(1+|X_r^\veps|^{p_1}+|Y_r^\veps|^{p_2}\right)\d r\\\nonumber
&
\leq C_T\left(\sqrt{\veps}\left(\mbE|X_t^\veps-X_s^\veps|^2\right)^{\frac12}
+\left(\veps\mbE|Y_t^\veps-Y_s^\veps|^2\right)^{\frac12}
+|t-s|^{\frac12}\right)\leq C_T|t-s|^{\frac12}.
\end{align}
Combining \eqref{TightLemeq02}, \eqref{TightLemeq03} and \eqref{TightLemeq05}, we obtain
\begin{align*}
\mbE|Z_t^\veps-Z_s^\veps|
&
\leq C_T|t-s|^{\frac12}.
\end{align*}
\end{proof}

\begin{lemma}\label{Fluctuation}
Assume that {\bf (H$_y^1$)}--{\bf (H$_y^5$)} hold and let $\phi\in C^{2,4,2}(\R^{d_1+d_2+d_1})$
satisfying {\bf (H$_\phi^2$)}.
Furthermore, suppose that $\phi(x,y,\cdot)\in C_b^2(\R^{d_1})$
and $\partial_y^j\partial_x^i\phi(x,y,\cdot)\in C_b(\R^{d_1})$ for all $(x,y)\in\R^{d_1+d_2}$ and $0\leq 2i+j\leq4$.
Then there exists $C_T>0$ such that for all $0\leq t\leq T$
\begin{equation*}
\mbE\left(\int_0^t\delta\phi(X_s^\veps,Y_s^\veps,Z_s^\veps)\d s\right)\leq C_T\veps^{\frac12},
\end{equation*}
where $\delta\phi(x,y,z):=\phi(x,y,z)-\int_{\R^{d_2}}\phi(x,y,z)\mu^x(\d y),~\forall (x,y,z)\in\R^{d_1+d_2+d_1}$.
\end{lemma}
\begin{proof}
 For any $(x,z)\in\R^{2d_1}$, let $\psi$ be the solution to
\begin{equation*}\label{Fluctuation01}
\mfL_2\psi(x,y,z)=\delta\phi(x,y,z),\quad y\in\R^{d_2}.
\end{equation*}
It follows from Lemma \ref{SolPoisson} that $\psi\in C^{2,2,2}(\R^{d_1+d_2+d_1})$
and that there exist constants $C,m_1,m_2>0$ such that for all $(x,y,z)\in\R^{d_1+d_2+d_1}$
\begin{align}
&
|\psi(x,y,z)|+|\partial_x\psi(x,y,z)|+|\partial_x^2\psi(x,y,z)|\nonumber\\
&\quad
+|\partial_z\psi(x,y,z)|+|\partial_z^2\psi(x,y,z)|+|\partial_y\psi(x,y,z)|\label{Fluctuation02}\\\nonumber
&\leq C\left(1+|x|^{m_1}+|y|^{m_2}\right).
\end{align}

By It\^o's formula, one sees that
\begin{align*}
\mbE\psi(X_t^\veps,Y_t^\veps,Z_t^\veps)-\psi(x,y,0)
&
=\mbE\int_0^t\mfL_1\psi(X_s^\veps,Y_s^\veps,Z_s^\veps)\d s
+\frac{1}{\veps}\mbE\int_0^t\mfL_2\psi(X_s^\veps,Y_s^\veps,Z_s^\veps)\d s\\\nonumber
&\quad
+\frac{1}{\sqrt{\veps}}\mbE\int_0^t\left\langle f(X_s^\veps,Y_s^\veps)-\bar{f}(\bar{X}_s),\partial_z\psi(X_s^\veps,Y_s^\veps,Z_s^\veps)\right\rangle\d s,
\end{align*}
which by \eqref{Fluctuation02}, H\"older's inequality, {\bf (H$_x^2$)} and {\bf (H$_x^3$)}  implies that there exist $p_1,p_2,p_3>0$ such that
\begin{align}
\mbE\int_0^t\delta\phi(X_s^\veps,Y_s^\veps,Z_s^\veps)\d s
&
=\veps\left(\mbE\psi(X_t^\veps,Y_t^\veps,Z_t^\veps)-\psi(x,y,0)-\mbE\int_0^t\mfL_1\psi(X_s^\veps,Y_s^\veps,Z_s^\veps)\d s\right)\nonumber\\
&\quad
-\sqrt{\veps}\mbE\int_0^t\left\langle f(X_s^\veps,Y_s^\veps)-\bar{f}(\bar{X}_s),\partial_z\psi(X_s^\veps,Y_s^\veps,Z_s^\veps)\right\rangle\d s\label{Fluctuation03}\\\nonumber
&
\leq \sqrt{\veps}C\mbE\left(1+\sup_{t\in[0,T]}|X_t^\veps|^{p_1}+\sup_{t\in[0,T]}|Y_t^\veps|^{p_2}
+\sup_{t\in[0,T]}|\bar{X}_t|^{p_3}\right).
\end{align}
Combining \eqref{Fluctuation03} and Lemma \ref{Pest},  we have
\begin{equation*}
\mbE\int_0^t\delta\phi(X_s^\veps,Y_s^\veps,Z_s^\veps)\d s\leq  C_T\veps^{\frac{1}{2}}.
\end{equation*}
\end{proof}

Now we recall the following lemma, which is from \cite[Proposition 3]{PV2001}.
\begin{lemma}\label{TightLem2}
If $\{Z^\veps,\bar{Z}:0<\veps\leq1\}\subset C([0,T];\R^{d_1})$ is tight, then
for any $\delta>0$ there exist $N\in\mathbb N$ and $z^1,...,z^N\in C([0,T];\R^{d_1})$ such that for any $0<\veps\leq1$
\begin{equation*}
\mbP\left(\bigcap_{k=1}^N\left\{\sup_{t\in[0,T]}|Z_t^{\veps}-z_t^k|>\delta\right\}\right)<\delta,
\quad
\mbP\left(\bigcap_{k=1}^N\left\{\sup_{t\in[0,T]}|\bar{Z}_t-z_t^k|>\delta\right\}\right)<\delta.
\end{equation*}
\end{lemma}

\begin{remark}\label{GRem}\rm
(i) Note that we also have
\begin{equation}\label{ExG}
GG^T(x)=\int_{\R^{d_2}}\left(f(x,y)-\bar{f}(x)\right)\psi(x,y)\mu^x(\d y), ~\forall x\in\R^{d_1},
\end{equation}
where $\psi(x,y)$ is the solution to
$\mfL_2(x,y)\psi(x,y)=f(x,y)-\bar{f}(x),~y\in\R^{d_2}$.

(ii) It follows from {\bf(H$_x^2$)}, \eqref{ExG} and Lemma \ref{SolPoisson} that
there exist constants $m>0$ and $C>0$ such that
for any $x\in\R^{d_1}$
\[
|G(x)|\vee|GG^T(x)|\vee|\nabla(GG^T(x))|\leq C(1+|x|^m).
\]
Therefore, we have for any $x_1,x_2\in\R^{d_1}$
\[
|GG^T(x_1)-GG^T(x_2)|\leq C(1+|x_1|^{m}+|x_2|^{m})|x_1-x_2|.
\]
\end{remark}

We are now in a position to prove Theorem \ref{NDth}.
\begin{proof}[Proof of Theorem \ref{NDth}]
Note that it follows from It\^o's formula that
\begin{align}\label{NDeq01}
\varphi(\bar{Z}_t)
&
=\varphi(0)+\int_0^t\langle\nabla\bar{f}(\bar{X}_s)\bar{Z}_s,\nabla\varphi(\bar{Z}_s)\rangle+\frac12 Tr[\nabla^2\varphi(\bar{Z}_s)GG^T(\bar{X}_s)]\d s\\\nonumber
&\quad
+\int_0^t\langle\nabla\varphi(\bar{Z}_s),G(\bar{X}_s)\d \widetilde{W}_s^1\rangle,
\end{align}
\begin{equation}\label{NDeq02}
\varphi(Z^\veps_t)=\varphi(0)+\int_{0}^t \frac{1}{\sqrt{\veps}}\langle f(X_s^\veps,Y_s^\veps)-\bar{f}(\bar{X}_s),\nabla\varphi(Z^\veps_s) \rangle\d s.
\end{equation}
Combining \eqref{NDeq01} and \eqref{NDeq02}, we have
\begin{align*}
&
|\mbE\varphi(Z_t^\veps)-\mbE\varphi(\bar{Z}_t)|\nonumber\\
&
=\Bigg|\mbE\int_{0}^t \bigg(\frac{1}{\sqrt{\veps}}\langle f(X_s^\veps,Y_s^\veps)-\bar{f}(\bar{X}_s),\nabla \varphi(Z^\veps_s) \rangle\nonumber
-\langle\nabla\bar{f}(\bar{X}_s)\bar{Z}_s,\nabla\varphi(\bar{Z}_s) \rangle\\
&\qquad
-\frac12 Tr[\nabla^2_z\varphi(\bar{Z}_s)GG^T(\bar{X}_s)]\bigg)\d s\Bigg|
\leq \msI_1+\msI_2+\msI_3,
\end{align*}
where
\[
\msI_1:=\left|\mbE\int_{0}^t\frac{1}{\sqrt{\veps}}\langle f(X_s^\veps,Y_s^\veps)-\bar{f}(X_s^\veps),\nabla\varphi(Z_s^\veps)\rangle-\frac12 Tr[\nabla^2\varphi(Z_s^\veps)GG^T(X_s^\veps)]\d s\right|,
\]
\begin{align*}
\msI_2:
&=\left|\mbE\int_0^t\frac{1}{\sqrt{\veps}}\langle \bar{f}(X_s^\veps)-\bar{f}(\bar{X}_s),\nabla\varphi(Z_s^\veps)\rangle-\langle\nabla\bar{f}(\bar{X}_s)\bar{Z}_s,\nabla\varphi(\bar{Z}_s)\rangle\d s\right|\\
&
=\left|\mbE\int_0^t\langle \nabla\bar{f}(\bar{X}_s+\iota(X_s^\veps-\bar{X}_s))Z_s^\veps,\nabla\varphi(Z_s^\veps)\rangle-\langle\nabla\bar{f}(\bar{X}_s)\bar{Z}_s,\nabla\varphi(\bar{Z}_s)\rangle\d s\right|,
\end{align*}
\[
\msI_3:=\left|\mbE\int_0^t\frac12 Tr[\nabla^2\varphi(Z_s^\veps)GG^T(X_s^\veps)]
-\frac12 Tr[\nabla^2\varphi(\bar{Z}_s)GG^T(\bar{X}_s)]\d s\right|
\]
for some $\iota\in[0,1]$.
Thus, we only need to demonstrate that $\lim_{\veps\rightarrow0}\msI_1=0$,
$\lim_{\veps\rightarrow0}\msI_2=0$, and $\lim_{\veps\rightarrow0}\msI_3=0$
to establish our result. To this end, we will break the proof into three steps.

{\bf{(Step 1)}}
Let
$
\phi(x,y,z):=\langle f(x,y),\nabla\varphi(z)\rangle,~
\bar{\phi}(x,z):=\langle \bar{f}(x),\nabla\varphi(z)\rangle,~
\forall(x,y,z)\in\R^{d_1+d_2+d_1}$.
It is obvious that
$\Psi(x,y,z):=-\langle \psi(x,y),\nabla\varphi(z)\rangle$ is the solution to
$$
\mfL_2\Psi(x,y,z)=-\left(\phi(x,y,z)-\bar{\phi}(x,z)\right),
$$
where $\psi(x,y)$ is the solution to
$
\mfL_2\psi(x,y)=\left(f(x,y)-\bar{f}(x)\right).
$
Then by Lemma \ref{SolPoisson}, one sees that $\Psi\in C^{2,2,2}(\R^{d_1+d_2+d_1})$ and
that there exist constants $C,m_1,m_2>0$ such that
\begin{align}
&
|\Psi(x,y,z)|+|\partial_x\Psi(x,y,z)|+|\partial_x^2\Psi(x,y,z)|\nonumber\\
&\quad
+|\partial_z\Psi(x,y,z)|+|\partial_z^2\Psi(x,y,z)|+|\partial_y\Psi(x,y,z)|\label{NDeq04-1}\\\nonumber
&\leq C\left(1+|x|^{m_1}+|y|^{m_2}\right).
\end{align}
Applying It\^o's formula to $t\mapsto\Psi(X_t^\veps,Y_t^\veps,Z_t^\veps)$, we have
\begin{align*}
\mbE\Psi(X_t^\veps,Y_t^\veps,Z_t^\veps)-\mbE\Psi(x,y,0)
&
=\mbE\int_{0}^t\Big(\mfL_1\Psi(X_s^\veps,Y_s^\veps,Z_s^\veps)
+\frac1\veps\mfL_2\Psi(X_s^\veps,Y_s^\veps,Z_s^\veps)
\\\nonumber
&\qquad
+\frac{1}{\sqrt{\veps}}\langle f(X_s^\veps,Y_s^\veps)-\bar{f}(\bar{X}_s),\partial_z\Psi(X_s^\veps,Y_s^\veps,Z_s^\veps)\rangle\Big)\d s.
\end{align*}
Therefore, by \eqref{NDeq04-1}, {\bf(H$_x^2$)}, {\bf(H$_x^3$)}, H\"older's inequality,
Lemma \ref{Pest} and Theorem \ref{Opcrth} we obatin
\begin{align}
\msI_1
&
=\Bigg|\sqrt{\veps}\left(\mbE\Psi(x,y,0)-\mbE\Psi(X_{t}^\veps,Y_{t}^\veps,Z_{t}^\veps)
+\mbE\int_{0}^t\mfL_1\Psi(X_s^\veps,Y_s^\veps,Z_s^\veps)\d s\right)\nonumber\\
&\quad
+\mbE\int_{0}^t\langle f(X_s^\veps,Y_s^\veps)-\bar{f}(\bar{X}_s),\partial_z\Psi(X_s^\veps,Y_s^\veps,Z_s^\veps)\rangle-\frac12 Tr[\nabla^2\varphi(Z_s^\veps)GG^T(X_s^\veps)]\d s\Bigg|\nonumber\\
&
\leq C\sqrt{\veps}+\left|\mbE\int_{0}^t\langle \bar{f}(X_s^\veps)-\bar{f}(\bar{X}_s),\partial_z\Psi(X_s^\veps,Y_s^\veps,Z_s^\veps)\rangle \right|+\msI_1^3\label{NDeq04-3}\\\nonumber
&
\leq C\sqrt{\veps}+\left|\mbE\int_{0}^t|\nabla \bar{f}(\bar{X}_s+\iota(X_s^\veps-\bar{X}_s))||X_s^\veps-\bar{X}_s||\partial_z\Psi(X_s^\veps,Y_s^\veps,Z_s^\veps)|\d s\right|+\msI_1^3\\\nonumber
&
\leq C\sqrt{\veps}+\msI_1^3
\end{align}
for some $\iota\in[0,1]$, where
$$
\msI_1^3:=\left|\mbE\int_{0}^t\langle f(X_s^\veps,Y_s^\veps)-\bar{f}(X^\veps_s),\partial_z\Psi(X_s^\veps,Y_s^\veps,Z_s^\veps)\rangle-\frac12 Tr[\nabla^2\varphi(Z_s^\veps)GG^T(X_s^\veps)]\d s\right|.
$$
Note that
\begin{align*}
&
\langle f(x,y)-\bar{f}(x),\partial_z\Psi(x,y,z)\rangle-\frac12Tr[\nabla^2\varphi(z)GG^T(x)]\\
&
=\sum_{i,j=1}^{d_1}\left(f_i(x,y)-\bar{f}_i(x)\right)\psi_j(x,y)\partial_{z_jz_i}^2\varphi(z)\\
&\quad
-\sum_{i,j=1}^{d_1}\int_{\R^{d_2}}\left(f_i(x,y)-\bar{f}_i(x)\right)\psi_j(x,y)\partial_{z_jz_i}^2\varphi(z)
\mu^x(\d y).
\end{align*}
By Lemma \ref{Fluctuation}, one sees that
$\lim_{\veps\rightarrow0}\msI_1^3=0$,
which by \eqref{NDeq04-3} implies that
$\lim_{\veps\rightarrow0}\msI_1=0$.

{\bf{(Step 2)}}
Now we show that $\lim_{\veps\rightarrow0}\msI_2=0$.
By Lemmas \ref{TightLem1} and \ref{TightLem2}, for any $\delta>0$
there exist $z^1,...,z^N\in C([0,T];\R^{d_1})$ such that for any $0<\veps\leq1$
\begin{equation}\label{NDeq04-4}
\mbP\left( \bigcap_{k=1}^N\left\{\sup_{t\in[0,T]}|Z_t^\veps-z_t^k|>\delta\right\}\right)<\delta,
\end{equation}
\begin{equation}\label{NDeq05}
\mbP\left(\bigcap_{k=1}^N\left\{\sup_{t\in[0,T]}|\bar{Z}_t-z_t^k|>\delta\right\}\right)<\delta.
\end{equation}
Define
\[
\Omega_{1,k}^\veps:=\left\{\sup_{t\in[0,T]}|Z_t^\veps-z_t^k|<2\delta,\sup_{t\in[0,T]}|\bar{Z}_t-z_t^k|<2\delta\right\},
\]
\[
\Omega_2^\veps:=\bigcap_{k=1}^N\left\{\sup_{t\in[0,T]}|Z_t^\veps-z_t^k|>\delta\right\},
\quad
\Omega_3:=\bigcap_{k=1}^N\left\{\sup_{t\in[0,T]}|\bar{Z}_t-z_t^k|>\delta\right\}.
\]
Let
$
\widetilde{\Omega}_{1,1}^\veps=\Omega_{1,1}^\veps,~
\widetilde{\Omega}_{1,k}^\veps=\Omega_{1,k}^\veps\backslash\left(\cup_{i=1}^{k-1}\widetilde{\Omega}_{1,i}^\veps\right),k=2,...,N.
$
It is obvious that for any $0<\veps\leq1$
$$
\Omega=\cup_{k=1}^N\widetilde{\Omega}_{1,k}^\veps\cup\Omega_2^\veps\cup\Omega_3,
\quad
\widetilde{\Omega}_{1,i}^\veps\cap\widetilde{\Omega}_{1,j}^\veps=\emptyset, i\neq j.
$$
Therefore,
\begin{align}\label{NDeq06}
\msI_2
&
=\left|\mbE\int_0^t\langle\nabla\bar{f}(\bar{X}_s+\iota(X_s^\veps-\bar{X}_s))Z_s^\veps,\nabla\varphi(Z_s^\veps)\rangle
-\langle\nabla\bar{f}(\bar{X}_s)\bar{Z}_s,\nabla\varphi(\bar{Z}_s)\rangle\d s\right|\\\nonumber
&
\leq\msI_{2,1}+\msI_{2,2}+\sum_{k=1}^N\msI_{2,3}^k,
\end{align}
where
\begin{equation*}
\msI_{2,1}:=\left|\mbE\left[\chi_{\Omega_2^\veps}\int_0^t\langle\nabla\bar{f}(\bar{X}_s+\iota(X_s^\veps-\bar{X}_s))Z_s^\veps,\nabla\varphi(Z_s^\veps)\rangle
-\langle\nabla\bar{f}(\bar{X}_s)\bar{Z}_s,\nabla\varphi(\bar{Z}_s)\rangle\d s\right]\right|,
\end{equation*}
\begin{equation*}
\msI_{2,2}:=\left|\mbE\left[\chi_{\Omega_3}\int_0^t\langle\nabla\bar{f}(\bar{X}_s+\iota(X_s^\veps-\bar{X}_s))Z_s^\veps,\nabla\varphi(Z_s^\veps)\rangle
-\langle\nabla\bar{f}(\bar{X}_s)\bar{Z}_s,\nabla\varphi(\bar{Z}_s)\rangle\d s\right]\right|,
\end{equation*}
\begin{equation*}
\msI_{2,3}^k:=\left|\mbE\left[\chi_{\widetilde{\Omega}_{1,k}^\veps}\int_0^t\langle\nabla\bar{f}(\bar{X}_s+\iota(X_s^\veps-\bar{X}_s))Z_s^\veps,\nabla\varphi(Z_s^\veps)\rangle
-\langle\nabla\bar{f}(\bar{X}_s)\bar{Z}_s,\nabla\varphi(\bar{Z}_s)\rangle\d s\right]\right|.
\end{equation*}

In view of \eqref{NDeq01},
Remark \ref{GRem}, the Burkholder-Davis-Gundy inequality and \eqref{PestAX}, for any $p>1$ there exists $m_1>0$ such that
\begin{align}\label{NDeq07}
&
\mbE\left|\int_0^t\langle\nabla\bar{f}(\bar{X}_s)\bar{Z}_s,
\nabla\varphi(\bar{Z}_s)\rangle\d s\right|^p
\\\nonumber
&
=\mbE\left|\varphi(\bar{Z}_t)-\varphi(0)-\int_0^t\frac12 Tr[\nabla^2\varphi(\bar{Z}_s)GG^T(\bar{X}_s)]\d s
-\int_0^t\langle\nabla\varphi(\bar{Z}_s),G(\bar{X}_s)\d \widetilde{W}_s^1\rangle\right|^p\\\nonumber
&
\leq C_{T,p,\|\nabla\varphi\|_\infty}\left(\|\varphi\|_\infty^p+\int_0^T\mbE|\bar{X}_s|^{m_1}\d s+1\right)\leq C_{p,T}.
\end{align}

It follows from H\"older's inequality, \eqref{NDeq04-4} and \eqref{NDeq07} that
\begin{align*}
&
\msI_{2,1}+\msI_{2,2}\\
&
\leq\mbE\left[\left(\chi_{\Omega_2^\veps}+\chi_{\Omega_3}\right)
\left|\int_0^t\langle
\nabla\bar{f}(\bar{X}_s+\iota(X_s^\veps-\bar{X}_s))Z_s^\veps,\nabla\varphi(Z_s^\veps)
\rangle\d s\right|\right]\\
&\quad
+\mbE\left[\left(\chi_{\Omega_2^\veps}+\chi_{\Omega_3}\right)\left|\int_0^t\langle
\nabla\bar{f}(\bar{X}_s)\bar{Z}_s,\nabla\varphi(\bar{Z}_s)\rangle\d s\right|\right]\\
&
\leq\left[\left(\mbP(\{\Omega_2^\veps\})\right)^{\frac13}+\left(\mbP(\{\Omega_3\})\right)^{\frac13}\right]
\left(\mbE\left|\int_0^t\langle\nabla\bar{f}(\bar{X}_s+\iota(X_s^\veps-\bar{X}_s))Z_s^\veps,
\nabla\varphi(Z_s^\veps)\rangle\d s\right|^{\frac32}\right)^{\frac23}\\
&\quad
+\left[\left(\mbP(\{\Omega_2^\veps\})\right)^{\frac{1}{2}}
+\left(\mbP(\{\Omega_3\})\right)^{\frac{1}{2}}\right]
\left(\mbE\left|\int_0^t\langle\nabla\bar{f}(\bar{X}_s)\bar{Z}_s,
\nabla\varphi(\bar{Z}_s)\rangle\d s\right|^{2}\right)^{\frac{1}{2}}\\
&
\leq C_T\delta^{\frac{1}{3}}\left(\mbE\int_0^T|\nabla\bar{f}(\bar{X}_s+\iota(X_s^\veps-\bar{X}_s))|^{\frac32}
|Z_s^\veps|^{\frac32}\|\nabla\varphi\|^{\frac32}_\infty\d s\right)^{\frac23}
+C_T\delta^{\frac{1}{2}},
\end{align*}
which by H\"older's inequality and {\bf (H$_x^2$)} implies that
\begin{align}\label{NDeq08}
\msI_{2,1}+\msI_{2,2}
&
\leq C_T\delta^{\frac13}\left(\mbE\int_0^T|Z_s^\veps|^{2}
+|\bar{X}_s|^{4\theta_1}
+|X_s^\veps|^{4\theta_1}\d s\right)^{\frac23}+C_T\delta^{\frac{1}{2}}.
\end{align}
Combining \eqref{NDeq08}, \eqref{PestX}, \eqref{PestAX}
and \eqref{Opcrsfeq}, we get
\begin{equation}\label{NDeq09}
\msI_{2,1}+\msI_{2,2}\leq C_T\delta^{\frac{1}{3}}.
\end{equation}

Note that for all $1\leq k\leq N$
\begin{align}
\msI_{2,3}^k:
&
\leq\left|\mbE\left[\chi_{\widetilde{\Omega}_{1,k}^\veps}\int_0^t\langle\left(\nabla\bar{f}(\bar{X}_s+\iota(X_s^\veps-\bar{X}_s))-\nabla\bar{f}(\bar{X}_s)\right)Z_s^\veps,\nabla\varphi(Z_s^\veps)\rangle
\d s\right]\right|\nonumber\\
&\quad
+\left|\mbE\left[\chi_{\widetilde{\Omega}_{1,k}^\veps}\int_0^t\langle\nabla\bar{f}(\bar{X}_s)Z_s^\veps,\nabla\varphi(Z_s^\veps)-\nabla\varphi(\bar{Z}_s)\rangle
\d s\right]\right|\label{NDeq10}\\\nonumber
&\quad
+\left|\mbE\left[\chi_{\widetilde{\Omega}_{1,k}^\veps}\int_0^t\langle\nabla\bar{f}(\bar{X}_s)(Z_s^\veps-\bar{Z}_s),\nabla\varphi(\bar{Z}_s)\rangle
\d s\right]\right|=:\sum_{i=1}^3\mfJ_i.
\end{align}
First of all, by H\"older's inequality, {\bf(H$_x^2$)},
Lemma \ref{Pest} and Theorem \ref{Opcrth} we have
\begin{align}
\mfJ_1
&
\leq\|\nabla\varphi\|_\infty\mbE\left[\chi_{\widetilde{\Omega}_{1,k}^\veps}\int_0^t\left|\nabla^2\bar{f}(\bar{X}_s+\iota\iota'(X_s^\veps-\bar{X}_s))\right|\left|X_s^\veps-\bar{X}_s\right||Z_s^\veps|\|\nabla\varphi\|_\infty\d s\right]\nonumber\\
&
\leq\|\nabla\varphi\|_\infty\left(\mbE\int_0^T|Z_s^\veps|^2\d s\right)^{\frac12}
\left(\mbE\int_0^T|X_s^\veps-\bar{X}_s|^2\d s\right)^{\frac14}\label{NDeq11}\\\nonumber
&\qquad
\times\left(\mbE\int_0^T|\nabla^2\bar{f}(\bar{X}_s+\iota\iota'(X_s^\veps-\bar{X}_s))|^{4}
\left(|X_s^\veps|+|\bar{X}_s|\right)^{2}\d s\right)^{\frac14}
\leq C_{T}\veps^{\frac14},
\end{align}
where $\iota'\in[0,1]$.
And in view of H\"older's inequality  we have
\begin{align}
\mfJ_2+\mfJ_3
&
\leq\mbE\left[\chi_{\widetilde{\Omega}_{1,k}^\veps}\int_0^t\left|\nabla\bar{f}(\bar{X}_s)\right|(|Z_s^\veps|\|\nabla^2\varphi\|_\infty+\|\nabla\varphi\|_\infty)|Z_s^\veps-\bar{Z}_s|
\d s\right]\label{NDeq12}\\\nonumber
&
\leq C\delta\mbE\left[\chi_{\widetilde{\Omega}_{1,k}^\veps}\int_0^T|\nabla\bar{f}(\bar{X}_s)|(|Z_s^\veps|+1)\d s\right].
\end{align}

Therefore, by \eqref{NDeq06}, \eqref{NDeq09}, \eqref{NDeq10}, \eqref{NDeq11}, \eqref{NDeq12},
H\"older's inequality, {\bf(H$_x^2$)}, Lemma \ref{Pest} and Theorem \ref{Opcrth} one sees that
\begin{align*}
\msI_2
&
\leq C_T\left(\delta^{\frac13}+N\veps^{\frac14}\right)
+\delta C\mbE\int_0^T|\nabla\bar{f}(\bar{X}_s)|\left(|Z_s^\veps|+1\right)\d s
\leq C_T\left(\delta^{\frac14}+N\veps^{\frac14}\right),
\end{align*}
which implies that
$
\lim_{\veps\rightarrow0}\msI_2=0
$
by first letting $\veps\rightarrow0$ and then letting $\delta\rightarrow0$.

{\bf{(Step 3)}}
Now we estimate $\msI_3$:
\begin{align}
\msI_3:
&
=\left|\mbE\int_0^t\frac12 Tr[\nabla^2\varphi(Z_s^\veps)GG^T(X_s^\veps)]
-\frac12 Tr[\nabla^2\varphi(\bar{Z}_s)GG^T(\bar{X}_s)]\d s\right|\nonumber\\\nonumber
&
\leq\left|\mbE\left[\chi_{\Omega_2^\veps}\int_0^t\frac12 Tr[\nabla^2\varphi(Z_s^\veps)GG^T(X_s^\veps)]
-\frac12 Tr[\nabla^2\varphi(\bar{Z}_s)GG^T(\bar{X}_s)]\d s\right]\right|\\
&\quad
+\left|\mbE\left[\chi_{\Omega_3}\int_0^t\frac12 Tr[\nabla^2\varphi(Z_s^\veps)GG^T(X_s^\veps)]
-\frac12 Tr[\nabla^2\varphi(\bar{Z}_s)GG^T(\bar{X}_s)]\d s\right]\right|\label{NDeq13}\\\nonumber
&\quad
+\sum_{k=1}^N\left|\mbE\left[\chi_{\widetilde{\Omega}_{1,k}^\veps}\int_0^t\frac12 Tr[\nabla^2\varphi(Z_s^\veps)GG^T(X_s^\veps)]
-\frac12 Tr[\nabla^2\varphi(\bar{Z}_s)GG^T(\bar{X}_s)]\d s\right]\right|\\\nonumber
&
=:\msI_{3,1}+\msI_{3,2}+\sum_{k=1}^N\msI_{3,3}^k.
\end{align}
Thanks to H\"older's inequality, \eqref{NDeq04-4}, \eqref{NDeq05}, Remark \ref{GRem},
\eqref{PestX} and \eqref{PestAX}, we have
\begin{align}\label{NDeq14}
\msI_{3,1}+\msI_{3,2}\leq C_T\delta^{1/2}.
\end{align}
And it follows from Remark \ref{GRem}, H\"older's inequality, Lemma \ref{Pest}
and Theorem \ref{Opcrth} that there exists $m>0$ such that
\begin{align}
\msI_{3,3}^k:
&
\leq \left|\mbE\left[\chi_{\widetilde{\Omega}_{1,k}^\veps}\int_0^t\frac12 Tr\left[\left(\nabla^2\varphi(Z_s^\veps)-\nabla^2\varphi(z_s^k)\right)GG^T(X_s^\veps)\right]\d s\right]\right|\nonumber\\\nonumber
&\quad
+\left|\mbE\left[\chi_{\widetilde{\Omega}_{1,k}^\veps}\int_0^t\frac12 Tr[\nabla^2\varphi(z_s^k)GG^T(X_s^\veps)]
-\frac12 Tr[\nabla^2\varphi(z_s^k)GG^T(\bar{X}_s)]\d s\right]\right|\\\nonumber
&\quad
+\left|\mbE\left[\chi_{\widetilde{\Omega}_{1,k}^\veps}\int_0^t\frac12 Tr\left[\left(\nabla^2\varphi(\bar{Z}_s)-\nabla^2\varphi(z_s^k)\right)GG^T(\bar{X}_s)\right]\d s\right]\right|\\
&
\leq C\left|\mbE\left[\chi_{\widetilde{\Omega}_{1,k}^\veps}\int_0^t
\|\nabla^3\varphi\|_\infty|Z_s^\veps-z_s^k|\left(1+|X_s^\veps|^{m}\right)
\d s\right]\right|\label{NDeq15}\\\nonumber
&\quad
+\left|\mbE\left[\chi_{\widetilde{\Omega}_{1,k}^\veps}\int_0^t
\|\nabla^2\varphi\|_\infty\left(1+|X_s^\veps|^{m}+|\bar{X}_s|^{m}\right)|X_s^\veps-\bar{X}_s|\d s\right]\right|\\\nonumber
&\quad
+\left|\mbE\left[\chi_{\widetilde{\Omega}_{1,k}^\veps}\int_0^t
\|\nabla^3\varphi\|_\infty|\bar{Z}_s-z_s^k|\left(1+|\bar{X}_s|^m\right)\d s\right]\right|\\\nonumber
&
\leq C_T\delta\left(\mbE\left[\chi_{\widetilde{\Omega}_{1,k}^\veps}\int_0^T
\left(1+|X_s^\veps|^{m}+|\bar{X}_s|^{m}\right)\d s\right]\right)+C_T\veps^{\frac12}.
\end{align}

Therefore, employing \eqref{NDeq13}, \eqref{NDeq14}, \eqref{NDeq15},
H\"older's inequality, {\bf(H$_x^2$)} and Lemma \ref{Pest}, we have
\begin{align*}
\msI_3
&
\leq C_T\left(\delta^{1/2}+N\veps^{\frac12}\right)+C_T\delta
\left(\mbE\int_0^T\left(1+|X_s^\veps|^{m}+|\bar{X}_s|^{m}\right)\d s\right)
\leq C_T(\delta^{1/2}+N\veps^{1/2}).
\end{align*}
Therefore, firstly letting $\veps\rightarrow0$, then taking $\delta\rightarrow0$, we obtain
$
\lim_{\veps\rightarrow0}\msI_3=0.
$
\end{proof}

\section{The global averaging principle and second averaging principle}

In Section \ref{PDS}, we begin with recalling some well-known definitions and results about
autonomous and nonautonomous dynamical systems (see e.g. \cite{KR2011}). Subsequently,
we prove Theorem \ref{GAPth} in Section \ref{PGAPth}.
Finally, we investigate the second averaging principle in Section \ref{PSAPth}.

\subsection{Preliminaries for dynamical systems}\label{PDS}
Let $(\mcX,d)$ be a complete metric space, and
$(\mcP,d_{\mcP})$ be a metric space.

\begin{definition}\rm
Let $T=\R$ or $\Z$.
A {\em semi-dynamical system} is defined as a continuous function $\phi:T^+\times\mcX\rightarrow\mcX$ that satisfies
$\phi(0,x)=x,~\forall x\in\mcX$,
and $\phi(t+s,x)=\phi(t,\phi(s,x))$ for all $s,t\in T^+$ and $x\in\mcX$.
\end{definition}

\begin{definition}\rm
We say a nonempty compact subset $A\subset\mcX$ is a {\em global attractor} of a semi-dynamical system $\phi$ on $\mcX$ if it is $\phi$-invariant and attracts bounded sets, i.e.
$\phi(t,A)=A$ for all $t\in T^+$
and
$
\lim_{t\rightarrow\infty}{\rm dist}\left(\phi(t,D),A\right)=0$
for any bounded subset $D\subset\mcX$.
\end{definition}

\begin{definition}\label{DefNDY}\rm
A {\em nonautonomous dynamical system}
$\left(\sigma,\varphi\right)$ (in short, $\varphi$) comprises two components:
\begin{enumerate}
  \item A {\em dynamical system} $\sigma$ on $\mathcal P$ with time
      set $T=\mathbb Z$ or $\R$, i.e.
$\sigma_{0}(\cdot)=Id_{\mathcal P}$,
$\sigma_{t+s}(p)=\sigma_{t}
(\sigma_{s}(p))$ for all $t,s\in T$
and $p\in \mathcal P$, and the mapping
          $(t,p)\mapsto\sigma_{t}(p)$ is continuous.
  \item
  A {\em cocycle} $\varphi:T^{+}\times\mathcal P\times \mathcal X
        \rightarrow \mathcal X$ satisfies
        \begin{enumerate}
          \item [(1)] $\varphi(0,p,x)=x$ for all
          $(p,x)\in \mathcal P\times\mathcal X$,
          \item [(2)] $\varphi(t+s,p,x)=\varphi(t,\sigma_{s}(p),
          \varphi(s,p,x))$ for all $s,t\in T^{+}$
          and $(p,x)\in\mathcal P\times\mathcal X$,
          \item [(3)] the mapping
          $(t,p,x)\mapsto\varphi(t,p,x)$ is continuous.
        \end{enumerate}
\end{enumerate}
Here $\mathcal P$ is called the {\em base} or {\em parameter space}
and $\mathcal X$ is the {\em fiber} or {\em state space}.
For convenience, we also write $\sigma_{t}(p)$ as $\sigma_{t}p$.

Furthermore, if $\sigma_t(p)(\cdot):=p(t+\cdot)$ then $(\mcP,\R,\sigma)$ is
called a shift dynamical system or Bebutov shift flow.
\end{definition}


\begin{definition}\rm
The autonomous semi-dynamical system $\pi$ on $\mcP\times\mcX$ defined by
\[
\pi:T^+\times\mcP\times\mcX\rightarrow\mcP\times\mcX,
\quad
\pi(t,(p,x)):=\left(\theta_t(p),\varphi(t,p,x)\right)
\]
is called the {\em skew product flow} associated with the nonautonomous dynamical system
$(\theta,\varphi)$.
\end{definition}

\begin{definition}\rm
Let $(\sigma,\phi)$ be a skew product flow on a metric phase space $(\mcX,d)$
with base set $\mcP$. A subset $\mcM$ of the extended phase
space $\mcP\times\mcX$ is called a {\em nonautonomous set}, and for each $p\in\mcP$, the set
$
M_p:=\{x\in\mcX:(p,x)\in\mcM\}
$
is called the $p$-{\em fiber} of $\mcM$. In general, $\mcM$ is said to have a
topological property (such as compactness or closeness) if each fiber of $\mcM$ has this property.
\end{definition}

\begin{definition}\rm
Let $(\sigma,\phi)$ be a skew product flow on a metric space $(\mcX,d)$ with base set $\mcP$.
A family $\mcA=\{A_p\}_{p\in\mcP}$ of nonempty subsets of $\mcX$ is called {\em invariant}
w.r.t. $(\sigma,\phi)$, or $\phi$-{\em invariant},
if
$\phi(t,p,A_{p})=A_{\sigma_t(p)}$ for all
$t\geq0$ and $p\in\mcP$.
\end{definition}

\begin{remark}\rm
The compact set-valued mapping $t\mapsto A_{\sigma_t(p)}$, induced by a $\phi$-invariant family
$(A_p)_{p\in\mcP}$ of compact subsets, is continuous in $t\in\R$ w.r.t. the Hausdorff
metric for each fixed $p\in\mcP$.
\end{remark}

\begin{definition}\label{PullbackA}\rm
Let $(\sigma,\phi)$ be a skew product flow. A nonempty, compact and invariant nonautonomous
set $\mcA$ is called a {\em pullback attractor} if the pullback convergence
\[
\lim_{t\rightarrow\infty}{\text{dist}}_\mcX(\phi(t,\sigma_{-t}(p),D),A_p)=0
\]
holds for every nonempty bounded subset $D$ of $\mcX$ and $p\in\mcP$.
\end{definition}

\begin{lemma}\label{globalALem}
Assume that a semi-dynamical system $\phi$ on $\mcX$ has an absorbing set $B$, i.e. for any bounded subset $D\subset\mcX$, there exists $T=T(D)\in T^+$ such that $\phi(t,D)\subset B$ for all $t\geq T$.
If $B$ is positively invariant, i.e. $\phi(t,B)\subset B$ for all $t\in T^+$, then $\phi$ has a unique attractor $A\subset B$ defined by
$
A=\bigcap_{t\geq0}\phi(t,B).
$

\end{lemma}

\begin{definition}\rm
Let $(\sigma,\phi)$ be a  skew product flow on $\mcX$. A nonempty compact subset $B$
of $\mcX$ is called {\em pullback absorbing} if for each $p\in\mcP$ and every bounded
subset $D$ of $\mcX$, there exists a $T=T(p,D)>0$ such that
$\phi(t,\sigma_{-t}(p),D)\subset B$
for all $t\geq T$.
\end{definition}

\begin{lemma}\label{pullALem}
Let $(\sigma,\phi)$ be a skew product flow on $\mcX$ with a compact pullback absorbing
set $B$ such that
$
\phi(t,p,B)\subset B
$
for all $t\geq0$ and $p\in\mcP$. Then there exists a pullback attractor $\mcA$ with fibers in $B$ uniquely
determined by
$
A_p=\bigcap_{\tau\geq0}\overline{\bigcup_{t\geq\tau}\phi(t,\sigma_{-t}(p),B)}
$
for all $p\in\mcP$.

\end{lemma}

\subsection{Proof of Theorem \ref{GAPth}}\label{PGAPth}
In order to prove Theorem \ref{GAPth}, we need the following decay estimates of
solutions to \eqref{sfmain} and \eqref{avemain}
under the dissipativity condition {\bf (H$_x^5$)}.

\begin{lemma}\label{BouXY}
Assume that {\bf(H$_y^1$)}, {\bf(H$_y^3$)} and {\bf(H$_x^2$)}--{\bf(H$_x^5$)} hold.
If $g\in C_b(\R^{d_1+d_2})$, then for any  $(\zeta^\veps,\xi^\veps)\in\mcL^2(\Omega,\msF_s,\mbP;\R^{d_1+d_2})$
there exists a unique solution $V_{s,t}^\veps(\zeta^\veps,\xi^\veps):=(X_{s,t}^\veps(\zeta^\veps),Y_{s,t}^\veps(\xi^\veps))$ to \eqref{eqsys}.
Furthermore, there exists $\veps_0>0$ such that for all
$p\geq1$ and $0<\veps\leq\veps_0$
\begin{align}\label{BouY2p}
\mbE|Y_{s,t}^\veps(\xi^\veps)|^{2p}\leq
&
\mbE|\xi^\veps|^{2p}{\rm e}^{-\frac14\veps^{-2\alpha}\eta p(t-s)}
+\varpi(p)\left(1-{\rm e}^{-\frac14\veps^{-2\alpha}\eta p(t-s)}\right),
\end{align}
and
\begin{align*}
\mbE|X_{s,t}^\veps(\zeta^\veps)|^2
&
\leq \mbE|\zeta^\veps|^2{\rm e}^{-\lambda_1(t-s)}
+\frac{K_5\varpi(\theta/2)+K_4}{\lambda_1}\left(1-{\rm e}^{-\lambda_1(t-s)}\right)\\\nonumber
&\quad
+K_5\mbE|\xi^\veps|^\theta\left(\left[\lambda_1^{-1}\left(1-{\rm e}^{-\lambda_1(t-s)}\right)\right]\wedge {\rm e}^{-\lambda_1(t-s)}\right),
\end{align*}
where $\varpi(p):=\left(8K_1+8(p-1)\|g\|_\infty\right)^p\eta^{-p},~\forall p\geq1$.

Moreover, for any $1<p<\frac{\lambda_1}{2L_\sigma^2}+1$ there exists a constant $C>0$ such that
\begin{align}\label{BouX2p}
\mbE|X_{s,t}^\veps(\zeta^\veps)|^{2p}
&
\leq \mbE|\zeta^\veps|^{2p}{\rm e}^{-\frac p2(\lambda_1-2(p-1)L_\sigma^2)(t-s)}+C(\mbE|\xi^\veps|^{\theta p}+1).
\end{align}
\end{lemma}
\begin{proof}
By \eqref{PestYB} and Gronwall's inequality,  one sees that there exists a $\veps_0>0$ such that for all $0<\veps\leq\veps_0$
\begin{align*}
\mbE|Y_{s,t}^\veps(\xi^\veps)|^{2p}
&
\leq\mbE|\xi^\veps|^{2p}{\rm e}^{-\frac14\veps^{-2\alpha}\eta p(t-s)}\\
&\quad
+\left(8K_1+8(p-1)\|g\|_\infty\right)^p\eta^{-p}\left(1-{\rm e}^{-\frac14\veps^{-2\alpha}\eta p(t-s)}\right).
\end{align*}

It follows from It\^o's formula, {\bf(H$_x^3$)}, {\bf(H$_x^5$)} and Young's inequality that
\begin{align*}
&
\mbE|X_{s,t}^\veps(\zeta^\veps)|^{2p}\\
&
\leq \mbE|\zeta^\veps|^{2p}+p\mbE\int_s^t|X_{s,\tau}^\veps(\zeta^\veps)|^{2p-2}
\left(-\lambda_1|X_{s,\tau}^\veps(\zeta^\veps)|^2
+K_5|Y_{s,\tau}^\veps(\xi)|^\theta+K_4\right)\d\tau\\
&\quad
+2p(p-1)\mbE\int_s^t|X_{s,\tau}^\veps(\zeta^\veps)|^{2p-2}
\left(L_\sigma^2|X_{s,\tau}^\veps(\zeta^\veps)|^2+K_7^2+2L_\sigma K_7|X_{s,\tau}^\veps(\zeta^\veps)|\right)\d\tau.
\end{align*}

If $p=1$, in view of Gronwall's inequality and \eqref{BouY2p}, one sees that there exists $\veps_0>0$ such that for all $0<\veps\leq\veps_0$
\begin{align*}
\mbE|X_{s,t}^\veps(\zeta^\veps)|^2
&
\leq \mbE|\zeta^\veps|^2{\rm e}^{-\lambda_1(t-s)}+\frac{1}{\lambda_1}\left(K_5\left(\mbE|\xi^\veps|^{\theta}+
\varpi(\theta/2)\right)+K_4\right)\left(1-{\rm e}^{-\lambda_1(t-s)}\right).
\end{align*}
On the other hand, in view of Gronwall's inequality and \eqref{BouY2p}, we obtain that there exists $\veps_0>0$ such that for all $0<\veps\leq\veps_0$
\begin{align*}
\mbE|X_{s,t}^\veps(\zeta^\veps)|^2
&
\leq \mbE|\zeta^\veps|^2{\rm e}^{-\lambda_1(t-s)}+\frac{K_5\varpi(\theta/2)+K_4}{\lambda_1}\left(1-{\rm e}^{-\lambda_1(t-s)}\right)\\
&\quad
+K_5\mbE|\xi^\veps|^{\theta}{\rm e}^{-\lambda_1(t-s)}\left(\frac{1}{8}\veps^{-2\alpha\eta\theta}-\lambda_1\right)^{-1}\left(1-{\rm exp}\{(\lambda_1-
\frac{\eta\theta}{8\veps^{2\alpha}}
(t-s)\}\right)\\
&
\leq \mbE|\zeta^\veps|^2{\rm e}^{-\lambda_1(t-s)}
+K_5\mbE|\xi^\veps|^{\theta}{\rm e}^{-\lambda_1(t-s)}
+\frac{K_5\varpi(\theta/2)+K_4}{\lambda_1}\left(1-{\rm e}^{-\lambda_1(t-s)}\right).
\end{align*}

In the case where $1<p<\frac{\lambda_1}{2L_\sigma^2}+1$, by Young's inequality we have
\begin{align*}
&
\mbE|X_{s,t}^\veps(\zeta^\veps)|^{2p}\\
&
\leq \mbE|\zeta^\veps|^{2p}+p\mbE\int_s^t\left(
-\frac12(\lambda_1-2(p-1)L_\sigma^2)|X_{s,\tau}^\veps(\zeta^\veps)|^{2p}
+C|Y_{s,\tau}^\veps(\xi^\veps)|^{\theta p}
+C\right)\d\tau,
\end{align*}
which by Gronwall's inequality and \eqref{BouY2p} implies that
\begin{align*}
\mbE|X_{s,t}^\veps(\zeta^\veps)|^{2p}
&
\leq \mbE|\zeta^\veps|^{2p}{\rm e}^{-\frac p2(\lambda_1-2(p-1)L_\sigma^2)(t-s)}+C(\mbE|\xi^\veps|^{\theta p}+1).
\end{align*}
\end{proof}

\begin{lemma}
Assume that {\bf{(H$_y^1$)}}, {\bf{(H$_y^2$)}}, {\bf{(H$_x^5$)}} and {\bf{(A$_f$)}} hold.
Then there exists a constant $K_8>0$ such that for all $x\in\R^{d_1}$
\begin{equation*}
2\langle \bar{f}(x),x\rangle+|\bar{\sigma}(x)|^2
\leq -\lambda_1|x|^2+K_8.
\end{equation*}
\end{lemma}
\begin{proof}
Let
$
K':=\sup_{x\in\R^{d_1}}\int_{\R^{d_2}}|y|^{\theta}\mu^x(\d y)<\infty.
$
It follows from {\bf(H$_x^3$)}, {\bf(H$_x^5$)}, H\"older's inequality and Young's inequality that for any $x\in\R^{d_1}$
\begin{align*}
&
2\langle \bar{f}(x),x\rangle+|\bar{\sigma}(x)|^2_{HS}\\\nonumber
&
=2\left\langle \bar{f}(x)-\frac{1}{T}\int_0^T\hat{f}(s,x)\d s,x\right\rangle
+2\left\langle \frac{1}{T}\int_0^T\hat{f}(s,x)\d s,x\right\rangle
+\left|\frac{1}{T}\int_0^T\sigma(s,x)\d s\right|^2\\\nonumber
&\quad
+\left|\bar{\sigma}(x)-\frac{1}{T}\int_0^T\sigma(s,x)\d s\right|^2
+2\left\langle\bar{\sigma}(x)-\frac{1}{T}\int_0^T\sigma(s,x)\d s,\frac{1}{T}\int_0^T\sigma(s,x)\d s\right\rangle\\\nonumber
&
\leq \frac{2}{T}\left|\int_0^T\left(\bar{f}(x)-\hat{f}(s,x)\right)\d s\right|\cdot|x|
+\frac{1}{T}\int_0^T\int_{\R^{d_2}}\left(2\left\langle f(s,x,y),x\right\rangle+|\sigma(s,x)|^2\right)\mu^x(\d y)\d s\\\nonumber
&\quad
+\frac{1}{T}\int_0^T|\bar{\sigma}(x)-\sigma(s,x)|^2\d s
+2\left|\bar{\sigma}(x)-\frac{1}{T}\int_0^T\sigma(s,x)\d s\right|\left|\frac{1}{T}\int_0^T\sigma(s,x)\d s\right|\\\nonumber
&
\leq -\lambda_1|x|^2+K_4+K'K_5+1
+\frac{2}{T}\left|\int_0^T\left(\bar{f}(x)-\hat{f}(s,x)\right)\d s\right|\cdot|x|\\\nonumber
&\quad
+\frac{1}{T}\int_0^T|\bar{\sigma}(x)-\sigma(s,x)|^2\d s
+\frac{1}{T}\int_0^T|\bar{\sigma}(x)-\sigma(s,x)|^2\d s
\left(2L_\sigma^2|x|^2+2K_7^2\right),
\end{align*}
which by {\bf{(A$_f$)}} and {\bf{(A$_\sigma$)}} implies that
\[
2\langle \bar{f}(x),x\rangle+|\bar{\sigma}(x)|^2_{HS}
\leq-\lambda_1|x|^2+K_4+K'K_5+1
\]
by letting $T\rightarrow\infty$.
\end{proof}

Similarly to the estimates provided in Lemma \ref{BouXY}, the following lemma can be derived by applying the same methodology. Therefore, we omit the proof.
\begin{lemma}\label{BouAX}
Assume that {\bf(H$_y^1$)}, {\bf(H$_y^3$)} and {\bf(H$_x^2$)}--{\bf(H$_x^5$)} hold.
If $g\in C_b(\R^{d_1+d_2})$, then for any $\zeta\in\mcL^2(\Omega,\msF_s,\mbP;\R^{d_1})$,
there exists a unique solution $\bar{X}_{s,t}(\zeta)$ to \eqref{avemain}.
Furthermore, there exists $\veps_0>0$ such that for all $0<\veps\leq\veps_0$
\begin{equation*}\label{BouAX1}
\mbE|\bar{X}_{s,t}(\zeta)|^2
\leq \mbE|\zeta|^2{\rm e}^{-\lambda_1(t-s)}
+\frac{K_8}{\lambda_1}\left(1-{\rm e}^{-\lambda_1(t-s)}\right).
\end{equation*}

Moreover, if $\zeta\in\mcL^{2p}(\Omega,\msF_s,\mbP;\R^{d_1})$
for any $1<p<\frac{\lambda_1} {2L_\sigma^2}+1$, then there exists a constant $C>0$ such that
\begin{align*}\label{BouAX2}
\mbE|\bar{X}_{s,t}(\zeta)|^{2p}
&
\leq \mbE|\zeta|^{2p}{\rm e}^{-\frac p2(\lambda_1-2(p-1)L_\sigma^2)(t-s)}+C.
\end{align*}
\end{lemma}

\begin{remark}\rm
Assume that $f$ and $\sigma$ satisfy {\bf (A$_f$)} and {\bf (A$_\sigma$)}.
Note that by \cite[Lemma 5.10]{CL2023} one sees that for any
$\widetilde{\mbF}:=\left(\widetilde{F},\widetilde{\sigma}\right)\in\mcH(\mbF)$,
$\widetilde{F}$ and $\widetilde{\sigma}$ satisfy {\bf (A$_f$)} and {\bf (A$_\sigma$)} provided $\mcH(\mbF)$ is compact.
\end{remark}

Following a similar approach as in the proof of \cite[Theorem 3.1]{CL2021},
we establish the following lemma for the continuous dependence of solutions to \eqref{eqsys}
on initial values and coefficients. For convenience,
we assume $\veps=1$ in this lemma without loss of generality.
\begin{lemma}\label{SICCD}
Assume that $\xi^n,\xi\in\mcL^{2}(\Omega,\msF_s,\mbP;\R^{d_1+d_2})$.
For any $n\in\N$
let $V_{s,t}^n(\xi^n)$ satisfy
\begin{equation*}
V_{s,t}^n(\xi^n)=\xi^n+\int_s^tF^n(r,V_{s,r}^n(\xi^n))\d r+\int_s^tG^n(r,V_{s,r}^n(\xi^n))\d W_r, \quad t\geq s,
\end{equation*}
and $V_{s,t}(\xi)$ satisfy
\begin{equation*}
 V_{s,t}(\xi)=\xi+\int_s^tF(r,V_{s,r}(\xi))\d r+\int_s^tG(r,V_{s,r}(\xi))\d W_r, \quad t\geq s.
\end{equation*}
Suppose that $F^n$, $G^n$, $F$ and $G$ satisfy
{\bf(H$_y^1$)}, {\bf(H$_y^3$)} and {\bf(H$_x^2$)}--{\bf(H$_x^5$)}.
Furthermore, assume that
$\lim_{n\rightarrow\infty}F^n(t,v)=F(t,v)$
and
$\lim_{n\rightarrow\infty}G^n(t,v)=G(t,v)$
for all $t\in\R$ and $v\in\R^{d_1+d_2}$.
If $\lim_{n\rightarrow\infty}d_{BL}\left(\msL(\xi^n),\msL(\xi)\right)=0$, then
for any $t\geq s$
\[
\lim_{n\rightarrow\infty}\sup_{r\in[s,t]}d_{BL}\left(\msL\left(V_{s,r}^n(\xi^n)\right),\msL\left(V_{s,r}(\xi)\right)\right)=0.
\]
\end{lemma}

\begin{lemma}\label{SPFth}
Suppose that {\bf(H$_y^1$)}, {\bf(H$_y^3$)} and {\bf(H$_x^2$)}--{\bf(H$_x^5$)} hold. If $g\in C_b(\R^{d_1+d_2})$, then
for any $0<\veps\leq1$,
$(\sigma,P^*_\veps)$ is a skew product flow on the phase space $(\msP_{2,\theta}(\R^{d_1+d_2}),d_{BL})$
with base space $\mcH(\mbF_\veps)$, where $\sigma:\R_+\times\mcH(\mbF_\veps)\rightarrow\mcH(\mbF_\veps)$ is defined by
$\sigma_t(\widetilde{\mbF}_\veps):=\widetilde{\mbF}_\veps(t+\cdot,\cdot)$ for all
$(t,\widetilde{\mbF}_\veps)\in\R_+\times\mcH(\widetilde{\mbF}_\veps)$.
\end{lemma}
\begin{proof}
By Remark \ref{Rem0406} and Lemma \ref{BouXY}, one sees that
$$P^*_\veps(t,\widetilde{\mbF}_\veps,\cdot):\msP_{2,\theta}(\R^{d_1+d_2})
\rightarrow\msP_{2,\theta}(\R^{d_1+d_2})$$
for any $t\geq0$, $\widetilde{\mbF}_\veps\in\mcH(\mbF_\veps)$ and $0<\veps\leq\veps_0$.
Employing Lemma \ref{SICCD}, the uniqueness in law of the solutions to \eqref{eqsys} follows,
which completes the proof.
\end{proof}

With the help of the aforementioned results, we are now in a position to prove Theorem \ref{GAPth}.
\begin{proof}[Proof of Theorem \ref{GAPth}]
(i)
For any $R>0$ and $r\geq1$, define
$$\mcD_R^{1,r}:=\left\{\mu\in \msP_{r}(\R^{d_1}):\int_{\R^{d_1}}|x|^{r}\mu(\d x)\leq R\right\},$$
$$\mcD_R^{2,r}:=\left\{\mu\in\msP_{r}(\R^{d_2}):\int_{\R^{d_2}}|y|^{r}\mu(\d y)\leq R\right\}.$$
Set
$
\varpi_1:=\lambda_1^{-1}\left(K_5\left(2\varpi(\theta/2)+1\right)+K_4\right)+1
$
and
$\varpi_2:=\varpi(\theta/2)+1.
$
Let $\varpi:=\varpi_1+\varpi_2$. Define
\begin{equation}\label{PGAeq0}
B:=\overline{\left\{m\in \msP_{2,\theta}(\R^{d_1+d_2}):
\int_{\R^{d_1+d_2}}|z|^2 m(\d z)\leq\varpi,
m\circ\pi_1^{-1}\in\mcD_{\varpi_1}^{1,2},
m\circ\pi_2^{-1}\in\mcD_{\varpi_2}^{2,\theta}\right\}}.
\end{equation}
It can be verified that $B$ is compact.
By Lemma \ref{BouXY}, one sees that $B\subset\msP_{2,\theta}(\R^{d_1+d_2})$
is a pullback absorbing set such that
$
P^*_\veps(t,\widetilde{\mbF}_\veps,B)\subset B
$
for all $t\geq0$ and $\widetilde{\mbF}_\veps\in\mcH(\mbF_\veps)$.
It follows from Lemma \ref{pullALem} that $(\sigma,P^*_\veps)$ has a
pullback attractor $\msA^\veps$ with component subsets
$$\msA_{\widetilde{\mbF}_\veps}:=\bigcap_{\tau\geq0}
\overline{\bigcup_{t\geq\tau}P^*_\veps(t,\sigma_{-t}(\widetilde{\mbF}_\veps),B)}\subset B,\quad
\widetilde{\mbF}_\veps\in\mcH(\mbF_\veps).$$

(ii)
Let
$$
B_1:=\overline{\left\{\mu\in\msP_2(\R^{d_1}):\int_{\R^{d_1}}|x|^2\mu(\d x)\leq K_8\lambda_1^{-1}+1\right\}}.
$$
Employing Lemma \ref{BouAX}, we show that $B_1$ is a positively invariant and absorbing set. Thanks to Lemma \ref{globalALem}, $\bar{P}^*$ admits a global attractor $\bar{\mcA}$, defined by
\[
\bar{\mcA}:=\bigcap_{t\geq0}\bar{P}^*(t,B_1).
\]

(iii)
For any $\delta>0$ and bounded subset $D\subset \msP_2(\R^{d_1})$,
since $\bar{\mcA}$ is the attractor of $\bar{P}^*$, there exists a $T>0$ such that
for all $t\geq T$
\begin{equation}\label{PGAeq1}
\bar{P}^*(t,D)\subset\mcO_{\delta/2}\left(\bar{\mcA}\right).
\end{equation}
In view of \eqref{FAPord}, we have
\begin{equation}\label{PGAeq2}
\sup_{0\leq t\leq 2T}
d(\Pi_1 P^*_\veps(t,\widetilde{\mbF}_\veps,m),\bar{P}^*(t,m\circ\pi_1^{-1}))<\eta(T,B)(\veps)
\end{equation}
for all $m\in B$ and $\widetilde{\mbF}_\veps\in\mcH(\mbF_\veps)$.
And there exists $\veps_0>0$ such that
$\eta(T,D)(\veps)<\delta/2$ for all $0<\veps\leq\veps_0$.

It follows from \eqref{PGAeq1} and \eqref{PGAeq2} that
\begin{equation*}
\bigcup_{\widetilde{\mbF}_\veps\in\mcH(\mbF_\veps)}
\Pi_1P^*_\veps(t,\widetilde{\mbF}_\veps,B)
\subset\mcO_{\delta}\left(\bar{\mcA}\right)
\end{equation*}
for all $T\leq t\leq2T$.
Taking some $t_0\in[T,2T]$, then we have
\begin{equation*}
\bigcup_{\widetilde{\mbF}_\veps\in\mcH(\mbF_\veps)}P^*_\veps(t_0,\sigma_{-t_0}(\widetilde{\mbF}_\veps),B)
\subset\mcO_\delta(\bar{\mcA}).
\end{equation*}

It follows from the $P^*_\veps$-invariance that
$P^*_\veps(t_0,\sigma_{-t_0}(\widetilde{\mbF}_\veps),\msA_{\sigma_{-t_{0}}(\widetilde{\mbF}_\veps)})
=\msA_{\widetilde{\mbF}_\veps},
$
which implies that for all $0<\veps\leq\veps_0$ and $\widetilde{\mbF}_\veps\in\mcH(\mbF_\veps)$
\begin{equation*}
\msA_{\widetilde{\mbF}_\veps}\subset \mcO_\delta(\bar{\mcA}),
\end{equation*}
because $\msA_{\sigma_{-t_0}(\widetilde{\mbF}_\veps)}\subset B$.
\end{proof}

\subsection{Proof of Theorem \ref{SAPth}}\label{PSAPth}
In this subsection, we will give the proof of the second Bogolyubov theorem.
To this end, we shall show the uniqueness and existence of bounded solutions to \eqref{sfmain}.
We say that the solution $V_t^\veps:=(X_t^\veps,Y_t^\veps),t\in\R$, of \eqref{sfmain} is
$\mcL^{2p}\left(\Omega,\mbP;\R^{d_1+d_2}\right)$-{\em bounded} if
$
\sup_{t\in\R}\mbE|V_t^\veps|^{2p}<\infty.
$

\begin{proposition}\label{BouSth}
Assume that $B(x,y)\equiv B(y)$ for all $(x,y)\in\R^{d_1+d_2}$ and $\beta\leq\alpha$.
Furthermore, suppose that {\bf(H$_y^1$)}, {\bf(H$_y^3$)}, {\bf(H$_y^6$)}
and {\bf(H$_x^2$)}--{\bf(H$_x^6$)}hold.
If $\beta=\alpha$ and $\lambda_1>\frac{L_b^2}{\eta}$,
then there exists $\veps_0>0$ such that for all $0<\veps\leq\veps_0$
\begin{equation}\label{ASSS1}
\mbE|V_{s,t}^{\veps}(\xi_1)-V_{s,t}^{\veps}(\xi_2)|^2
\leq\mbE|\xi_1-\xi_2|^2{\rm exp}\left(-\left(\frac{\lambda_1}{2}-\frac{L_b^2}{2\eta}\right)(t-s)\right).
\end{equation}
If $\beta<\alpha$ then there exists $\veps_0>0$ such that for all $0<\veps\leq\veps_0$
\begin{equation}\label{ASSS2}
\mbE|V_{s,t}^{\veps}(\xi_1)-V_{s,t}^{\veps}(\xi_2)|^2
\leq\mbE|\xi_1-\xi_2|^2{\rm exp}\left(-\frac{\lambda_1}{2}(t-s)\right).
\end{equation}

Moreover, for any $0<\var\leq\var_{0}$ there exists a unique solution
$V_t^\veps,t\in\R$, to \eqref{sfmain} such that
\begin{equation*}\label{pEBou}
\sup_{t\in\R}\mbE|V_t^{\veps}|^{2p}<\infty,
\end{equation*}
where $1\leq p<\frac{\lambda_1}{2L_\sigma^2}+1$.
\end{proposition}
\begin{proof}
In view of It\^o's formula, we have
\begin{align}
&
\mbE|V_{s,t}^{\veps}(\xi_1)-V_{s,t}^{\veps}(\xi_2)|^2\nonumber\\
&
=\mbE|\xi_1-\xi_2|^2
+\mbE\int_s^t\Big(2\langle F_\veps(\tau,V_{s,\tau}^{\veps}(\xi_1))-F_\veps(\tau,V_{s,\tau}^{\veps}(\xi_2)),
V_{s,\tau}^{\veps}(\xi_1)-V_{s,\tau}^{\veps}(\xi_2)\rangle \label{BouSeq00}\\\nonumber
&\qquad
+|G_\veps(\tau,V_{s,\tau}^{\veps}(\xi_1))-G_\veps(\tau,V_{s,\tau}^{\veps}(\xi_2))|_{HS}^2\Big)\d\tau.
\end{align}
By {\bf(H$_x^3$)}, {\bf(H$_x^6$)} and {\bf(H$_y^6$)},
for all $t\in\R$ and $v_1:=(x_1,y_1)^T,v_2:=(x_2,y_2)^T\in\R^{d_1+d_2}$ we have
\begin{align}
&
2\langle F_{\var}(t,v_1)-F_{\var}(t,v_2),v_1-v_2\rangle
+|G_\veps(t,v_1)-G_\veps(t,v_2)|_{HS}^2\nonumber\\\nonumber
&
=2\langle f_\veps(t,x_1,y_1)-f_\veps(t,x_2,y_2),x_1-x_2\rangle
+|\sigma_\veps(t,x_1)-\sigma_\veps(t,x_2)|^2_{HS}\\
&\quad
+\veps^{-2\alpha}2\langle B(y_1)-B(y_2),y_1-y_2\rangle
+\veps^{-2\alpha}|g(y_1)-g(y_2)|_{HS}^2
\label{BouSeq01}\\\nonumber
&\quad
+\veps^{-\beta}2\langle b(x_1,y_1)-b(x_2,y_2),y_1-y_2\rangle\\\nonumber
&
\leq-\lambda_{1}|x_{1}-x_{2}|^{2}+\lambda_2|y_1-y_2|^2
-\var^{-2\alpha}\eta|y_{1}-y_{2}|^{2}\\\nonumber
&\quad
+\var^{-\beta}2L_b(|x_{1}-x_{2}|+|y_{1}-y_{2}|)|y_{1}-y_{2}|.
\end{align}

If $\beta=\alpha$ and $\lambda_1>L_b^2/\eta$, then it follows from \eqref{BouSeq01} and Young's inequality that there exists a constant $\veps_0>0$ such that for all $0<\veps\leq\veps_0$
\begin{align*}
&
2\langle F_{\var}(t,v_1)-F_{\var}(t,v_2),v_1-v_2\rangle
+|G_\veps(t,v_1)-G_\veps(t,v_2)|_{HS}^2\nonumber\\
&
\leq -\left(\frac{\lambda_1}{2}-\frac{L_b^2}{2\eta}\right)|x_1-x_2|^2\label{BouSeq02}\\\nonumber
&\quad
-\veps^{-2\alpha}
\left(\eta\left(1-\frac{2L_b^2}{\lambda_1\eta+L_b^2}\right)-\veps^{2\alpha}\lambda_2
-\veps^{2\alpha-\beta}2L_b\right)|y_1-y_2|^2\\\nonumber
&
\leq -\left(\frac{\lambda_1}{2}-\frac{L_b^2}{2\eta}\right)|v_1-v_2|^2,
\end{align*}
which by \eqref{BouSeq00} and Gronwall's inequality implies that
\begin{equation*}
\mbE|V_{s,t}^{\veps}(\xi_1)-V_{s,t}^{\veps}(\xi_2)|^2
\leq \mbE|\xi_1-\xi_2|^2{\rm exp}\left(-\left(\frac{\lambda_1}{2}-\frac{L_b^2}{2\eta}\right)(t-s)\right).
\end{equation*}

If $\beta<\alpha$ then based on \eqref{BouSeq01} and Young's inequality there exists $\veps_0>0$ such that for all $0<\veps\leq\veps_0$
\begin{align*}
&
2\langle F_{\var}(t,v_1)-F_{\var}(t,v_2),v_1-v_2\rangle
+|G_\veps(t,v_1)-G_\veps(t,v_2)|_{HS}^2\nonumber\\
&
\leq-\frac{\lambda_1}{2}|x_1-x_2|^2
-\left(\veps^{-2\alpha}\eta-\veps^{-2\beta}\frac{2L_b^2}{\lambda_1}-\veps^{-\beta}2L_b-\lambda_2\right)|y_1-y_2|^2
\leq -\frac{\lambda_1}{2}|v_1-v_2|^2,
\end{align*}
which by \eqref{BouSeq00} and Gronwall's inequality implies that
\begin{equation*}
\mbE|V_{s,t}^{\veps,1}(\xi_1)-V_{s,t}^{\veps,2}(\xi_2)|^2
\leq \mbE|\xi_1-\xi_2|^2{\rm exp}\left(-\frac{\lambda_1}{2}(t-s)\right).
\end{equation*}

For any $n\in\N$, let $V_{-n,t}^\veps(0):=(X_{-n,t}^\veps(0),Y_{-n,t}^\veps(0)),t\geq-n$,
be the solution to \eqref{sfmain}. Thanks to \eqref{ASSS1}, \eqref{ASSS2} and
the classical pullback absorbing method (see e.g. \cite[Theorem 3.6]{CL2023} for more details),
there exists a $\mcL^2(\Omega,\mbP;\R^{d_1+d_2})$-bounded solution
$V_t^\veps:=(X_t^\veps,Y_t^\veps), t\in\R$, of \eqref{sfmain},
which is the limit of $V_t^\veps(-n,0)$ in
$\mcL^{2}(\Omega,\mbP;\R^{d_1+d_2})$ as $n\rightarrow\infty$.
Moreover, by \eqref{BouX2p} and \eqref{BouY2p}, one sees that
$\sup_{t\in\R}\mbE|V_t^\veps|^{2p}\leq \infty$,
where $1\leq p<\frac{\lambda_1}{2L_\sigma^2}+1$.
\end{proof}

\begin{remark}\rm
Note that Proposition \ref{BouSth} implies that the bounded solution of \eqref{sfmain}
is globally asymptotically stable in the square-mean sense.
\end{remark}

Now we can show that the bounded solution inherits the quasi-periodic property of the coefficients,
and establish the second Bogolyubov theorem.
\begin{proof}[Proof of Theorem \ref{SAPth}]

Similar to the proof of Theorem 3.14 in \cite{CL2023}, we show that if $f$ and $g$ are
quasi-periodic then for any $0<\veps\leq\veps_0$, $V_t^\veps,t\in\R$,
is quasi-periodic in distribution.

Note that it can be verified that for all $x_1,x_2\in\R^{d_1}$
\[
2\langle \bar{f}(x_1)-\bar{f}(x_2),x_1-x_2\rangle+|\bar{\sigma}(x_1)-\bar{\sigma}(x_2)|
\leq-\lambda_1|x_1-x_2|^2,
\]
which implies that the global attractor $\bar{\msA}$ of $\bar{P}^*$ is a singleton set
and
\begin{equation}\label{SAPtheq02}
 \bar{\msA}=\{\msL(\bar{X}_t)\},
\end{equation}
where $\bar{X}_t,t\in\R$ is the stationary solution to \eqref{avemain}.

For any $\widetilde{\mbF}_\veps\in\mcH(\mbF_\veps)$ define
\[
\mfB_{\widetilde{\mbF}_\veps}:=
\left\{\msL(V^{\widetilde{\mbF}_\veps}_t):t\in\R\right\},
\]
where $V^{\widetilde{\mbF}_\veps}_t$ is the bounded solution to \eqref{sfmain} with $\widetilde{\mbF}_\veps$ replacing $\mbF_\veps$. Note that
\begin{equation*}
P_\veps^*(t,\sigma_{-t}\mbF_\veps,\mfB_{\sigma_{-t}\mbF_\veps})
=\mfB_{\mbF_\veps}
\end{equation*}
and
$\mfB_{\sigma_{-t}\mbF_\veps}\subset B$,
where is $B$ defined by \eqref{PGAeq0}. Then we have
\begin{align*}
\mfB_{\mbF_\veps}
\subset\cap_{\tau\geq0}\overline{\cup_{t\geq\tau}
P_\veps^*(t,\sigma_{-t}\mbF,\mfB_{\sigma_{-t}\mbF_\veps})}
\subset\cap_{\tau\geq0}\overline{\cup_{t\geq\tau}
P_\veps^*(t,\sigma_{-t}\mbF,B)}=\msA_{\mbF_\veps},
\end{align*}
which by \eqref{SAPtheq02} and Theorem \ref{GAPth} implies that
\[
\lim_{\veps\rightarrow0}\sup_{t\in\R}d_{BL}
\left(\msL\left(X^\veps_t\right),
\msL\left(\bar{X}_t\right)\right)=0.
\]
\end{proof}

\appendix

\section{}\label{EUSFSDE}

In this section, we will show the existence and uniqueness of solutions to \eqref{sfmain}.

\begin{lemma}\label{EUSFSDEth}
Suppose that {\bf (H$_y^1$)}, {\bf (H$_y^3$)}
and {\bf (H$_x^1$)}--{\bf (H$_x^3$)} hold. Then for any $(x,y)\in\R^{d_1+d_2}$ there exists
a unique solution $(X_t^\veps(x),Y_t^\veps(y))$ to \eqref{sfmain} provided
$0<\veps\leq \sqrt[2\alpha]{\eta'/K_5}\wedge 1$.
\end{lemma}

\begin{proof}
Let $v:=(x,y)^T$,
$$F_{\varepsilon}(t,v):=\left(f_{\varepsilon}(t,x,y),\veps^{-2\alpha}
B(x,y)+\veps^{-\beta}b(x,y)\right)^{T},\quad
G_\veps(t,v):=\left(\sigma_\veps(t,x),\veps^{-\alpha} g(x,y)\right)^{T}$$
and $W:=\left(W^{1},W^{2}\right)^{T}$. Then
equation \eqref{sfmain} can be written as
\begin{equation*}\label{APeqsys}
\d V_{t}^{\varepsilon}=F_{\varepsilon}(t,V_{t}^{\varepsilon})\d t
+G_\veps(t,V_{t}^{\varepsilon})\d W_t.
\end{equation*}
Note that by {\bf (H$_x^1$)} and {\bf (H$_y^1$)}, we have
for any $t\in\R$ and $v:=(x,y)\in\R^{d_1+d_2}$
\begin{align*}
&
2\langle F_{\varepsilon}(t,v),v\rangle+|G_\veps(t,v)|^2_{HS}\\
&
=2\langle f_\veps(t,x,y),x\rangle
+|\sigma_\veps(t,x)|^2_{HS}
+\veps^{-2\alpha}\left(2\langle B(x,y),y\rangle+|g(x,y)|_{HS}^2\right)
+\veps^{-\beta}2\langle b(x,y),y\rangle\\
&
\leq K_4|x|^2+(-\veps^{-2\alpha}\eta+\veps^{-\beta}\widetilde{\eta})|y|^2
+(K_5-\veps^{-2\alpha}\eta')|y|^{\theta}+K_4+(\veps^{-2\alpha}+\veps^{-\beta})K_1.
\end{align*}
If $\theta=2$ or $K_5=0$,
then for all $0<\veps\leq1$
\begin{align*}
2\langle F_{\varepsilon}(t,v),v\rangle+|G_\veps(t,v)|^2_{HS}
\leq C_{K_4,\veps}|v|^2+C_{\veps,K_1,K_4},~\forall (t,v)\in\R^{1+d_1+d_2}.
\end{align*}
Otherwise, for all $0<\veps\leq \sqrt[2\alpha]{\eta'/K_5}$
\begin{align*}
2\langle F_{\varepsilon}(t,v),v\rangle+|G_\veps(t,v)|^2_{HS}
\leq K_4|v|^2+C_{\veps,K_1,K_4},~\forall (t,v)\in\R^{1+d_1+d_2}.
\end{align*}

By {\bf (H$_x^2$)}, {\bf (H$_x^3$)} and {\bf (H$_y^3$)}, one sees that
\begin{align*}
&
2\langle F_{\var}(t,v_1)-F_{\var}(t,v_2),v_1-v_2)\rangle
+|G_\veps(t,v_1)-G_\veps(t,v_2)|_{HS}^2\\
&
=2\langle f_\veps(t,x_1,y_1)-f_\veps(t,x_2,y_2),x_1-x_2\rangle
+|\sigma_\veps(t,x_1)-\sigma_\veps(t,x_2)|^2_{HS}\\
&\quad
+\veps^{-2\alpha}2\langle B(x_1,y_1)-B(x_2,y_2),y_1-y_2\rangle
+\veps^{-2\alpha}2|g(x_1,y_1)-g(x_2,y_2)|_{HS}^2\\
&\quad
+\veps^{-\beta}2\langle b(x_1,y_1)-b(x_2,y_2),y_1-y_2\rangle\\
&
\leq 2|f_\veps(t,x_1,y_1)-f_\veps(t,x_2,y_2)| |x_1-x_2|
+L_\sigma^2|x_1-x_2|^2\\
&\quad
+\veps^{-2\alpha}2\langle B(x_1,y_1)-B(x_1,y_2),y_1-y_2\rangle
+\veps^{-2\alpha}2\left(L_g|x_1-x_2|+L_g|y_1-y_2|\right)^2\\
&\quad
+\veps^{-2\alpha}2|B(x_1,y_2)-B(x_2,y_2)| |y_1-y_2|
+\veps^{-\beta}2|b(x_1,y_1)-b(x_2,y_2)| |y_1-y_2|\\
&
\leq 2K_7\left(1+|x_1|^{\theta_1}+|x_2|^{\theta_1}+|y_1|^{\theta_2}+|y_2|^{\theta_2}\right)
  \left(|x_1-x_2|+|y_1-y_2|\right)|x_1-x_2|+L_\sigma^2|x_1-x_2|^2\\
&\quad
-\veps^{-2\alpha}\eta|y_1-y_2|^2-\veps^{-2\alpha}\eta'|y_1-y_2|^{\theta}
+\veps^{-2\alpha}4L_g^2\left(|x_1-x_2|^2+|y_1-y_2|^2\right)\\
&\quad
+2\veps^{-2\alpha}K_3\left(1+|y_2|^{\kappa_2}\right)|x_1-x_2||y_1-y_2|\\
&\quad
+2\veps^{-\beta}K_4\left(1+|x_1|^{\kappa_1}+|x_2|^{\kappa_1}+|y_1|^{\kappa_2}+|y_2|^{\kappa_2}\right)
\left(|x_1-x_2|+|y_1-y_2|\right)|y_1-y_2|\\
&
\leq C_{K_4,K_7,L_\sigma,L_g,\veps}\left(1+|x_1|^{2\theta_1\vee2\kappa_1}+|x_2|^{2\theta_1\vee2\kappa_1}
+|y_1|^{2\theta_2\vee2\kappa_2}+|y_2|^{2\theta_2\vee2\kappa_2}\right)|x_1-x_2|^2\\
&\quad
+C_{K_4,K_7,L_\sigma,L_g,\veps,\widetilde{\eta}}\left(1+|x_1|^{2\theta_2\vee2\kappa_2}
+|x_2|^{2\theta_2\vee2\kappa_2}
+|y_1|^{2\theta_2\vee2\kappa_2}+|y_2|^{2\theta_2\vee2\kappa_2}\right)|y_1-y_2|^2
\end{align*}
for any $t\in\R$ and $v_1:=(x_1,y_1),v_2:=(x_2,y_2)\in\R^{d_1+d_2}$.
Therefore, it follows from \cite[Theorem 3.1.1]{LR2015} that for any $(x,y)\in\R^{d_1+d_2}$
there exists a unique solution $(X_t^\veps(x),Y_t^\veps(y))$ to \eqref{sfmain}
for all $0<\veps\leq \sqrt[2\alpha]{\eta'/K_5}\wedge1$.
\end{proof}

\section*{Acknowledgements}
The first author would like to acknowledge the
warm hospitality of Bielefeld University.
The second author was supported by NSFC Grants
11871132, 11925102, and
Dalian High-level Talent Innovation Project (Grant 2020RD09).
The third author was supported by the Deutsche Forschungsgemeinschaft
(DFG, German Research Foundation) - SFB 1283/2 2021 - 317210226.

\end{document}